\documentclass[11pt, reqno]{amsart}
\usepackage{amssymb, amsthm, amsmath, amsfonts,mathtools}
\usepackage{array, epsfig}
\usepackage{bbm}
\usepackage{MnSymbol}

\usepackage{hyperref}
\usepackage[numbers,square]{natbib}
\usepackage{tikz-cd}%commutatve diagram

\allowdisplaybreaks

\numberwithin{equation}{section}

  \evensidemargin
\oddsidemargin
\newcommand{\ii}{{\rm{i}}}

\newcommand{\bT}{\mathbb{T}}
\newcommand{\bH}{\mathbb{H}}
\newcommand{\bD}{\mathbb{D}}

\newcommand{\cD}{\mathcal{D}}

\DeclareMathOperator*{\cotanh}{cotanh}

\DeclareMathOperator{\arccosh}{arccosh}

\def\dint{\textup{d}}

\newcommand{\E}{\mathbb E}

\newcommand{\R}{\mathbb{R}}
\newcommand{\N}{\mathbb{N}}

\newcommand{\C}{\mathbb{C}}
\newcommand{\Z}{\mathbb{Z}}

\renewcommand{\Re}{\operatorname{Re}}
\renewcommand{\Im}{\operatorname{Im}}

\newcommand{\eps}{\varepsilon}

\newcommand{\toweak}{\overset{w}{\underset{n\to\infty}\longrightarrow}}
\newcommand{\toweakd}{\overset{w}{\underset{d\to\infty}\longrightarrow}}

\newcommand{\ton}{\overset{}{\underset{n\to\infty}\longrightarrow}}

\newcommand{\ind}{\mathbbm{1}}

\newcommand{\dd}{{\rm d}}
\newcommand{\eee}{{\rm e}}

\theoremstyle{plain}
\newtheorem{theorem}{Theorem}[section]
\newtheorem{lemma}[theorem]{Lemma}
\newtheorem{corollary}[theorem]{Corollary}
\newtheorem{proposition}[theorem]{Proposition}
\newtheorem{conjecture}[theorem]{Conjecture}

\theoremstyle{definition}

\newtheorem{example}[theorem]{Example}

\theoremstyle{remark}
\newtheorem{remark}[theorem]{Remark}

\begin{document}

\author{Zakhar Kabluchko}
\address{Zakhar Kabluchko: 
Institut f\"ur Mathematische Stochastik,
Universit\"at M\"unster,
Orl\'eans-Ring 10,
48149 M\"unster, 
Germany}
\email{zakhar.kabluchko@uni-muenster.de}

%%%%%%%%%%%%%%%%%%%%%%%%%%%%%%%%%%%%%%%%%%%%%%%%%%%%%%%%%%%%%%%%%%%%%%%%%%%%%%%%%%%%%%%%%%%%%%%%%%
%%%% All numerical simulations and figures are in "Trigonometric Laguerre Polynomials.nb"
%%%%%%%%%%%%%%%%%%%%%%%%%%%%%%%%%%%%%%%%%%%%%%%%%%%%%%%%%%%%%%%%%%%%%%%%%%%%%%%%%%%%%%%%%%%%%%%%%%

\title[Repeated differentiation and free unitary Poisson process]{Repeated differentiation and free unitary Poisson process}



\begin{abstract}
We investigate the hydrodynamic behavior of zeroes of trigonometric polynomials under repeated differentiation.
We show that if the zeroes of a real-rooted, degree $d$  trigonometric polynomial are distributed according to some probability measure $\nu$ in the large $d$ limit, then the zeroes of its $[2td]$-th derivative, where $t>0$ is fixed, are distributed according to the free multiplicative convolution of $\nu$ and the free unitary Poisson distribution with parameter $t$. In the simplest special case, our result states that the zeroes of the $[2td]$-th derivative of the trigonometric polynomial $(\sin \frac \theta 2)^{2d}$ (which can be thought of as the trigonometric analogue of the Laguerre polynomials) are distributed according to the free unitary Poisson distribution with parameter $t$, in the large $d$ limit. The latter distribution is defined in terms of the function $\zeta=\zeta_t(\theta)$ which solves the implicit equation $\zeta - t \tan \zeta = \theta$ and satisfies
$$
\zeta_t(\theta)= \theta + t \tan (\theta + t \tan (\theta + t \tan (\theta +\ldots))),
\qquad
\Im \theta >0,
\;\;
t>0.
$$
\end{abstract}

\maketitle
%\tableofcontents

\section{Main results}

\subsection{Introduction}
Let $Q\in \C[z]$ be a polynomial with complex coefficients and degree $\deg Q\geq 1$. The \textit{empirical distribution of roots} of $Q$ is the probability measure on the complex plane $\C$ given by
$$
\mu\lsem Q\rsem := \frac 1{\deg Q} \sum_{z\in \C:\, Q(z) = 0} m_Q(z) \delta_{z},
$$
where $m_Q(z)$ denotes the multiplicity of the root at $z$, and $\delta_z$ is the unit delta-measure at $z$.  Let $(P_n(z))_{n\in \N}$ be a sequence of polynomials such that $\deg P_n = n$ and suppose that $\mu\lsem P_n\rsem$ converges weakly to some probability measure $\mu$ on $\C$, as $n\to\infty$. We are interested in the asymptotic distribution of complex roots of the repeated derivatives $P_n',P_n'',\ldots$, in the large $n$ limit. For concreteness, let us consider random polynomials of the form
$$
P_n(z) = (z-\xi_1)\ldots (z-\xi_n),
\qquad n\in \N,
$$
where $\xi_1,\xi_2,\ldots$ are independent complex-valued random variables with distribution $\mu$. By the law of large numbers, for almost every realization of the sequence $\xi_1,\xi_2,\ldots$ we have $\mu\lsem P_n\rsem\to \mu$ weakly on $\C$.

Regarding the first derivative, a result of~\cite{kabluchko15}, proving a conjecture of Pemantle and Rivin~\cite{pemantle_rivin}, states that $\mu\lsem P_n'\rsem \to \mu$ in probability on $\mathcal M_\C$, the space of probability measures on $\C$ endowed with the L\'evy-Prokhorov metric (which generates the topology of weak convergence). For more results in this direction, we refer to~\cite{subramanian,subramanian_phd,orourke_matrices,orourke_williams_local_pairing,orourke_williams_pairing,Tulasi16,tulasi_phd,byun,kabluchko_seidel}.  In fact, \citet{byun} showed that the same conclusion continues to hold if $P_n'$ is replaced by the $k$-th derivative $P_n^{(k)}$, where $k\in \N$ is fixed. Recently, \citet{angst_malicet_poly} (for $k=1$) and~\citet{michelen_vu_a_s_iterated} (for $k\in \N$) showed that convergence in probability can be replaced by a.s.\ convergence.

These probabilistic results should be compared to the deterministic results of Totik~\cite{totik_distr_crit_points}. He proved that if $P_n$ is a deterministic sequence of polynomials such that $\mu\lsem P_n\rsem\to \mu$ weakly on $\C$ for some measure $\mu$ supported on a compact set $K\subseteq \C$ whose complement is connected, then $\mu\lsem P_n'\rsem\to \mu$. The latter condition cannot be removed, as the example $P_n(z) = z^n -1$ shows; see~\cite{totik_distr_crit_points} for more on such exceptional cases and~\cite{totik_critical_hung} for a recent review of deterministic results on critical points.

Knowing that under a single differentiation the empirical distribution of zeroes changes only infinitesimally, it is natural to ask about the distribution of zeroes of the $k$-th derivative $P_n^{(k)}$, where  $k=k(n)\to\infty$ as $n\to\infty$. In fact, it is known that under one differentiation, the zeroes move by a quantity of order $1/n$. There exist several results~\cite{hanin_poly,hanin_gauss,hanin_riemann,kabluchko_seidel,orourke_williams_local_pairing,orourke_williams_pairing} stating the existence of a local ``pairing'' between the roots of $P_n$ and the roots of $P_n'$. More precisely, for every zero of $P_n$, say $\xi_1$, with high probability there exists a zero of $P_n'$ at $\xi_1 - 1/(nG(\xi_1)) + o(1/n)$, where
$$
G(x) = \int_\C \frac{\mu(\dd y)}{x-y}
$$
is the Cauchy-Stieltjes transform of $\mu$. To derive this formula heuristically, let us write the equation for the zeroes of $P_n'(z)/P_n(z)$ in the form
$$
\frac 1 {z-\xi_1} = -\sum_{j=2}^n \frac{1}{z-\xi_j}  \approx -n G(z).
$$
If we are looking for a solution $z$ close to $\xi_1$, then the right-hand side can be approximated by $-nG(\xi_1)$, and we obtain $z\approx \xi_1 - 1/(nG(\xi_1))$.
Following~\cite{steinerberger_real,orourke_steinerberger_nonlocal,bogvad_etal} it is therefore natural to conjecture the emergence of a non-trivial hydrodynamic behavior of the roots in the regime when the number of differentiations satisfies  $k(n)/n \to t$ as $n\to\infty$, where the parameter $t\in [0,1)$ can be interpreted as the ``time'' of the system.
\begin{conjecture}[cf.~\cite{steinerberger_real,orourke_steinerberger_nonlocal,bogvad_etal}]\label{conj:main}
Fix some $0 \leq t<1$ and let $k=k(n)$ be such that $k/n \to t$ as $n\to\infty$. Suppose, for simplicity, that $\mu$ is compactly supported.  Then, the random probability measure $\mu\lsem P_n^{(k)}\rsem$  converges\footnote{In the sense of stochastic (or even a.s.)\ convergence of random $\mathcal M_\C$-valued elements.} to a certain deterministic probability measure $\mu_t$,  as $n\to\infty$.
\end{conjecture}
In particular, in the regime when $k=o(n)$, it is conjectured that $\mu\lsem P_n^{(k)}\rsem$ converges to $\mu_0=\mu$, the initial distribution of roots (see the papers of~\citet{michelen_vu_growing_number} and~\citet{galligo_najnudel_vu} for  results in this direction), while for $t\neq 0$ the measures $\mu_t$ and $\mu$ should be in general different.
%To  be more precise, the convergence can be understood as stochastic (or even a.s.)\ convergence of random $\mathcal M_\C$-valued elements.
It follows from the Gauss-Lucas theorem that the convex hull of the support of $\mu_t$ is nonincreasing in $t$. Moreover, by~\cite[Theorem~3]{totik_critical_hung} (which is based on the work of Ravichandran~\cite{ravichandran}), the diameter of the support of $\mu_t$ converges to $0$ as $t\uparrow 1$ with an explicit bound on the rate of decrease. By a theorem of Malamud-Pereira, see~\cite[\S~3]{totik_critical_hung}, the function $t\mapsto \int_{\C} \varphi(x) \mu_t (\dint x)$ is non-increasing for every convex function $\varphi: \C\to \R$.
Although Conjecture~\ref{conj:main} remains unproven for measures on the complex plane, several approaches have been suggested to explicitly describe the measures $\mu_t$ under additional assumptions on the initial distribution $\mu$. Let us briefly review the existing results.

\medskip
\noindent
\textit{Partial differential equations}.
If at time $t=0$ the zeroes have a probability density $p_0(x) = \dd \mu/ \dd x$ w.r.t.\ the two-dimensional Lebesgue measure, then one may conjecture that the density $p_t(x)= (1-t) \dd \mu_t/\dd x$ of $(1-t)\mu_t$  satisfies the following non-local PDE:
\begin{equation}\label{eq:PDE_general}
\frac{\partial p_t(x)}{\partial t} = -\nabla \left(\vec v_t(x) p_t(x)\right),
\quad
\vec v_t(x) = -\left(\Re \frac{1}{G_t(x)}, \Im \frac{1}{G_t(x)}\right),
\quad
G_t(x) = \int_\C \frac{p_t(y) \dd y}{x-y},
\end{equation}
where $0\leq t <1$.
Note that this is the convection equation  in which $\vec v_t(x)$ is the velocity of the flow at position $x$.
There seem to be no results on this general PDE, but in the case when the initial distribution $\mu$ is rotationally invariant, the PDE  simplifies considerably.  The PDE corresponding to this isotropic special case has been first non-rigorously derived by O'Rourke and Steinerberger~\cite{orourke_steinerberger_nonlocal} and can be solved explicitly, as has been shown in~\cite{hoskins_kabluchko}; see also~\cite{feng_yao} for a particular case when $\mu$ is the uniform distribution on the unit circle.

Assume now that all zeroes of $P_n$ are real and have some probability density $u_0(x)$ w.r.t.\ the one-dimensional Lebesgue measure. By Rolle's theorem, all zeroes of $P_n',P_n'',\ldots$ are real, too. Let $u_t(x)$ be the density of zeroes at time $t\in [0,1)$, with the normalization $\int_{\R} u_t(x)\dint x = 1-t$.  The following non-local PDE for $u_t(x)$ has been non-rigorously derived by Steinerberger~\cite{steinerberger_real}:
\begin{equation}\label{eq:PDE_steinerberger}
\frac{\partial u_t(x)}{\partial t}  + \frac 1 \pi \frac{\partial}{\partial x} \arctan \frac{(H u_t)(x)}{u_t(x)}=0,
\end{equation}
where $(H u_t) (x) = \frac 1 \pi\, \text{P.V.} \int_{\R} \frac{u_t(y)}{x-y}\dd y$ is the Hilbert transform of $u_t$ and the integral is understood in the sense of principal value.
Local and global regularity results for this PDE with periodic initial condition have been obtained by Granero-Belinch{\'o}n~\cite{granero_belinchon}, Kiselev and Tan~\cite{kiselev_tan} and Alazard et al.~\cite{alazard_etal}. As Kiselev and Tan~\cite{kiselev_tan} rigorously showed, the periodic version of Steinerberger's PDE~\eqref{eq:PDE_steinerberger} describes the hydrodynamic limit of the roots of \textit{trigonometric} polynomials, which is the setting we shall concentrate on in the present paper. A key ingredient used in~\cite{steinerberger_real} and~\cite{kiselev_tan} to derive the PDE is the fact that under repeated differentiation, the real roots crytallize~\cite{pemantle_subramanian_crystallization}, i.e.\ they approach locally an arithmetic progression.

\medskip
\noindent
\textit{Free probability}. In the case when all roots are real, there is an intriguing  connection between repeated differentiation and free probability~\cite{steinerberger_free,steinerberger_conservation,hoskins_steinerberger,shlyakhtenko_tao,hoskins_kabluchko,arizmendi_garza_vargas_perales}.  Steinerberger~\cite{steinerberger_free} non-rigorously argued that $\mu_t$, up to rescaling, coincides with the free  additive self-convolution power $\mu^{\boxplus 1/(1-t)}$. In fact, he showed that both objects satisfy the same PDE. Similar calculations appeared earlier in the paper of Shlyakhtenko and Tao~\cite{shlyakhtenko_tao}.
Steinerberger's claim has been rigorously proven in~\cite{hoskins_kabluchko} without using PDE's.
A more natural proof, given by Arizmendi, Garza-Vargas and Perales~\cite[Theorem~3.7]{arizmendi_garza_vargas_perales}, uses \textit{finite} free probability, a subject developed recently in the papers of Marcus~\cite{marcus},  Marcus, Spielman, Srivastava~\cite{marcus_spielman_srivastava} and Gorin and Marcus~\cite{gorin_marcus}. In the approach of~\cite{arizmendi_garza_vargas_perales}, the $k$-th derivative of $P_n$ is essentially identified with the finite free multiplicative convolution of $P_n$ and the  polynomial $x^k(1-x)^{d-k}$. Since the empirical distribution of the latter converges to $t \delta_0 + (1-t) \delta_1$ and since finite free probability converges to the usual free probability in the sense of~\cite[Theorem~1.4]{arizmendi_garza_vargas_perales}, $\mu_t$ can essentially be identified with the free multiplicative convolution of $\mu$ and the probability measure $t \delta_0 + (1-t) \delta_1$. By Equation~(14.13) in~\cite{nica_speicher_book}, the latter convolution can be expressed through $\mu^{\boxplus 1/(1-t)}$.

%The conservation laws derived in~\cite{steinerberger_conservation} are, in fact, the free additive cumulants of the corresponding distribution.

\medskip
\noindent
\textit{Saddle point method}. Motivated by the Rodrigues formula for Legendre polynomials, B{\o}gvad, H{\"{a}}gg and Shapiro~\cite{bogvad_etal} considered polynomials of the form
$$
R_{n, [nt], Q}(z) := \left(\frac{\dint }{\dint z} \right)^{[nt]} (Q(z))^n,
$$
where $Q(z)$ is a fixed polynomial and $0<t<\deg Q$. The main results of~\cite{bogvad_etal} describe the asymptotic distribution of the zeroes of these so-called Rodrigues' descendants, as $n\to\infty$, and essentially prove Conjecture~\ref{conj:main} in the case when $\mu$ is a finite convex combination of delta-measures. The approach of~\cite{bogvad_etal} is to represent $R_{n, [nt], Q}(z)$ as a contour integral using the Cauchy formula, and then to apply the saddle-point method to evaluate the large $n$ limit of $\frac 1n \log |R_{n, [nt], Q}(z)|$, thus identifying  the logarithmic potential of the limit distribution of zeroes of $R_{n, [nt], Q}(z)$.

\subsection{Main result}\label{subsec:main_res}
The aim of the present work is to study the hydrodynamic limit of roots of \textit{trigonometric} polynomials under repeated differentiation and to relate it to the free multiplicative Poisson process on the unit circle. The same problem has been studied from the point of view of PDE's in~\cite{granero_belinchon,kiselev_tan,alazard_etal}. Our approach is different and, in some sense, provides an explicit solution to the PDE~\eqref{eq:PDE_steinerberger} (with periodic initial condition) studied in these papers.
%In the present paper we shall be interested in the behavior of roots of \textit{trigonometric} polynomials under repeated differentiation.
For every $d\in \N$ we consider a degree $d$ trigonometric polynomial of the form
\begin{equation}\label{eq:trig_poly_general_form}
T_{2d}(\theta) = c_{2d} + \sum_{j=1}^d (a_{j;2d} \cos (j\theta) + b_{j;2d} \sin (j\theta)),\qquad \theta\in \R,
\end{equation}
with real coefficients $c_{2d}$, $a_{1;2d},\ldots,a_{d;2d}$, $b_{1;2d},\ldots, b_{d;2d}$. Clearly, $T_{2d}$ is a $2\pi$-periodic function, that is $T_{2d}(\theta + 2\pi) = T_{2d}(\theta)$.
The maximal number of zeroes of $T_{2d}$ (counted with multiplicities, if not otherwise stated) in the interval $[-\pi,\pi)$ is $2d$; see Section~\ref{subsec:alg_and_trig_poly} where we shall explain in more detail this and other well-known facts about trigonometric polynomials.  We are interested in the setting when the trigonometric polynomial $T_{2d}$ is \textit{real-rooted}, that is when it has $2d$ zeroes in $[-\pi,\pi)$. Denoting these zeroes by $\theta_{1;2d},\ldots, \theta_{2d;2d}$, we can write
$$
T_{2d}(\theta) = \text{const} \cdot \prod_{j=1}^{2d} \sin \frac{\theta - \theta_{j;2d}}{2}.
$$
The \textit{empirical distribution of zeroes} associated with $T_{2d}$ is defined by
$$
\nu\lsem T_{2d}\rsem := \frac {1}{2d} \sum_{j=1}^{2d} \delta_{\eee^{\ii \theta_{j;2d}}}.
$$
Note that  we mapped the zeroes to the unit circle $\bT:= \{z\in \C: |z|=1\}$ and that $\nu \lsem T_{2d}\rsem$ is a probability measure on $\bT$. We shall be interested in the asymptotic distribution of zeroes of the $k=k(d)$-th derivative of $T_{2d}$, where $(k(d))_{d\in \N}$ is a sequence of natural numbers such that
\begin{equation}\label{eq:k(n)}
t:= \lim_{d\to\infty} \frac{k(d)}{2d} \in [0,\infty).
\end{equation}
For example, we may take $k(d) = [2t d]$.
Note that the property of real-rootedness is preserved under repeated differentiation by Rolle's theorem. Our main result reads as follows.

\begin{theorem}\label{theo:main}
Let $(T_{2d}(\theta))_{d\in \N}$ be a sequence of real-rooted trigonometric polynomials such that $\nu \lsem T_{2d}\rsem$ converges weakly to some probability measure $\nu$ on the unit circle $\bT$, as $d\to\infty$. If $k(d)$ satisfies~\eqref{eq:k(n)},  then the empirical distribution of zeroes of the $k(d)$-th derivative of $T_{2d}(\theta)$ in $\theta$ converges weakly on $\bT$ to the free multiplicative convolution of $\nu$ and the free unitary Poisson distribution $\Pi_t$ with parameter $t$ (to be defined in Section~\ref{subsec:free_unitary_poi_def} below). That is,
$$
\nu \lsem T_{2d}^{(k(d))}\rsem \toweakd \nu \boxtimes \Pi_t.
$$
\end{theorem}
\begin{example}[The ``fundamental solution'']
The ``simplest'' real-rooted trigonometric polynomial is $h_{2d}(\theta) := (\sin \frac\theta 2)^{2d}$. It has a zero of multiplicity $2d$ at $\theta= 0$. %and (in some sense) plays the role of the delta function.
To see that $h_{2d}(\theta)$ can be written in the form~\eqref{eq:trig_poly_general_form}, use the formula $(\sin \frac\theta 2)^2 = (1-\cos \theta)/2 = (2 - \eee^{\ii \theta} - \eee^{-\ii\theta})/4$, and expand its $d$-th power.   In this special case, Theorem~\ref{theo:main} states that the empirical distribution of zeroes of the $[2td]$-th derivative of $h_{2d}(\theta)$ converges weakly to $\Pi_t$.
\end{example}

Our strategy for proving Theorem~\ref{theo:main} is as follows. We first associate to a real-rooted trigonometric polynomial $T_{2d}$ an algebraic polynomial $P_n$ of degree $n = 2d$ with roots on the unit circle. The repeated differentiation of $T_{2d}$ corresponds to the repeated action of a certain differential operator $\cD_{n}$ on $P_n$ (Section~\ref{subsec:alg_and_trig_poly}).
We then observe that $\mathcal{D}_n^k P_n(z) = P_n(z) \boxtimes_n L_{n,k}(z)$, where $\boxtimes_n$ is the finite free multiplicative convolution, and $L_{n,k}(z)$ are the circular Laguerre polynomials (Section~\ref{subsec:finite_free_conv_laguerre}). In a suitable sense, $\boxtimes_n$ converges to the free multiplicative convolution $\boxtimes$ as $n \to \infty$ (Proposition~\ref{prop:finite_free_mult_conv_to_free_mult}).
The limiting zero distribution of $L_{n,k}(z)$ is then computed via the saddle-point method in Section~\ref{sec:zeroes_laguerre}, where we also prove logarithmic asymptotics for $L_{n,k}(z)$ (Theorem~\ref{theo:limit_log_der}). The implicit function $\zeta = \zeta_t(\theta)$, defined by $\zeta - t \tan \zeta = \theta$, and appearing in these results, is studied in Section~\ref{sec:principal_branch}.
Finally, Section~\ref{sec:free_unitary_poi_properties} gathers some properties of the free unitary Poisson distribution.

In his blog, Tao~\cite{tao_blog1,tao_blog2} discusses the evolution of zeroes of a polynomial which undergoes a heat flow instead of repeated differentiation. The dynamics of zeroes is very similar to that in~\eqref{eq:PDE_general}, but with $1/G_t(x)$ replaced by $G_t(x)$ in the velocity term. Stochastic differential equations of the same type arise in connection with random matrices (Dyson's Brownian motion) and Bessel processes in Weyl chambers; see~\cite[Section~4.3]{anderson_etal_book}. Their hydrodynamic limits are related to free \textit{additive} L\'evy processes generated by the Wigner or Marchenko-Pastur distributions, see~\cite[Section~4.3]{anderson_etal_book} and~\cite{voit_werner2}, while the limits when the variance of the driving white noise goes to zero are related to zeroes of Hermite and Laguerre polynomials~\cite{andraus2012,dumitriu_edelman_large_beta,gorin_kleptsyn,voit_woerner1}.

\subsection{Free probability on the unit circle}\label{subsec:free_unitary_poi_def}
Let us recall the notions from free probability appearing in Theorem~\ref{theo:main}.
Free multiplicative convolution  of probability measures on the unit circle, introduced by Voiculescu in~\cite{voiculescu_symmetries} and~\cite{voiculescu_multiplication}, can be defined as follows; see~\cite[\S~3.6]{voiculescu_nica_dykema_book} and~\cite{voiculescu_multiplication}.
Let $\mathcal M_\bT$ be the set of probability measures on $\bT$. Given $\mu_1\in \mathcal M_\bT$ and $\mu_2\in \mathcal M_\bT$ it is possible to construct a $C^*$-probability space and two mutually free unitaries $u_1$ (with spectral distribution $\mu_1$) and $u_2$ (with spectral distribution $\mu_2$). Then, $\mu_1 \boxtimes \mu_2 \in \mathcal M_\bT$ is the spectral distribution of $u_1 u_2$.

For our purposes, the following characterization of $\boxtimes$ suffices.   The \textit{$\psi$-transform} of a probability measure $\mu\in \mathcal M_\bT$ is the function $\psi_\mu(z)$, defined on the open unit disk $\bD= \{z\in \C : |z|<1\}$ by
\begin{equation}\label{eq:psi_transf_def}
\psi_\mu(z) = \int_\bT \frac{uz}{1-uz} \mu(\dint u) = \sum_{\ell=1}^\infty z^\ell \int_{\bT} u^\ell \mu(\dint u),
\qquad z\in \bD.
\end{equation}
Let $\mathcal M_\bT^*$ be the set of all $\mu\in \mathcal M_\bT$ with $\psi_\mu'(0) = \int_\bT u \mu(\dint u) \neq 0$. If $\mu\in \mathcal M_{\bT}^*$, then the function $\psi_\mu$ has an inverse on some sufficiently small disk around the origin.
The \textit{$S$-transform} and the \textit{$\Sigma$-transform} of $\mu$ are defined by
\begin{equation}\label{eq:S_transf_def}
S_\mu(z) = \frac {1+z}{z} \psi^{-1}_\mu(z),
\qquad
\Sigma_\mu(z) = S_\mu \left(\frac {z}{1-z}\right)
=
\frac 1z \psi^{-1}_\mu\left(\frac{z}{1-z}\right),
\end{equation}
respectively, if $|z|$ is sufficiently small.
Then, it is known~\cite{voiculescu_multiplication} that
\begin{equation}\label{eq:S_transform_linearizes_conv}
S_{\mu \boxtimes \nu} (z) = S_{\mu}(z)  S_\nu (z),
\qquad
\Sigma_{\mu \boxtimes \nu} (z) = \Sigma_{\mu}(z)  \Sigma_\nu (z),
\qquad
\mu,\nu \in \mathcal M_{\bT}^*.
\end{equation}

Let us now recall the analogue of the Poisson limit theorem for $\boxtimes$. Fix some $t\geq 0$ and consider the probability measures $\mu_n := (1-\frac tn) \delta_1 + \frac {t} n \delta_{-1}$, $n\in \N$, on $\bT$, where $\delta_z$ is the Dirac unit mass at $z$.   The Poisson limit theorem~\cite[Lemma~6.4]{bercovici_voiculescu_levy_hincin} for $\boxtimes$ states that, as $n\to\infty$, the sequence of probability measures $\mu_n \boxtimes \ldots \boxtimes \mu_n$ ($n$ times) converges weakly to the \textit{free unitary Poisson distribution} $\Pi_t$ which is characterized by the following $S$- and $\Sigma$-transforms:
\begin{equation}\label{eq:S_transf_poi}
S_{\Pi_t}(z) = \exp\left\{\frac{t}{z+\frac 12}\right\},
\qquad
\Sigma_{\Pi_t}(z) = \exp\left\{2 t\, \frac{1-z}{1+z}\right\}.
\end{equation}
In Section~\ref{sec:free_unitary_poi_properties}, we will derive some properties of $\Pi_t$ including a characterization of its support and a series expansion for its absolutely continuous part.
In fact, defining $\mu_n := (1-\frac tn) \delta_1 + \frac t n \delta_{\zeta}$  with arbitrary $\zeta\in \bT$ leads to a more general two-parameter family of free Poisson distributions~\cite[Lemma~6.4]{bercovici_voiculescu_levy_hincin}, but we shall need only the case $\zeta= -1$ here. For further information on $\boxtimes$ including a characterization of $\boxtimes$-infinitely divisible distributions and more general limit theorems we refer to~\cite{bercovici_pata,bercovici_wang,biane,chistyakov_goetze_limitII,chistyakov_goetze_arithmetic,zhong_free_normal,zhong_free_normal,zhong_free_brownian,cebron}.

\section{Proof of Theorem~\ref{theo:main}: Reduction to the fundamental solution}\label{sec:proof_reduction_to_delta}
\subsection{Trigonometric and algebraic polynomials}\label{subsec:alg_and_trig_poly}
For $d\in \N$ let $\text{Trig}_{2d}[\theta]$ be the set of all trigonometric polynomials in the variable $\theta$ of the form
\begin{equation}\label{eq:def_T_n}
T_{2d}(\theta) = \sum_{\ell=-d}^d c_\ell \eee^{\ii \ell \theta},
\end{equation}
where $c_{-d},\ldots, c_d\in \C$ are complex coefficients. Clearly, $\text{Trig}_{2d}[\theta]$ is a $(2d+1)$-dimensional vector space over $\C$ with basis $1, \eee^{\pm \ii \theta},\ldots, \eee^{\pm \ii d \theta}$. Furthermore, for $n\in \N$ let  $\C_{n}[z]$ be the set of all algebraic polynomials, with complex coefficients, of degree at most $n$ in the variable $z$. Clearly, $\C_{n}[z]$ is an $(n+1)$-dimensional vector space over $\C$ with basis $1,z,\ldots, z^{n}$. We define an isomorphism  $\Psi: \text{Trig}_{2d}[\theta] \to \C_{2d}[z]$ between these $(2d+1)$-dimensional vector spaces by
$$
\Psi(\eee^{i \ell \theta}) := z^{\ell+d}, \qquad \ell\in \{-d,\ldots, d\},
$$
and then extending this map by linearity.
Note that for every trigonometric polynomial $T_{2d}\in \text{Trig}_{2d}[\theta]$ the corresponding algebraic polynomial $\Psi(T_{2d})$ satisfies
\begin{equation}\label{eq:alg_trig_corr}
T_{2d}(\theta) = \frac{\Psi(T_{2d}) (\eee^{\ii \theta})}{\eee^{\ii d\theta}}.
\end{equation}
For arbitrary $n\in \N$, consider the differential operator $\cD_{n}: \C_n[z] \to \C_n[z]$ defined by
\begin{equation}\label{eq:D_def}
\cD_{n} P (z) = \ii \left(z \frac{\dint P}{\dint z} - \frac n2  P(z) \right),
\qquad P\in \C_{n}[z].
\end{equation}
The action of  $\cD_{2d}$ on the space $\C_{2d}[z]$ of algebraic polynomials
corresponds to the action of $\frac{\dint }{\dint \theta}$ on the space $\text{Trig}_{2d}[\theta]$ of trigonometric polynomials in the sense that the following diagram is commutative:
\begin{equation}\label{eq:comm_diagram}
\begin{tikzcd}
\text{Trig}_{2d}[\theta] \arrow[r, "\Psi"] \arrow[d, "\frac{\dint}{\dint \theta}"]
& \C_{2d}[z] \arrow[d, "\cD_{2d}"] \\
\text{Trig}_{2d}[\theta] \arrow[r, "\Psi"]
& \C_{2d}[z]
\end{tikzcd}
\end{equation}
Indeed, it follows from~\eqref{eq:D_def} that for all $\ell\in \{-d,\ldots, d\}$ we have
\begin{equation}\label{eq:D_basis}
\cD_{2d}(\Psi(\eee^{\ii \ell \theta})) = \cD_{2d}(z^{d + \ell}) = \ii \left((d+\ell) z^{d+\ell} - d z^{d+\ell}\right) = \ii \ell z^{d+\ell} = \Psi\left(\frac{\dint}{\dint \theta}\eee^{\ii \ell \theta}\right).
\end{equation}
Knowing this, the identity  $\cD_{2d}(\Psi(T_{2d}))= \Psi(\frac{\dint}{\dint \theta}T_{2d})$ for every $T_{2d}\in \text{Trig}_{2d}[\theta]$ follows by linearity.
Thus, instead of studying the effect of repeated differentiation on trigonometric polynomials we may study the repeated action of the operator $\cD_{2d}$ on algebraic polynomials of degree $2d$.

It follows from~\eqref{eq:alg_trig_corr} that real-rooted trigonometric polynomials correspond to algebraic polynomials with roots on the unit circle. We are now going to state a result on such algebraic polynomials which is equivalent to Theorem~\ref{theo:main}.
Consider an algebraic polynomial $P_{n}(z)$ of degree $n\in \N$ all of whose zeroes are located on the unit circle. We can represent it in the form
\begin{equation}\label{eq:algebraic_poly_decomposition}
P_{n}(z) = a_{n} \cdot \prod_{j=1}^{n} (z-\eee^{\ii \theta_{j;n}})
\;\;
\text{ for some } \theta_{1;n},\ldots, \theta_{n;n}\in [-\pi,\pi)
\text{ and }
a_{n}\in \C\backslash \{0\}.
\end{equation}
Given arbitrary $\theta_{1;n},\ldots, \theta_{n;n}\in [-\pi,\pi)$ it is convenient to take $a_{n}:= (2\ii)^{-n} \prod_{j=1}^{n} \eee^{-\ii \theta_{j;n}/2}$ in~\eqref{eq:algebraic_poly_decomposition}. With this choice, the following functional equation is satisfied:
\begin{equation}\label{eq:funct_eq}
\bar z^{n} P_{n}(1/\bar z) = \overline{P_{n}(z)}, \qquad z\in \C\backslash\{0\}.
\end{equation}
By comparison of coefficients one shows that for a polynomial represented in the form $P_{n}(z) = \sum_{j=0}^{n} a_j z^j$ the functional equation~\eqref{eq:funct_eq} is satisfied if and only if
\begin{equation}\label{eq:coeff_poly_hermite}
a_j = \overline{a_{n-j}}
\quad
\text{ for all }
j\in \{0,\ldots,n\}.
\end{equation}
If $n=2d$ is even, then $P_n(z)$ corresponds to a trigonometric polynomial $T_{2d}(\theta)$ which can be written in the form
\begin{equation}\label{eq:T_n_theta}
T_{n}(\theta) := \frac{P_{n}(\eee^{\ii \theta})}{\eee^{\ii n \theta /2 }} = \prod_{j=1}^{n} \frac{\eee^{\ii \theta} - \eee^{\ii \theta_{j;n}}}{2\ii\cdot  \eee^{\ii \theta/2} \eee^{\ii \theta_{j;n}/2}} =  \prod_{j=1}^{n} \sin \frac{\theta - \theta_{j;n}}{2}.
\end{equation}
If $n$ is odd, the above equation remains valid but does not define a trigonometric polynomial. Factorization of the form~\eqref{eq:T_n_theta} is well known; see, e.g., \cite[Corollary~1.3]{dzyadyk_shevchuk_book}.
The next lemma is an analogue of Rolle's theorem for polynomials with roots on the unit circle.
\begin{lemma}\label{lem:rolles_unit_circle}
If all roots of a degree $n$ algebraic polynomial $P_{n}\in \C_{n}[z]$ are located on the unit circle, then all roots of $\cD_{n}P_{n}$ also lie on the unit circle.
\end{lemma}
\begin{proof}
Multiplying $P_{n}$ by a suitable $c\in \C\backslash\{0\}$ we may assume that~\eqref{eq:T_n_theta} holds implying that the function $T_{n}(\theta)= P_{n}(\eee^{\ii \theta})/\eee^{\ii n \theta /2 }$ takes real values for $\theta\in \R$. One checks that
\begin{equation}\label{eq:identity_trig_algebr}
\frac{\dint}{\dint \theta} T_{n}(\theta)
=
\frac{(\cD_{n} P_{n}) (\eee^{\ii \theta})}{\eee^{\ii n\theta/2}}.
\end{equation}
Let $n=2d$ be even. The real-valued function $T_{n}(\theta)$ is $2\pi$-periodic and has $n$ zeroes in $[-\pi,\pi)$, counting multiplicities. Rolles's theorem implies that  $\frac{\dint}{\dint \theta}T_n(\theta)$ has $n$ zeroes in $[-\pi,\pi)$, with multiplicities.
%The corresponding algebraic polynomial is $\cD_{n} P_{n}$, see~\eqref{eq:comm_diagram},  meaning that
It follows from~\eqref{eq:identity_trig_algebr} that $\cD_{n} P_{n}$ has $n$ zeroes on the unit circle. Let now $n$ be odd. Then, the function $T_n(\theta)$ is periodic with period $4\pi$ and $T_n(\theta + 2\pi) = -T_n(\theta)$. Since $P_n$ has $n$ zeroes on the unit circle, it follows from the definition of $T_n$ that it has $2n$ zeroes in $[-2\pi, 2\pi)$, counting multiplicities. Rolle's theorem, applied to the $4\pi$-periodic function $T_n$, yields that $\frac{\dint}{\dint \theta}T_n(\theta)$ has $2n$ zeroes on $[-2\pi, 2\pi)$. Finally, \eqref{eq:identity_trig_algebr} implies that $\cD_n P_n$ has $n$ zeroes on the unit circle.
%In the following, we restrict our attention to trigonometric polynomials $T_n$ satisfying two additional requirements.   Equivalently, they satisfy the functional equation
%and form a vector space (over the real numbers) that stays invariant under the action of the operator $\cD_{2n}$.
%If $z\in \C\backslash\{0\}$ is a zero of $P_{2n}$, then $1/\bar z$ is a zero as well.
%The second requirement is that $T_n(\theta)$ has $2n$ zeroes in $[-\pi,\pi)$ or, equivalently, $P_{2n}(z)$ has $2n$  zeroes on the unit circle $\bT$, counted with multiplicities.
\end{proof}

%\begin{remark}
% Similarly,  $T_n(\theta) = \sum_{k=-n}^n c_k \eee^{\ii \theta k}$ takes real values for real $\theta$ if and only if $c_k = \overline c_{-k}$ for all $k\in \{-n,\ldots,n\}$; see~\eqref{eq:def_T_n}.
%\end{remark}

The above discussion shows that the following proposition is essentially equivalent to Theorem~\ref{theo:main}. (Actually, the proposition is slightly stronger since for Theorem~\ref{theo:main} we need only even $n=2d$).
\begin{proposition}\label{prop:main_for_algebraic}
For every $n\in \N$ let $P_{n}(z)$ be an algebraic polynomial with complex coefficients written in the form~\eqref{eq:algebraic_poly_decomposition}.
Suppose that the empirical distribution of roots of $P_{n}$ converges weakly to some probability measure $\nu$ on the unit circle, as $n\to\infty$, that is
$$
\mu\lsem P_n\rsem = \frac {1}{n} \sum_{j=1}^{n} \delta_{\eee^{\ii \theta_{j;n}}} \toweak \nu.
$$
Fix $t>0$ and let $k=k(n)$ be a sequence of positive integers such that $k(n)/n \to t$ as $n\to\infty$.  Then, the empirical distribution of roots of the polynomial $\cD_{n}^{k(n)} P_{n}$ converges  to $\nu \boxtimes \Pi_t$ weakly on the unit circle $\bT$, as $n\to\infty$.
\end{proposition}

\subsection{Finite free multiplicative convolution and circular Laguerre polynomials}\label{subsec:finite_free_conv_laguerre}
Let us recall some notions from finite free probability which was developed in~\cite{marcus_spielman_srivastava} and~\cite{marcus}.
The \textit{finite free multiplicative convolution} $\boxtimes_{n}$ is a bilinear operation on the space $\C_n[z]$ of polynomials of degree at most $n$ which is defined  as follows:
$$
\left(\sum_{j=0}^{n} \alpha_j z^j \right)\boxtimes_{n}\left(\sum_{j=0}^{n} \beta_j z^j \right)
=
\sum_{j=0}^{n} (-1)^{n-j}  \frac{\alpha_j \beta_j}{\binom {n}{j}} z^j.
$$
The polynomial $(z-1)^{n}$ plays the role of the unit element under $\boxtimes_{n}$, namely
$$
(z-1)^{n}\boxtimes_{n} p(z) =  p(z) \text{ for all } p(z) \in \C_{n}[z].
$$
Let $p(z) = \det (zI_{n\times n} - A)$ and $q(z) = \det (zI_{n\times n}  - B)$ be characteristic polynomials of arbitrary normal $n\times n$-matrices $A$ and $B$, where $I_{n\times n}$ is the $n\times n$ identity matrix. Then it is known from~\cite[Theorem~1.5]{marcus_spielman_srivastava} that
\begin{equation}\label{eq:mult_conv_char_poly_haar}
p(z) \boxtimes_{n} q (z) =  \E \det (z I_{n\times n} - A Q B Q^*),
\end{equation}
where $Q$ is a random orthogonal (or unitary) $n\times n$-matrix with Haar distribution.

The next lemma provides an interpretation of the repeated action of the differential operator $\cD_{n}$ on algebraic polynomials  in terms of $\boxtimes_{n}$. A similar interpretation of the repeated action of $\frac{\dd}{\dd z}$ on algebraic polynomials can be found in~\cite[Lemma~3.5]{arizmendi_garza_vargas_perales}.
%The differentiation of algebraic polynomials can be interpreted in a way  similar to~\eqref{eq:diff_oper_as_conv} using the .  should be
%Let us slightly extend the definition of this operator by allowing its index to be odd:  For $d\in \N$ we define the linear operator $\cD_d : \C_d[z] \to \C_d[z]$ by
%\begin{equation}\label{eq:D_def_extended}
%\cD_{d} P (z) = \ii \left(z \frac{\dint P}{\dint z} - \frac d2 P(z) \right),
%\qquad P\in \C_{d}[z].
%\end{equation}

\begin{lemma}\label{lem:laguerre_trig_poly_def}
For all $n\in \N$,  $k\in \N_0$, and all $P\in \C_{n}[z]$ we have
\begin{equation}\label{eq:diff_oper_as_conv}
 \cD_{n}^k P(z) = P(z) \boxtimes_{n} L_{n,k}(z),
\end{equation}
where
\begin{equation}\label{eq:L_2n_def}
L_{n,k} (z) := \ii^k  \sum_{j=0}^{n}(-1)^{n-j} \binom {n}{j} \left(j - \frac n2\right)^k z^j =  \cD_{n}^k (z-1)^{n}.
\end{equation}
\end{lemma}
\begin{proof}
If we choose $1,z,\ldots, z^{n}$ as a basis of $\C_{n}[z]$, the linear operator $\cD_{n}$ becomes diagonal, namely
\begin{equation}\label{eq:D_diagonal}
\cD_{n} z^j = \ii \left(j-\frac n2\right) z^j, \qquad j\in \{0,\ldots, n\}.
\end{equation}
Using this identity, \eqref{eq:diff_oper_as_conv} can be easily easily checked for $P(z) = z^j$, $j\in \{0,\ldots, n\}$. The general case follows by linearity.
\end{proof}

The polynomials $L_{n,k}(z)$ will be referred to as the \textit{circular Laguerre polynomials}\footnote{Circular Hermite polynomials appeared very recently in~\cite{mirabelli_diss} and were studied in~\cite{kabluchko_lee_yang_curie_weiss}.} since they appear in the finite free \textit{multiplicative} analogue of the Poisson limit theorem in the same way as the classical Laguerre polynomials appear in the finite free \textit{additive} analogue~\cite[Section~6.2.3]{marcus} of the Poisson limit theorem; see Section~\ref{subsec:zeroes_asympt_laguerre} for an explanation.  The following semigroup property is a consequence of the definition of $\boxtimes_{n}$:
$$
L_{n,k_1}(z) \boxtimes_{n} L_{n, k_2}(z) = L_{n, k_1 + k_2}(z), \qquad k_1,k_2\in \N_0,
\qquad
L_{n,0}(z) = (z-1)^{n}.
$$

\begin{lemma}\label{lem:laguerre_roots_unit_circle}
For all $n\in \N$ and $k\in \N_0$ all roots of the algebraic polynomial $L_{n,k}(z)$ (whose degree is $n$) are located on the unit circle.
\end{lemma}
\begin{proof}
The claim follows from~\eqref{eq:L_2n_def} and Lemma~\ref{lem:rolles_unit_circle}.
\end{proof}

Conditions ensuring the real-rootedness of the classical Laguerre polynomials are discussed in~\cite[Example 2.8]{arizmendi_garza_vargas_perales}.

\begin{figure}[t]
	\centering
	\includegraphics[width=0.22\columnwidth]{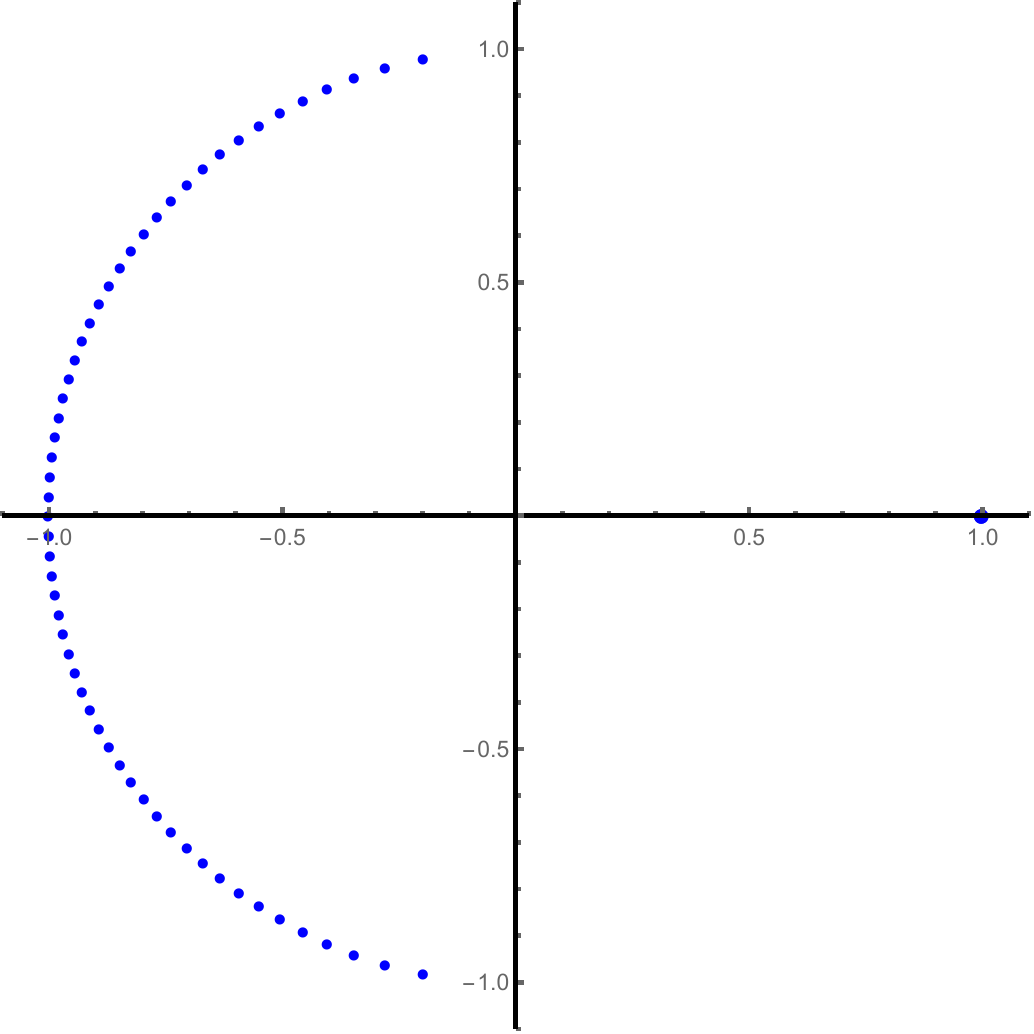}
	\includegraphics[width=0.22\columnwidth]{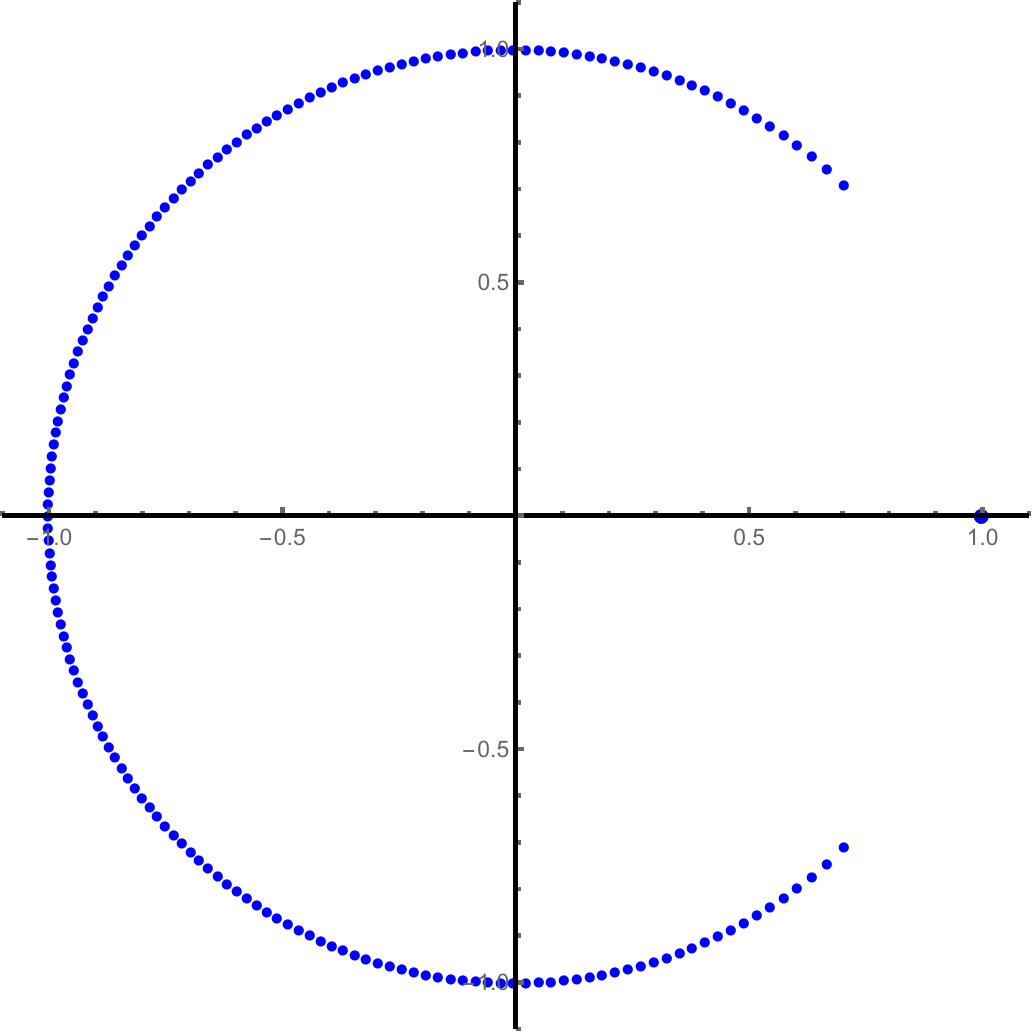}
	\includegraphics[width=0.22\columnwidth]{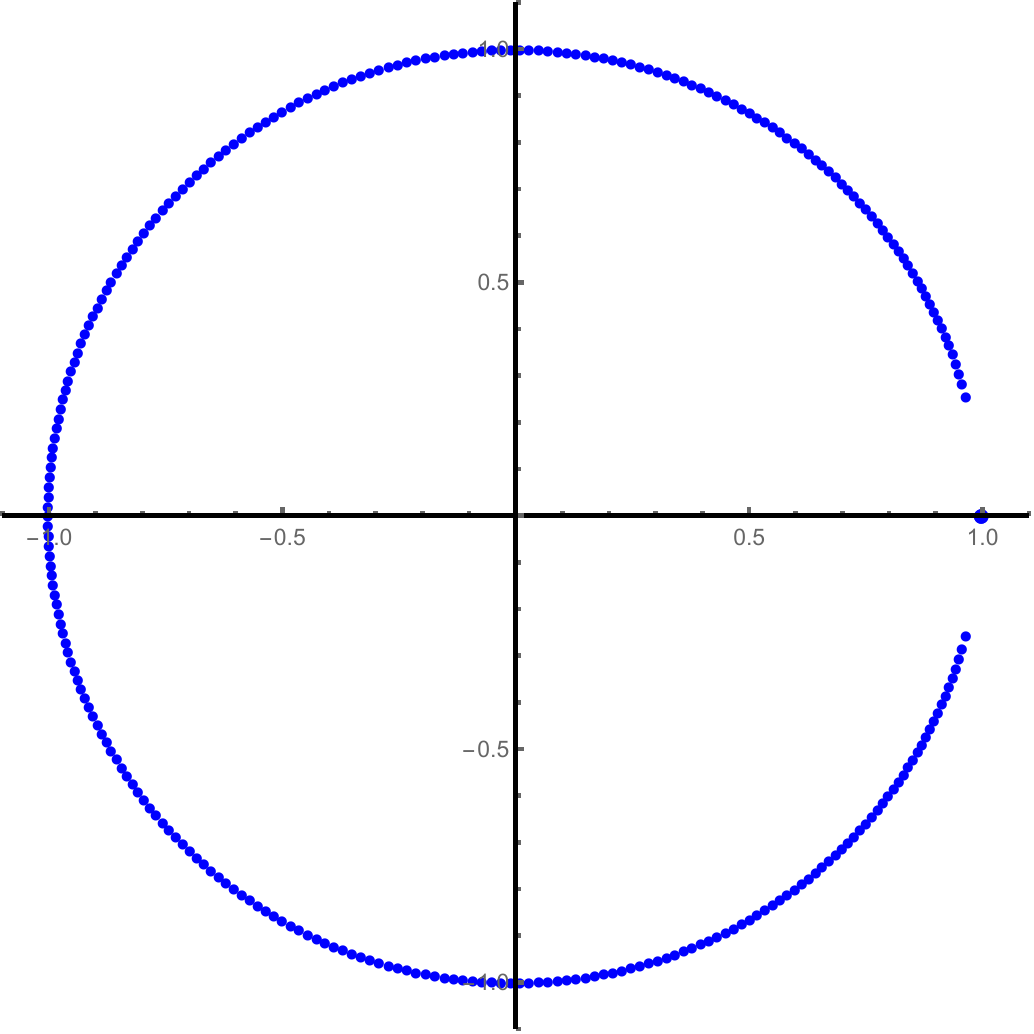}
	\includegraphics[width=0.22\columnwidth]{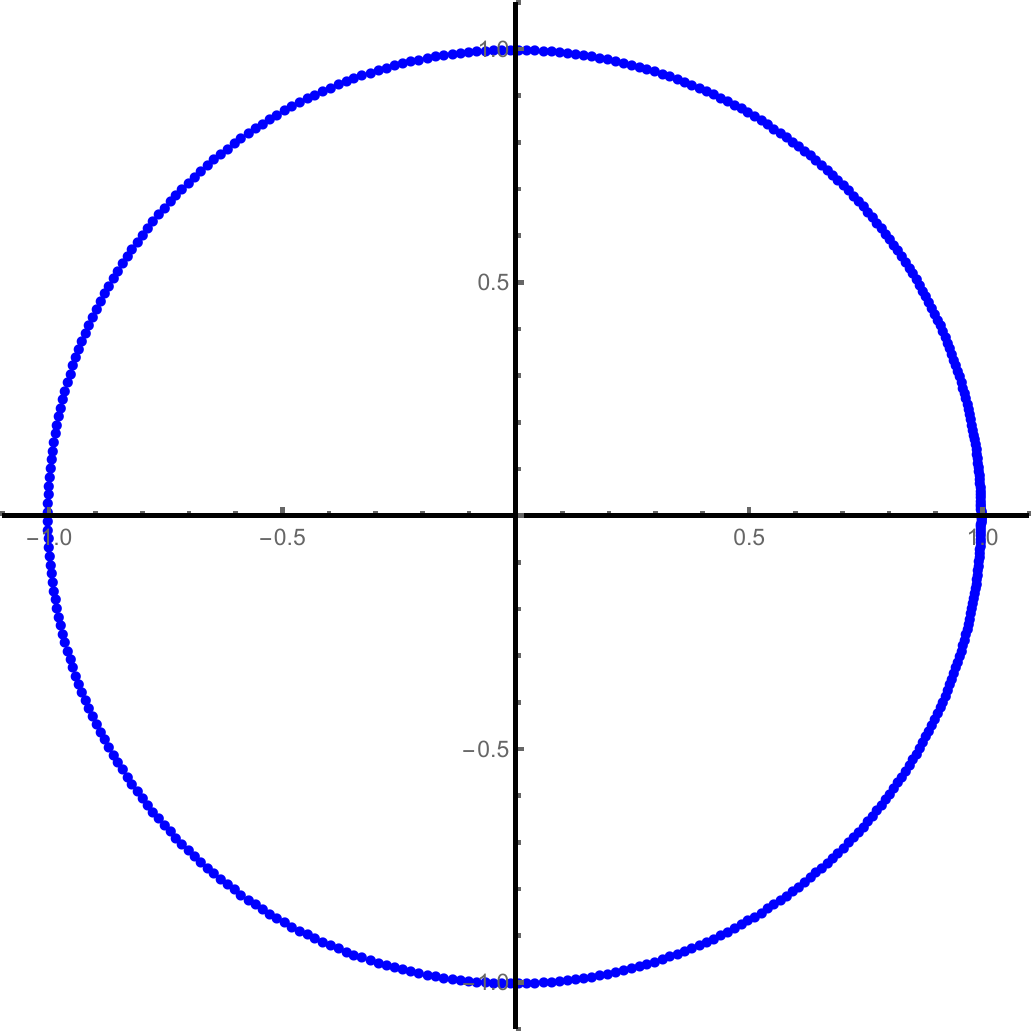}
	\caption{Zeroes of $L_{n,k}(z)$ with $n=400$ and $k=[tn]$, where $t\in \{\frac 17, \frac 37, \frac 57, 1\}$. The multiplicity of the zero at $z=1$  is $n-k$.}
\label{fig:zeroes_laguerre}
\end{figure}

\begin{lemma}\label{lem:trig_laguerre}
For all $n\in \N$ and $k\in \N_0$ we have %The trigonometric polynomial corresponding to $L_{2n,k}(z)$ is given by
\begin{equation}\label{eq:trigonometric_laguerre}
T_{n,k} (\theta) := \frac{L_{n,k}(\eee^{\ii \theta})}{\eee^{\ii n \theta/2}} = (2\ii)^n \frac{\dint^k}{\dint \theta^k} \left(\sin \frac{\theta}{2}\right)^{n}.
\end{equation}
If $n=2d$ is even, then $T_{2d,k} (\theta)$ is the trigonometric polynomial corresponding to $L_{2d,k}(z)$ via~\eqref{eq:alg_trig_corr}.
\end{lemma}
\begin{proof}
By the binomial theorem, we have
$$
(2\ii)^n  \left(\sin \frac{\theta}{2}\right)^{n}
=
(\eee^{\ii \theta/2} -\eee^{-\ii \theta/2})^n
=
\sum_{j=0}^n  (-1)^{n-j} \binom nj  \eee^{\ii \theta(j - \frac n2)}.
$$
The claim follows by differentiating this formula $k$ times in $\theta$.
%
%The trigonometric polynomial corresponding to $L_{2n,0}(z) = (1-z)^{2n}$ is
%$$
%\frac{L_{2n,0}(\eee^{\ii \theta})}{\eee^{\ii n \theta}} = (\eee^{-\ii \theta/2} - \eee^{\ii \theta /2})^{2n} = (-4)^n %\left(\sin\frac\theta 2\right)^{2n}.
%$$
%By the commutativity of the diagram~\eqref{eq:comm_diagram}, the trigonometric polynomials corresponding to $L_{2n,k} (z) = %\cD_{2n}^k (1-z)^{2n}$ are those appearing on the right-hand side of~\eqref{eq:trigonometric_laguerre}.
\end{proof}

\begin{lemma}
For $n\in \N$ and $0\leq k\leq n$, the polynomial $L_{n,k}(z)$ is divisible by $(z-1)^{n-k}$.
\end{lemma}
\begin{proof}
The function $\theta\mapsto \frac{\dint^k}{\dint \theta^k} (\sin \frac{\theta}{2})^{n}$ has a zero of multiplicity $n-k$ at $\theta=0$. By~\eqref{eq:trigonometric_laguerre}, it corresponds to a zero of $L_{n,k}(z)$ of multiplicity $n-k$ at $z=1$.
\end{proof}

\subsection{Zeroes and asymptotics of the circular Laguerre polynomials}\label{subsec:zeroes_asympt_laguerre}
The next theorem, whose proof is postponed to Section~\ref{sec:zeroes_laguerre},
describes the asymptotic distribution of zeroes of $L_{n,k}$ in the regime when $n\to\infty$ and $k= k(n) \sim tn$ for some constant $t>0$; see Figure~\ref{fig:zeroes_laguerre}.  Recall that the empirical distribution of zeroes of an algebraic polynomial $P$ is the probability measure on $\C$ defined by
$$
\mu\lsem P \rsem := \frac 1{\deg P} \sum_{z\in \C:\, P(z) = 0} m_P(z) \delta_{z},
$$
where $m_P(z)$ is the multiplicity of the zero at $z$.
\begin{theorem}\label{theo:zeroes_trig_laguerre}
Fix $t>0$ and let $k=k(n)$ be a sequence of positive integers with $k(n)/n \to t$ as $n\to\infty$.
%and let $k=k(n)= [2nt]$.
Then the empirical distribution of zeroes of the algebraic polynomial $L_{n,k(n)}(z)$ converges weakly on $\bT$ to the free unitary Poisson distribution with parameter $t$,  that is
$$
\mu\lsem L_{n, k(n)}\rsem \toweak \Pi_t.
$$
\end{theorem}

Let us explain why the appearance of the free unitary Poisson distribution in Theorem~\ref{theo:zeroes_trig_laguerre} is natural.
The empirical distribution of zeroes of the polynomial $L_{n, 1}(z) = \cD_{n} (z-1)^{n} = \ii (n/2) (z-1)^{n-1}(z+1)$ is given by
$$
\mu \lsem L_{n, 1}\rsem =  \left(1-\frac 1 {n}\right) \delta_1 + \frac 1 {n} \delta_{-1}.
$$
The classical Poisson limit theorem states that the $[t n]$-th convolution of this distribution converges to the Poisson distribution with parameter $t>0$. Its finite free \textit{multiplicative} analogue states that for every $k\in \N$,
\begin{equation}\label{eq:poisson_limit_with_n}
L_{n,k} (z) =  \underbrace{L_{n, 1}(z)\boxtimes_{n} \ldots \boxtimes_{n} L_{n, 1}(z)}_{k \text{ times}} = (L_{n, 1}(z))^{\boxtimes_{n}\, \text{($k$ times)}}
\end{equation}
is the $k$-th $\boxtimes_{n}$-convolution of $L_{n,1}$; see~\cite[Theorem~6.9]{marcus} for the finite free \textit{additive} analogue in which the classical Laguerre polynomials appear.
Now we let $n\to\infty$. Then,  Theorem~\ref{theo:zeroes_trig_laguerre} states that for $k\sim tn$ the empirical distribution of zeroes of these polynomials converges weakly  to $\Pi_t$.
%\begin{equation}\label{eq:poisson_limit_with_n_rep_poly}
%\mu_{\underbrace{L_{2n, 1}(z)\boxtimes_{2n} \ldots \boxtimes_{2n} L_{2n, 1}(z)}_{k(n) \text{ times}}} \toweak \Pi_t.
%\end{equation}
%Now, the zero-counting measure of the algebraic polynomial $L_{n, 1}(z) = \cD_{n} (1-z)^{n} = -\ii (n/2) (1-z)^{n-1}(1+z)$ is given %by
%$$
%\mu \lsem L_{n, 1}\rsem =  \left(1-\frac 1 {n}\right) \delta_1 + \frac 1 {n} \delta_{-1}.
%$$
It has to be compared to the Poisson limit theorem for $\boxtimes$ (not $\boxtimes_n$!) on the unit circle, see~\cite[Lemma~6.4]{bercovici_voiculescu_levy_hincin} or Section~\ref{subsec:free_unitary_poi_def}, which states that
\begin{equation}\label{eq:poisson_limit_without_n}
%\underbrace{\left(\left(1-\frac 1 {2n}\right) \delta_1 + \frac 1 {2n} \delta_{-1} \right)\boxtimes \ldots \boxtimes \left(\left(1-\frac 1 {2n}\right) \delta_1 + \frac 1 {2n} \delta_{-1} \right)}_{k(n) \text{ times}} \toweak \Pi_t.
\left(\left(1-\frac 1 {n}\right) \delta_1 + \frac 1 {n} \delta_{-1} \right)^{\boxtimes \; \text{($k(n)$ times)}} \toweak \Pi_t.
\end{equation}
To deduce Theorem~\ref{theo:zeroes_trig_laguerre} from this statement, one essentially needs to justify that replacing $\boxtimes$ by $\boxtimes_{n}$ (and measures by polynomials) does not change the large $n$ limit. This task is less trivial than it looks at a first sight.
%%%old text from the first arxiv version:
%For example, trying to prove Theorem~\ref{theo:zeroes_trig_laguerre} by computing moments\footnote{Using Vieta's and  Girard–Newton's identities it is possible to compute the arithmetic mean of the $\ell$-th powers of the roots of $L_{n,k}$, and then to check that it converges to the $\ell$-th moment of $\Pi_t$. Although this can be done for every given $\ell=1,2,\ldots$, we were unable to find a general proof based on this approach.} or cumulants leads to non-trivial combinatorial identities.????
We shall prove Theorem~\ref{theo:zeroes_trig_laguerre} in Section~\ref{sec:zeroes_laguerre} by the saddle-point method\footnote{After this paper appeared on the arXiv, a combinatorial proof has been found by\citet{arizmendi_fujie_ueda_new_comb}}.  As a by-product, we shall identify the rate of growth of $L_{n,k}$ and its logarithmic derivative on the  open unit disk.
%The next result will be established in Section~\ref{sec:zeroes_laguerre}.
\begin{theorem}\label{theo:L_n_k_asymptotics}
Fix $t>0$ and let $k=k(n)$ be a sequence of positive integers with $k(n)/n \to t$ as $n\to\infty$. For every $\theta$ in the open upper half-plane $\bH:= \{\theta \in \C: \Im \theta >0\}$ we have
\begin{align}
\lim_{n\to\infty}\frac 1n \log \left(\frac{L_{n,k}(\eee^{2\ii \theta})}{L_{n,k}(0)}\right)
&=
-\ii \left(\zeta_t(\theta) - \theta - \ii t\right) + \log (1- \eee^{2\ii \zeta_t(\theta)}) - t \log\left(\frac{\zeta_t(\theta) - \theta}{\ii t}\right),\label{eq:asympt_L_n_k_1}\\
\lim_{n\to\infty}\frac 1n \frac{L_{n,k}'(\eee^{2\ii \theta})}{L_{n,k}(\eee^{2\ii \theta})}
&=
\frac {1 - \ii \cot \zeta_t(\theta)} {2 \eee^{2\ii \theta}},\label{eq:asympt_L_n_k_2}
\end{align}
where $\zeta= \zeta_t(\theta)$ is the unique solution of the equation $\zeta - t \tan \zeta = \theta$ in $\bH$. The function $\zeta_t:\bH\to \bH$ is analytic and satisfies $\zeta_t(\theta) - \theta \to \ii t$ as $\Im \theta \to +\infty$ (uniformly in $\Re \theta\in \R$); see Section~\ref{sec:principal_branch} for more properties. The logarithms in~\eqref{eq:asympt_L_n_k_1} are chosen such that $\log 1 =0$ and all functions of the form $\log (\ldots)$ are continuous (and analytic).
\end{theorem}

It should be possible to prove more refined asymptotic properties of the zeroes of $L_{n,k}$. For example, the zeroes should exhibit crystallization phenomenon in the bulk, that is they should approach locally an arithmetic progression; c.f.~\cite{pemantle_subramanian_crystallization}. For $t\in (0,1)$, the minimal argument of the zeroes (excluding the zero at $1$) is conjectured to converge to $2 \arccos\sqrt t - 2 \sqrt{t(1-t)}$; see~\eqref{eq:supp_free_poi} for a corresponding formula on the support of $\Pi_t$.

Asymptotic rate of growth, as in Theorem~\ref{theo:L_n_k_asymptotics}, and limiting distribution of zeros, as in Theorem~\ref{theo:zeroes_trig_laguerre}, are known for many  sequences of polynomials; see, e.g., \cite{lubinski_review} and~\cite{van_assche_book_asymptotics_ortho_polys} for reviews  and~\cite{dette_studden,elbert1,elbert2,gawronski_on_the_asymptotic,gawronski_strong_laguerre_hermite,lubinski_sidi_strong_asympt} for specific examples.

%The proof of~\eqref{eq:poisson_limit_without_n} can be found in~\cite[Lemma~6.4]{bercovici_voiculescu_levy_hincin}.

\subsection{Products of random reflections}
%???? Theorem~\ref{theo:zeroes_trig_laguerre}.
Let $U_{n}$ be a random vector distributed uniformly on the unit sphere in $\R^{n}$ and let $T_{n}: \R^n \to \R^n$ be a random reflection w.r.t.\ the linear hyperplane orthogonal to $U_n$, that is $T_n v = v - 2 \langle v, U_n \rangle U_n$ for $v\in \R^n$. The characteristic polynomial of $T_n$ is
$$
\det(z I_{n\times n} - T_n) = (z+1) (z-1)^{n-1} = -\frac{2 \ii}{n} L_{n,1}(z).
$$
%where $I_{n\times n}$ is the $n\times n$-identity matrix.
If $T_{n;1}, \ldots, T_{n;k}$ are $k=k(n)$ stochastically independent copies of $T_n$, then by~\cite[Theorem~1.5]{marcus_spielman_srivastava}, see~\eqref{eq:mult_conv_char_poly_haar}, the expected characteristic polynomial of their product is given by
$$
\E \det (z I_{n\times n}  - T_{n;1} \ldots T_{n;k}) = \left((z+1) (z-1)^{n-1}\right)^{\boxtimes_n \, \text{($k$ times)}} = \left(-\frac{2 \ii}{n} L_{n,1}(z)\right)^{\boxtimes_n \, \text{($k$ times)}}.
$$
Theorem~\ref{theo:zeroes_trig_laguerre} suggests that in the regime when $k(n) \sim tn$ as $n\to\infty$, the empirical eigenvalue distribution of the product $T_{n;1} \ldots T_{n;k}$ converges to  $\Pi_t$ weakly on $\mathcal M_\bT$, while in the regime when $k(n) / n \to \infty$, the empirical eigenvalue distribution converges weakly to the uniform distribution on $\bT$. The former claim is a special case of Theorem 3 in the paper of Marchenko and Pastur~\cite{marchenko_pastur}.
Later products of random reflections were studied by Diaconis and Shahshahani~\cite{diaconis_shahshahani} and Porod~\cite{porod_phd,porod_cut_off,porod_cut_off_II} (see also~\cite{rosenthal_rotations} and~\cite{meliot_cut_off} for related results)  who showed that there is a phase transition for the total variation distance between the distribution of $T_{n;1} \ldots T_{n;k}$ and the Haar measure on the orthogonal group in the regime when $k(n) = \frac 12 n \log n + O(n)$.

\subsection{Proof of Theorem~\ref{theo:main} given Theorem~\ref{theo:zeroes_trig_laguerre}}
As explained in Section~\ref{subsec:alg_and_trig_poly}, it suffices to prove Proposition~\ref{prop:main_for_algebraic}.
Recall from~\eqref{eq:diff_oper_as_conv} that  $\cD_{n}^{k(n)} P_{n} = P_{n} \boxtimes_{n} L_{n,k(n)}$.
By Theorem~\ref{theo:zeroes_trig_laguerre}, $\mu \lsem L_{n,k(n)}\rsem$ converges weakly to $\Pi_t$, while $\mu\lsem P_n\rsem$ converges weakly to  $\nu$ by assumption. The subsequent proposition implies that $\mu \lsem \cD_{n}^{k(n)}P_n\rsem$ converges weakly to  $\nu \boxtimes \Pi_t$  and completes the proof of Proposition~\ref{prop:main_for_algebraic}.
\begin{proposition}\label{prop:finite_free_mult_conv_to_free_mult}
Let $(p_n(z))_{n\in \N}$ and $(q_n(z))_{n\in \N}$ be sequences of polynomials in $\C[z]$ with $\deg p_n = \deg q_n = n$. Suppose that all roots of $p_n$ and $q_n$ are located on the unit circle and that  $\mu\lsem p_{n}\rsem$ and $\mu\lsem q_{n}\rsem$ converge weakly to two probability  measures $\nu$ and $\rho$ on $\bT$,  as $n\to\infty$.
Then, all roots of the polynomial $p_n \boxtimes_n q_n$ are also located on the unit circle and $\mu \lsem p_n \boxtimes_n q_n\rsem$ converges weakly to $\nu \boxtimes \rho$.
\end{proposition}
\begin{proof}
All roots of $p_n \boxtimes_n q_n$ are located on $\bT$ by~\cite[Satz~3, p.~36]{szegoe_bemerkungen}.
Since the empirical measures $\mu\lsem p_{n}\rsem$ and $\mu\lsem q_{n}\rsem$ are concentrated on $\bT$, the corresponding moments (or Fourier coefficients) also converge to the moments of $\nu$ and $\rho$. By~\cite[Proposition~3.4]{arizmendi_garza_vargas_perales}, the moments of  $\mu\lsem p_{n} \boxtimes_n q_{n}\rsem$ converge to the moments of $\nu \boxtimes \rho$. Actually, this applies to the moments of positive order $\ell\in \N$, but the same argument applied to the polynomials $\overline{p_n(\bar z)}$ and $\overline{q_n(\bar z)}$ yields the same conclusion for  moments of arbitrary integer order $\ell\in \Z$. To complete the proof, we apply  Weyl's criterion~\cite[p.~50]{billingsley_book68}, which states that for probability measures on the unit circle,  the convergence of Fourier coefficients is equivalent to weak convergence.
\end{proof}

\section{Asymptotics of the circular Laguerre polynomials}\label{sec:zeroes_laguerre}
\subsection{Statement of results and method of proof}
In this section we study the asymptotics of the circular Laguerre polynomials and their zeroes in the regime when $k$ grows linearly with $n$. Let $k=k(n)$ be a sequence of natural numbers satisfying $k(n)\sim t n$ for some fixed $t>0$.  The zeroes of the circular Laguerre  polynomial $L_{n,k}(z)$ defined in~\eqref{eq:L_2n_def} lie on the unit circle by Lemma~\ref{lem:laguerre_roots_unit_circle}, and we denote them by $\eee^{\ii \theta_{1;n}}, \ldots, \eee^{\ii \theta_{n;n}}$ with some $\theta_{1;n}, \ldots, \theta_{n;n} \in [-\pi, \pi)$.
The next theorem is a restatement of Theorem~\ref{theo:zeroes_trig_laguerre}.
\begin{theorem}\label{theo:zeroes_laguerre_restatement}
The following weak convergence of probability measures on the unit circle $\bT$ holds:
\begin{equation}\label{eq:zeroes_laguerre_restatement}
\mu_{n,k}:= \mu\lsem L_{n, k}\rsem = \frac 1 {n}  \sum_{j=1}^{n} \delta_{\eee^{\ii \theta_{j;n}}}
\toweak
\Pi_t.
\end{equation}
\end{theorem}

%\medskip
%\noindent
%\textit{Plan of the proof Theorem~\ref{theo:zeroes_laguerre_restatement}}.
Let us describe our strategy to prove  Theorem~\ref{theo:zeroes_laguerre_restatement}.
Recall that the $\psi$-transform of a probability measure $\mu$ on $\bT$ is given by
\begin{equation}\label{eq:psi_transf_def_rep}
\psi_\mu(z) = \int_\bT \frac{uz}{1-uz} \mu(\dint u) = \sum_{\ell=1}^\infty z^\ell \int_{\bT} u^\ell \mu(\dint u),
\qquad |z|<1.
\end{equation}
To prove~\eqref{eq:zeroes_laguerre_restatement}, it suffices to show that for some  $0 < r_1 <r_2 < 1$ we have
\begin{equation}\label{eq:psi_transf_conv_unif}
\lim_{n\to\infty} \psi_{\mu_{n,k}}(z) = \psi_{\Pi_t}(z)
\quad
\text{ uniformly on } r_1 \leq |z|\leq r_2.
\end{equation}
Indeed, by Cauchy's formula, uniform convergence of analytic functions on the annulus $\{r_1 \leq |z|\leq r_2\}$ implies convergence of their $\ell$-th derivatives at $0$, for all $\ell\in \N_0$.  Applying~\eqref{eq:psi_transf_def_rep} combined with this observation  yields
$$
\int_{\bT} u^\ell \mu_{n,k}(\dint u) =  \frac{\psi_{\mu_{n,k}}^{(\ell)}(0)}{\ell!}  \ton \frac{\psi_{\Pi_t}^{(\ell)}(0)}{\ell!} = \int_{\bT} u^\ell \Pi_t(\dint u),
$$
for all $\ell\in \N_0$. Since we are dealing with measures invariant under conjugation $z\mapsto \bar z$, the same claim holds in fact for all $\ell\in \Z$. By the Weyl criterion~\cite[p.~50]{billingsley_book68}, this implies the weak convergence $\mu_{n,k} \to  \Pi_t$ stated in Theorem~\ref{theo:zeroes_laguerre_restatement}.

%\begin{theorem}\label{theo:zeroes_laguerre_restatement}
%Let $\theta_{1;n},\ldots, \theta_{2n; n}$ be the zeroes of $T_n(\theta)$ in $[-\pi,\pi)$, counting multiplicities.
%\end{theorem}

\medskip
\noindent
\textit{The $\psi$-transform of $\mu_{n,k}$}.
So, our aim is to prove~\eqref{eq:psi_transf_conv_unif}. To express the $\psi$-transform of $\mu_{n,k}$, it will be convenient to introduce the following functions that are closely related to $L_{n,k}(z)$:
\begin{equation}\label{eq:W_n_def}
W_{n,k}(\theta)
:=
\frac{1}{k!} \left(\frac{\dint}{\dint \theta}\right)^{k} \left(\sin \frac{\theta}{2}\right)^{n}
=
\frac{1}{(2\ii)^n k!} \frac{L_{n,k}(\eee^{\ii \theta})}{\eee^{i \theta n/2}}.
\end{equation}
Note that the second equality follows from Lemma~\ref{lem:trig_laguerre}.

\begin{lemma}\label{lem:psi_transf_emp_zeroes}
For all $\theta\in \C$ with $\Im \theta >0$ and all $n\in\N$, $k\in \N_0$ we have
$$
\psi_{\mu_{n,k}} (\eee^{\ii \theta}) = \frac \ii {n} \frac{W_{n,k}'(\theta)}{W_{n,k}(\theta)} - \frac 12.
$$
\end{lemma}
\begin{proof}
%Recalling~\eqref{subsec:alg_and_trig_poly}, we can write
%$$
%T_{n}(\theta) = c_n \prod_{j=1}^{2n} \sin \frac{\theta - \theta_{j;n}}{2}
%$$
%for suitable constants $c_n$.
%By Lemma~\ref{lem:trig_laguerre}, $W_n(\theta)$ is closely related to $L_{n,k}(\eee^{\ii \theta})$. Indeed,
Recalling~\eqref{eq:L_2n_def} we can write
$$
L_{n,k}(z)
=
\ii^{k} \sum_{j=0}^{n} (-1)^{n-j} \binom {n}{j} \left(j - \frac n2\right)^{k} z^j
=
\left(\frac {\ii n}{2}\right)^{k} \prod_{j=1}^n (z-\eee^{\ii \theta_{j;n}}).
$$
Taking $z=0$, we obtain $\prod_{j=1}^n \eee^{\ii \theta_{j;n}} = (-1)^k$ and hence $\prod_{j=1}^n \eee^{\ii \theta_{j;n}/2} = \pm \ii^k$.
Then, by~\eqref{eq:W_n_def} we have
\begin{equation*}
W_{n,k}(\theta) = \frac{1}{(2\ii)^n k!} \frac{L_{n,k}(\eee^{\ii \theta})}{\eee^{i \theta n/2}}
=
\pm \frac {n^k} {2^{k} k!}  \prod_{j=1}^{n} \sin \frac{\theta - \theta_{j;n}}{2}.
\end{equation*}
Taking the logarithmic derivative  and dividing by $n$ we obtain
\begin{equation}\label{eq:log_der_T_n_tech}
\frac 1{n} \frac{W_{n,k}'(\theta)}{W_{n,k}(\theta)} = \frac 1 {2n} \sum_{j=1}^{n} \cot \frac{\theta - \theta_{j;n}}{2}
=
- \frac \ii 2 - \ii \sum_{\ell = 1}^\infty \eee^{\ii \theta \ell}\frac 1{n}\sum_{j=1}^{n} \eee^{-\ii \ell \theta_{j;n}},
\end{equation}
where we used the formula
\begin{equation}\label{eq:cotan_asympt}
\cot \frac {\Theta} 2
=
\ii\, \frac{\eee^{\frac{\ii \Theta}{2}} + \eee^{-\frac{\ii \Theta}{2}}}{\eee^{\frac{\ii \Theta}{2}} - \eee^{-\frac{\ii \Theta}{2}}}
=
 -\ii \, \frac {1 + \eee^{\ii \Theta}}{1 - \eee^{\ii \Theta}}
=
-\ii\, \left( 1 + 2 \sum_{\ell = 1}^\infty \eee^{\ii \ell \Theta}\right)
,
\qquad \Im \Theta>0.
\end{equation}
By~\eqref{eq:W_n_def}, the function $W_{n,k}(\theta)$ is even (respectively, odd) if $n-k$ is even (respectively, odd). In either case, its zeroes are symmetric w.r.t.\ $0$ and we can replace $-\ii \ell \theta_{j;n}$ by $\ii \ell \theta_{j;n}$ in~\eqref{eq:log_der_T_n_tech}.  Recall from~\eqref{eq:psi_transf_def_rep} that the $\psi$-transform of $\mu_{n,k}$ can be written as
$$
\psi_{\mu_{n,k}} (z) = \sum_{\ell=1}^\infty z^\ell \frac 1{n}\sum_{j=1}^{n} \eee^{\ii \ell \theta_{j;n}},
\qquad |z|<1.
$$
Together with~\eqref{eq:log_der_T_n_tech}, this yields the statement of the lemma if we put $z:= \eee^{\ii \theta}\in \bD$.
\end{proof}

\medskip
\noindent
\textit{Asymptotic results on $W_{n,k}(\theta)$}. In the following, we shall investigate the asymptotic behaviour of $\log W_{n,k}(\theta)$ and $W_{n,k}'(\theta)/W_{n,k}(\theta)$ as $n\to\infty$ using the saddle-point method. Our results will be stated in terms of a family of  analytic functions $\zeta_t(\theta)$  (indexed by a parameter $t>0$) that, as it will turn out,  solve the saddle-point equation.

\begin{theorem}\label{theo:riemann_surface_z_i}
Fix $t > 0$. Let $\bH := \{\theta\in \C: \Im \theta > 0\}$ be the upper half-plane. For every  $\theta\in \bH$,  the equation $\zeta - t \tan \zeta = \theta$ has a unique solution $\zeta= \zeta_t(\theta)$ in $\bH$, and this solution is simple.\footnote{Simplicity means that the derivative of $\zeta\mapsto \zeta - t \tan \zeta$ does not vanish at $\zeta_t(\theta)$.} The function $\zeta_t:\bH\to \bH$ is analytic on $\bH$, admits a continuous extension to the closed upper half-plane $\bar \bH$, and satisfies
$\zeta_t(\theta + \pi ) = \zeta_t(\theta) + \pi$ as well as $\Im \zeta_t(\theta) > \Im \theta$ for all $\theta\in \bH$.  Finally, we have $\zeta_{t} (\theta) -\theta \to  \ii t$ as $\Im \theta\to +\infty$ (uniformly in $\Re \theta\in \R$).
\end{theorem}

We shall prove Theorem~\ref{theo:riemann_surface_z_i} in Section~\ref{sec:principal_branch}, where more properties of the function $\zeta_t(\theta)$ will be established. The next result on the asymptotics of $W_{n,k}$ will be proved in Section~\ref{sec:trig_laguerre_zeroes_proof}.
% Explicit series expansions of $z_t (\theta)$ will be stated in Section~\ref{sec:free_unitary_poi_properties}.

\begin{theorem}\label{theo:limit_log_der}
Let $t>0$ be fixed and recall that $k/n \to t$ as $n\to\infty$.  Then, uniformly on compact subsets of the upper half-plane $\bH = \{\theta\in \C: \Im \theta >0\}$ it holds that
\begin{align}
\lim_{n\to\infty} \frac 1{n} \log |W_{n,k}(\theta)|
&=
\log |\sin \zeta_t(\theta/2)| - t \log |2 \zeta_{t}(\theta/2) - \theta|,
\\
\lim_{n\to\infty} \frac 1{n} \frac{W_{n,k}'(\theta)}{W_{n,k}(\theta)}
&=
\frac 12  \cot \zeta_t(\theta/2)
=
\frac{t}{2 \zeta_t(\theta/2) - \theta}. \label{eq:W_log_der_conv}
%\end{equation}
%With a suitable choice of the branches of the logarithms, we also have, again uniformly on compact subsets of $\bH$,
%\begin{equation}
\end{align}
\end{theorem}

The next lemma expresses the $\psi$-transform of the free unitary Poisson distribution $\Pi_t$ in terms of the function $\zeta_t$.
\begin{lemma}\label{lem:free_poi_psi_transf}
Fix some $t>0$. Then, for all $\theta \in \bH$ we have
\begin{equation}\label{eq:free_poi_psi_transf}
\psi_{\Pi_t}(\eee^{\ii \theta})
=
\frac \ii 2 \cot \zeta_t(\theta/2)-\frac 12
=
\frac{\ii t}{2\zeta_t(\theta/2) - \theta}-\frac 12.
\end{equation}
\end{lemma}
\begin{proof}
By the definition of the free unitary Poisson distribution $\Pi_t$ given in~\eqref{eq:S_transf_def} and~\eqref{eq:S_transf_poi}, its $\psi$-transform satisfies
\begin{equation}\label{eq:identify_poi1}
w = \psi_{\Pi_t}\left(\frac{w}{1+w} \exp\left\{ \frac{t}{w + \frac 12}\right\}\right)
\end{equation}
provided $w\in \C$ and $|w|$ is sufficiently small. Let us put
$$
w := \frac \ii 2\cot \zeta_t(\theta/2) - \frac 12 =\frac{\ii t}{2\zeta_t(\theta/2) - \theta}- \frac 12.
$$
If $\Im \theta$ is sufficiently large, then $|w|$ is sufficiently small (since $2\zeta_{t} (\theta/2) -\theta \to  2\ii t$ as $\Im \theta\to +\infty$) and~\eqref{eq:identify_poi1} holds.  Elementary transformations yield
\begin{equation}\label{eq:identify_poi2}
\frac{w}{1+w} \exp\left\{\frac {t}{w+\frac 12}\right\}
=
\frac{ \frac \ii {2} \cot \zeta_t(\theta/2) - \frac 12}{\frac \ii {2} \cot \zeta_t(\theta/2) + \frac 12}
\eee^{-\ii (2\zeta_t(\theta/2) - \theta)}
=
\eee^{2\ii  \zeta_t(\theta/2)}
\eee^{-\ii (2 \zeta_t(\theta/2) - \theta)}
=
\eee^{\ii \theta}.
\end{equation}
It follows that~\eqref{eq:free_poi_psi_transf} holds provided $\Im \theta$ is sufficiently large. By the principle of analytic continuation, \eqref{eq:free_poi_psi_transf} continues to hold for all $\theta\in \bH$.
\end{proof}

\begin{proof}[Completing the proof of Theorem~\ref{theo:zeroes_laguerre_restatement}]
As already explained above, it suffices to prove~\eqref{eq:psi_transf_conv_unif}, i.e.\ that
$\psi_{\mu_{n,k}}(z) \to  \psi_{\Pi_t}(z)$ uniformly on  $\{r_1 \leq |z| \leq r_2\}$.
Combining Lemma~\ref{lem:psi_transf_emp_zeroes}, Theorem~\ref{theo:limit_log_der} and  Lemma~\ref{lem:free_poi_psi_transf} we obtain
$$
\psi_{\mu_{n,k}} (\eee^{\ii \theta}) =
\frac \ii {n} \frac{W_{n,k}'(\theta)}{W_{n,k}(\theta)} - \frac 12
\ton
\frac{\ii t}{2\zeta_t(\theta/2) - \theta}- \frac 12
=
\psi_{\Pi_t}(\eee^{\ii \theta})
$$
for all $\theta\in \bH$, the convergence being  uniform on compact subsets of $\bH$. This implies~\eqref{eq:psi_transf_conv_unif} and proves Theorem~\ref{theo:zeroes_laguerre_restatement}. Observe that for this proof a weaker form of Theorem~\ref{theo:limit_log_der} suffices: we only need that~\eqref{eq:W_log_der_conv} holds uniformly on some rectangle of the form $\{-\pi \leq \Re \theta\leq \pi, \theta_1 \leq \Im \theta \leq \theta_2\}$.
\end{proof}

\subsection{Proof of Theorem~\ref{theo:limit_log_der}: A saddle-point argument}
\label{sec:trig_laguerre_zeroes_proof}
The proof is subdivided into several steps. It will be convenient to write
%to replace $k$ by $k-1$, $\theta/2$ by $\theta$ and to write
\begin{equation}\label{eq:V_W}
V_{n,k}(\theta):=
%W_{n,k-1}(\theta)
\frac{1}{(k-1)!} \left(\frac{\dint}{\dint \theta}\right)^{k-1} (\sin\theta)^{n} = 2^{k-1} W_{n,k-1}(2\theta).
\end{equation}
%\medskip
%The remainder of the present Section~\ref{sec:trig_laguerre_zeroes_proof} is devoted to the proof of~\eqref{eq:psi_transf_conv_unif}. To compute the limit of $\psi_{\mu_{n,k}}$, which can be expressed through the logarithmic derivative of $T_n$ by Lemma~\ref{lem:psi_transf_emp_zeroes}, we shall compute the limit of $\frac 1n \log T_n(\theta)$ for sufficiently large $\Im \theta>0$, and then take the derivative.
%Let $\mu$ be a probability measure on $[-\pi, \pi)$. Denote by $\mu_{\text{per}}$ its periodic extension to $\R$.
%$$
%G(x) = \frac 12 \int_{-\pi}^\pi \cot \frac {x-y}2 \mu (\dint y) = \int_{-\infty}^{\infty} \frac{1}{x-y} \mu_{\text{per}}(\dint y).
%$$
%For convenience we shall replace $k$ by $k-1$.

\medskip
\noindent
\textit{Cauchy's formula}.
We start by representing $V_{n,k}(\theta)$ as an integral. For all $\theta\in \C$, Cauchy's formula yields
\begin{equation}\label{eq:cauchy_for_der}
V_{n,k}(\theta)
=
\frac 1 {2\pi \ii} \oint \frac{(\sin z)^{n}}{(z-\theta)^k} \dint z
=
\frac 1 {2\pi \ii} \oint (f_{k/n}(z; \theta))^k \dint z,
\end{equation}
where the integral is taken over some small counterclockwise loop around $\theta$, and for $s>0$ we defined
\begin{equation}\label{eq:f_n_def}
f_s(z;\theta)
:=
\frac{\exp\left\{\frac {1}{s} \log \sin z\right\}}{z-\theta},
\qquad \Im z>0,\; z\neq \theta.
\end{equation}
Since the function $z\mapsto \sin z$ vanishes if and only if $z\in \pi \Z$, restricting our attention to the upper half-plane $\bH$ we can construct a well-defined branch of the function $\log \sin z$. In fact, there exist infinitely many  branches differing by  $2\pi \ii m$, $m \in \Z$.
%Although this is not essential for what follows,
We can specify the concrete branch we are interested in by exhibiting its Fourier series:
\begin{equation}\label{eq_log_sin_def}
\log \sin z = \log \frac {\eee^{-\ii z}(1 - \eee^{2\ii z})}{-2\ii}
:=
- \ii z - \log 2 + \frac {\pi \ii}{2} - \sum_{\ell=1}^\infty\frac{\eee^{2 \ii \ell z}}{\ell}.
\end{equation}
Note that the series converges uniformly as long as $\Im z \geq  c$ for some $c>0$.

%If $\tau\to+\infty$, then %uniformly in  $\sigma \in \R$ it holds that
%$$
%\sin \frac z2 = \frac {\eee^{\frac{\ii \sigma - \tau}{2}} - \eee^{-\frac{\ii \sigma - %\tau}{2}}}{2\ii}
%= \frac{\ii}{2} \eee^{(\tau -\ii \sigma)/2} (1+o(1)).
%$$

%\begin{remark}
%Other than the classical Laguerre polynomials, their circular analogues $L_{n,k}$ are not well-defined  for non-integer $k$ because %of the term $(j-\frac n2)^k$ appearing in~\eqref{eq:L_2n_def}. On the other hand, for $k\in \N_0$ and non-integer $n>0$ it is %possible to use~\eqref{eq:cauchy_for_der} as an alternative definition.
%\end{remark}

\medskip
\noindent
\textit{Saddle point equation}.
We are looking for critical (saddle) points of $f_{s}(z;\theta)$ in the upper half-plane $\bH$. The partial derivative of $\log f_{s}(z;\theta)$ in $z$ is given by
$$
\frac {\partial}{\partial z} \log f_{s}(z;\theta)
=
%=
%\frac{\exp\left\{\frac {1}{s} \log \sin z\right\}}{(z-\theta)^2} \left(\frac {1}{s} \cot z\cdot  (z-\theta)  - 1\right)
%=
%f_{s}(z;\theta) \cdot
\frac {1}{s} \cot z - \frac {1}{z-\theta}.
$$
Hence, the saddle points are the solutions $\zeta\in \bH$ of the equation
\begin{equation}\label{eq:saddle_point_eq}
%\frac 12 \cot \frac z2 = \frac {k/n}{z-\theta}, \qquad z\neq \theta.
\zeta - s \tan \zeta = \theta.
\end{equation}
It will be shown in Section~\ref{sec:principal_branch} that for every $\theta\in \bH$ this equation has a \textit{unique, simple} solution $\zeta=\zeta_s(\theta)$ in $\bH$. Moreover, we shall show that $\zeta_s(\theta) - \theta \to \ii s$ as $\Im \theta \to +\infty$.
If not otherwise stated, this and similar claims hold uniformly in $\Re \theta\in \R$.

\medskip
\noindent
\textit{Existence of the saddle-point contour}.
To apply the saddle-point method, we need to deform the contour of integration in~\eqref{eq:cauchy_for_der} to  some closed $C^\infty$-contour $\gamma_{s,\theta}$  that depends on the parameters $s$ and $\theta$,  passes through the saddle point $\zeta_{s}(\theta)$ and has the property that
\begin{equation}\label{eq:saddle_point_contour}
|f_s(\zeta_{s}(\theta); \theta)| > |f_s(z; \theta)|
\text{ for all }
z\in \gamma_{s,\theta}\backslash\{\zeta_{s}(\theta)\}.
\end{equation}
%Write $\theta = \sigma + \ii \tau$ .
%If $\sigma\to+\infty$, then uniformly in  $\tau \in \R$ it holds that
%\begin{equation}\label{eq:cotan_asympt}
%\cot \frac z2
%=
%\ii\, \frac{\eee^{\frac{\ii \sigma - \tau}{2}} + \eee^{-\frac{\ii \sigma - %\tau}{2}}}{\eee^{\frac{\ii \sigma - \tau}{2}} - \eee^{-\frac{\ii \sigma - \tau}{2}}}
%=
% -\ii \, \frac {1 + \eee^{-\tau + \ii \sigma}}{1 - \eee^{-\tau + \ii \sigma}} = -\ii %+o(1).
%\end{equation}
%Write $\theta = \sigma + \ii \tau$ with $\sigma,\tau\in \R$.
In the following we shall argue that a saddle-point contour satisfying~\eqref{eq:saddle_point_contour} exists if $\Im \theta$ is sufficiently large. Consider the level lines of the function $z\mapsto |f_s(z;\theta)|$, defined on the half-plane $\{z\in \C: \Im z\geq 1\}$.  Near the pole at $z=\theta$, the level lines are small loops enclosing $\theta$. Let us now decrease the level. According to Morse theory, the topology of the level lines can change only if they pass through some critical point or if the level lines cross the boundary $\Im z = 1$.  We know that the \textit{unique, non-degenerate} critical point in the upper half-plane $\bH$ is $\zeta_s(\theta)$.  We shall prove in a moment that for sufficiently large $\Im \theta$,
\begin{equation}\label{eq:saddle_point_higher_than_reals}
|f_s(\zeta_s(\theta); \theta)| >  \sup_{z\in \C: \Im z = 1} |f_s(z;\theta)|.
\end{equation}
Assuming~\eqref{eq:saddle_point_higher_than_reals} it follows that $\zeta_s(\theta)$ is the \textit{first} critical point which appears on the level lines if we decrease the level. Since the critical point $\zeta_s(\theta)$ is simple, the level line passing through it has a $\gamma$-shaped topology and contains a loop enclosing $\theta$;  see Figure~\ref{fig:contour_map}. It follows that a saddle-point contour satisfying~\eqref{eq:saddle_point_contour} and passing through $\zeta_s(\theta)$ exists.
\begin{figure}[t]
\centering
\includegraphics[width=0.4\columnwidth]{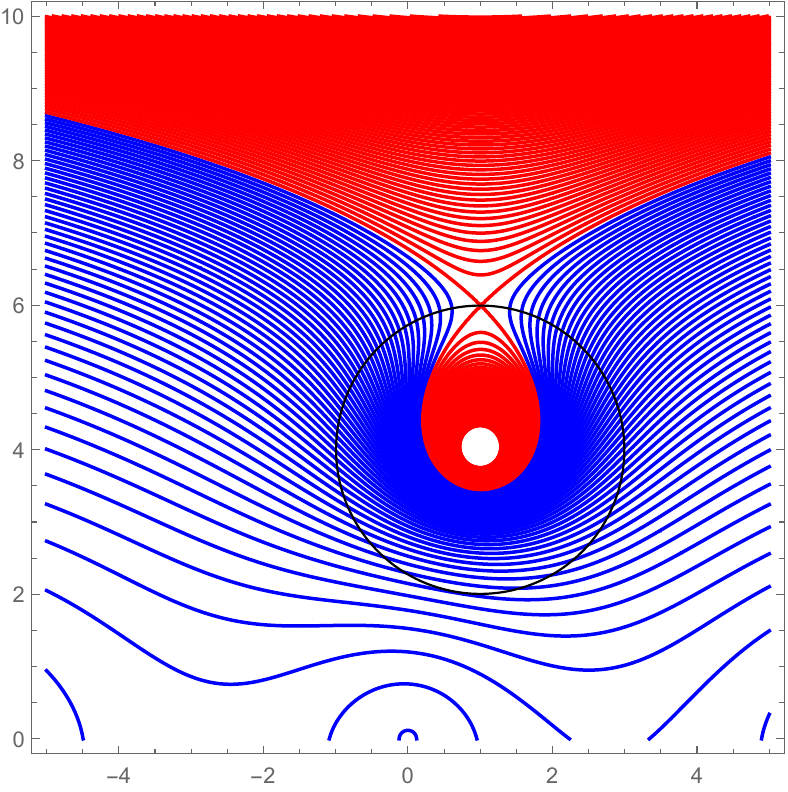}
\caption{Level lines of $z\mapsto |f_s(z;\theta)|$ for $s=1$ and $\theta = 1 + 4\ii$. A saddle-point contour is shown in black.}
\label{fig:contour_map}
\end{figure}

Let us prove~\eqref{eq:saddle_point_higher_than_reals}.  It follows from~\eqref{eq_log_sin_def} that
$\log \sin z$ differs from   $- \ii z$ by a function that stays bounded on $\{\Im z \geq 1\}$.
Therefore, uniformly over all $z\in \C$ with $\Im z = 1$ we have the estimate
\begin{equation}\label{eq:level_lines_est_1}
|f_s(z; \theta)|
=
\frac{\exp\left\{\frac 1s \Re \log \sin z\right\}}{|z - \theta|}
\leq
\frac{\exp\left\{O(1)\right\}}{\Im \theta - 1} \to 0,
\end{equation}
as $\Im \theta \to +\infty$.
On the other hand, recalling that $\zeta_s(\theta) - \theta \to \ii s$ as $\Im \theta \to +\infty$ we see that
\begin{equation}\label{eq:level_lines_est_2}
|f_s(\zeta_s(\theta); \theta)|
=
\frac{\exp\left\{\frac 1s \Re \log \sin \zeta_s(\theta)\right\}}{|\zeta_s(\theta) - \theta|}
=
\frac{\exp\{\frac 1s \Im \zeta_s(\theta) + O(1)\}}{|\zeta_s(\theta) - \theta|}
=
\frac{\exp\{\frac 1s \Im \theta + O(1)\}}{|\zeta_s(\theta) - \theta|}
\to
+\infty,
\end{equation}
as $\Im \theta \to +\infty$.   Taking~\eqref{eq:level_lines_est_1} and~\eqref{eq:level_lines_est_2} together we conclude that~\eqref{eq:saddle_point_higher_than_reals} holds provided $\Im \theta$ is sufficiently large.

\begin{remark}
We strongly believe that a saddle-point contour satisfying~\eqref{eq:saddle_point_contour} exists for \textit{all} $\theta\in \bH$. One way to prove this would be to show that the value of $|f_s(z; \theta)|$ at $z=\zeta_s(\theta)$ is larger than all values for $z\in \R\backslash \pi \Z$. However, numerical simulations suggest that the latter claim is false without the assumption that $\Im \theta$ is sufficiently large.
\end{remark}

\medskip
\noindent
\textit{The saddle-point asymptotics}.
Let $s>0$ and $\Im \theta >\tau_0$ with sufficiently large $\tau_0>0$.  We have verified the following conditions:
\begin{itemize}
\item $\zeta_s(\theta)$ is a simple saddle point of $f_s(z;\theta)$, that is $\frac{\partial}{\partial z} f_s(z;\theta)|_{z= \zeta_s(\theta)} = 0$, $\frac{\partial^2}{\partial z^2} f_s(z;\theta)|_{z= \zeta_s(\theta)} \neq 0$;
\item $\gamma_{s,\theta}$ is a closed $C^\infty$-contour passing through $\zeta_s(\theta)$;
\item for all $z\in \gamma_{s,\theta}\backslash\{\zeta_{s}(\theta)\}$ we have  $|f_s(\zeta_{s}(\theta); \theta)| > |f_s(z; \theta)|$; see~\eqref{eq:saddle_point_contour}.
\end{itemize}
Under these conditions, the saddle-point method applies; see \S~45.4 in~\cite{sidorov_fedoryuk_shabunin_book}. In particular, Equation~(45.7) in~\cite{sidorov_fedoryuk_shabunin_book} gives
\begin{equation*}
\frac 1 {2\pi \ii} \oint_{\gamma_{s,\theta}} (f_{s}(z; \theta))^k \dint z
\sim
\frac {(f_s(\zeta_s(\theta);\theta))^k}{\sqrt{2\pi k \cdot (\log f_{s})'' (\zeta_s(\theta);\theta)}}
\qquad
\text{ as }
k\to\infty.
%&=
%\frac{1}{\sqrt{2\pi k}}  \left(\frac{1-s}{(\zeta_s(\theta) - \theta)^2} - \frac 1s\right)^{-1/2}
%\frac{\eee^{\frac ks \log \sin \zeta_s(\theta)}}{(\zeta_s(\theta) - \theta)^k},
\end{equation*}
The choice of the branch of the square root is explained in~\cite[\S~45.4]{sidorov_fedoryuk_shabunin_book} and will be irrelevant for our purposes.  Recalling~\eqref{eq:f_n_def} we obtain
\begin{align}
\lim_{k\to\infty}
\frac{\sqrt{2\pi k} \cdot \frac 1 {2\pi \ii} \oint_{\gamma_{s,\theta}} (f_{s}(z; \theta))^k \dint z}{\eee^{\frac ks \log \sin \zeta_s(\theta)}/(\zeta_s(\theta) - \theta)^{k}}
=
\left(\frac{1-s}{(\zeta_s(\theta) - \theta)^2} - \frac 1s\right)^{-1/2}.
\label{eq:saddle_point_asymptotics}
\end{align}
Note that the terms appearing in this formula cannot become $0$ or $\infty$ since $\Im \zeta_t(\theta)>\Im \theta$ for all $\theta\in \bH$ by Lemma~\ref{lem:zeta_nevanlinna} and since the critical point at $\zeta_t(\theta)$ is non-degenerate.
In fact, the arguments of~\cite[\S~45.4]{sidorov_fedoryuk_shabunin_book}  show that~\eqref{eq:saddle_point_asymptotics} holds uniformly as long as $(\theta, s)$ stays in a compact subset of $\{\Im \theta >\tau_0\}\times (0,\infty)$.
%Hence,
%\begin{multline*}
%\lim_{k\to\infty}
%\left(\log \sqrt{2\pi k} + \log\left|\frac 1 {2\pi \ii} \oint (f_{s}(z; \theta))^k \dint z\right|
%
%\frac{k}s\log |\sin \zeta_s(\theta)| + k \log |\zeta_s(\theta) - \theta|\right)\\
%\frac 12 \log \left(\frac{1-s}{(\zeta_s(\theta) - \theta)^2} - \frac 1s\right).
%\end{multline*}
%More precisely, the branches of the logarithms are defined such that the above relation holds at some fixed point and such that all functions of the form $\log (\ldots)$ appearing above are continuous.
Applying the function $w\mapsto \frac 1k \log |w|$ to both sides of~\eqref{eq:saddle_point_asymptotics} and removing terms going to $0$ yields
$$
\lim_{k\to\infty} \frac 1k \log\left|\frac 1 {2\pi \ii} \oint_{\gamma_{s,\theta}} (f_{s}(z; \theta))^k \dint z\right|
=
\frac{1}s\log |\sin \zeta_s(\theta)| - \log |\zeta_s(\theta) - \theta|
$$
locally uniformly in $(\theta,s)\in \{\Im \theta >\tau_0\}\times (0,\infty)$.
Recalling the formula for $V_{n,k}(\theta)$ stated in~\eqref{eq:cauchy_for_der} and the relation $k/n \to t$, we arrive at
\begin{equation}\label{eq:V_n_k_asymptotics}
\lim_{n\to\infty} \frac 1n \log |V_{n,k}(\theta)|
=
\log |\sin \zeta_t(\theta)| - t \log |\zeta_t(\theta) - \theta|.
\end{equation}
The above results can be summarized in the following lemma which proves Theorem~\ref{theo:limit_log_der} for sufficiently large $\Im \theta$.
\begin{lemma}\label{lem:log_der_T_n_asympt}
Let $t>0$ be fixed and recall that $k/n \to t$ as $n\to\infty$. If $\tau_0>0$ is sufficiently large, then locally uniformly on $\{ \theta\in \C: \Im \theta > \tau_0\}$ we have
%For all $\theta\in \C$ with $\Im \theta \geq  \tau_0$ and sufficiently large $\tau_0$ we have, with a suitable choice of the branches of the logarithms,
\begin{align}
\lim_{n\to\infty} \frac 1n \log |W_{n,k}(\theta)|
&=
\log |\sin \zeta_t(\theta/2)| - t \log |2\zeta_t(\theta/2) - \theta|,\label{eq:asympt_W_nk_local_1}\\
\lim_{n\to\infty} \frac 1{n} \frac{W_{n,k}'(\theta)}{W_{n,k}(\theta)}
&=
\frac 12 \cot \zeta_t(\theta/2)
=
\frac{t}{2 \zeta_t(\theta/2) - \theta}.  \label{eq:asympt_W_nk_local_2}
\end{align}
\end{lemma}
\begin{proof}
The first formula is just a restatement of~\eqref{eq:V_n_k_asymptotics} using~\eqref{eq:V_W}.
To prove the second formula, consider the sequence of analytic functions
\begin{equation}\label{eq:h_n}
h_n (\theta) := \frac 1n \log W_{n,k}(\theta) - \log \sin \zeta_t(\theta/2) + t \log (2\zeta_t(\theta/2) - \theta),
\qquad
\theta\in \bH,
\end{equation}
where the logarithms are chosen such that all functions of the form $\log (\ldots)$ are continuous (and analytic) on $\bH$. From~\eqref{eq:asympt_W_nk_local_1} we know that $\Re h_n(\theta) \to 0$ locally uniformly on $\{\Im \theta > \tau_0\}$, and it follows from Lemma~\ref{lem:differentiate_limit_real_parts} below that $h_n'(\theta) \to 0$ locally uniformly. To complete the proof of~\eqref{eq:asympt_W_nk_local_2} note that, after some calculus,
\begin{equation}\label{eq:h_n_prime}
h_n'(\theta)
=
\frac 1{n} \frac{W_{n,k}'(\theta)}{W_{n,k}(\theta)} - \frac{t}{2 \zeta_t(\theta/2) - \theta}.
\end{equation}
The terms involving $\zeta_t'(\theta/2)$ cancel because $\zeta_t(\theta/2) - \theta/2  = t \tan \zeta_t(\theta/2)$.
\end{proof}

%%%%%%%%%%%%%%%%Old version of the lemma, it is weaker.
\begin{lemma}\label{lem:differentiate_limit_real_parts}
Let $h_1(z), h_2(z),\ldots$ be a sequence of analytic functions defined on some domain $\mathcal D\subseteq \C$. If $\Re h_n(z) \to
0$ locally uniformly on $\cD$,  then $h_n'(z) \to 0$ locally uniformly on $\mathcal D$.
\end{lemma}
\begin{proof}
The claim follows from the integral representation $h_n'(z) = \frac 1 {\pi \ii} \oint \frac{\Re h_n(w)}{(w-z)^2} \dint w$, where the integration is over a sufficiently small counter-clockwise  circle centered at $z$. To prove this formula, we may assume that $z=0$. It suffices to verify the representation for individual terms in the Taylor series of $h_n(w)$. Thus, we need to check that
$\frac 1 {\pi \ii} \oint \Re (a w^\ell) w^{-2} \dint w = a \ind_{\ell = 1}$ for all  $\ell\in \N_0$, $a\in \C$. Substituting $w= r\eee^{\ii \phi}$ with   $\phi\in [-\pi ,+\pi]$, the formula becomes $\frac 1 {\pi} \int_{-\pi}^{+\pi} \Re (a \eee^{\ii \ell \phi}) \eee^{-\ii \phi} \dint \phi = a \ind_{\ell = 1}$, which is elementary to verify.
\end{proof}

%We are now ready to prove that Theorem~\ref{theo:limit_log_der} holds for sufficiently large $\Im \theta$.
%\begin{lemma}\label{lem:log_der_T_n_asympt}
%For all $\theta\in \C$ with sufficiently large $\tau = \Im \theta > \tau_0$ we have
%$$
%\lim_{n\to\infty} \frac 1{n} \frac{W_{n,k}'(\theta)}{W_{n,k}(\theta)}
%=
%\frac 12 \cot \zeta_t(\theta/2)
%=
%\frac{t}{2 \zeta_t(\theta/2) - \theta}.
%$$
%%where $z_{\infty}(\theta)$ {eq:saddle_point_eq_limit}
%\end{lemma}
%\begin{proof}
%
%On the level of logarithmic asymptotics, we have
%$$
%
%\lim_{n\to\infty} \frac 1{2n} \log T_n(\theta)
%= \log \sin \frac{z_t(\theta)}{2} - t \log (z_{t}(\theta) - \theta)
%=:
%s(z_t(\theta); \theta).
%$$
%The above arguments show, in fact, that this convergence is a uniform convergence of analytic functions. Thus, interchanging the %derivative in $\theta$ and the limit in $n$ is justified. Denoting by $\partial_j s$ the derivative of $s(\cdot, \cdot)$ in the %$j$-th argument, $j\in \{1,2\}$, we have
%\begin{align*}
%\frac{\dint}{\dint \theta} s(z_t(\theta); \theta)
%&=
%(\partial_1 s)(z_t(\theta); \theta) z_t'(\theta) + (\partial_2 s)(z_t(\theta); \theta)
%=
%(\partial_2 s)(z_t(\theta); \theta)\\
%&=
%\frac{t}{z_t(\theta) - \theta}
%=
%\frac 12 \cot \frac{z_t(\theta)}{2},
%\end{align*}
%where we used that $(\partial_1 s)(z_t(\theta); \theta)=0$ because $z_t(\theta)$ is a solution %to~\eqref{eq:saddle_point_eq_limit}.
%\end{proof}

\begin{proof}[Completing the proof of Theorem~\ref{theo:limit_log_der}]
%In Lemma~\ref{lem:log_der_T_n_asympt} we already established the required convergence for for sufficiently large $\Im \theta$. It remains to extend it to the whole upper half-plane $\bH$.
Let us prove that~\eqref{eq:asympt_W_nk_local_2} holds locally uniformly on the whole of $\bH$. To this end we observe that in the above proof of Theorem~\ref{theo:zeroes_laguerre_restatement} we used the already established Lemma~\ref{lem:log_der_T_n_asympt} (rather than the full Theorem~\ref{theo:limit_log_der}). So, we know that $\mu_{n,k} \to \Pi_t$ weakly on $\bT$, as $n\to\infty$,
meaning that $\int_\bT F(u) \mu_{n,k} (\dint u) \to \int_\bT F(u) \Pi_t (\dint u)$ for every continuous function $F$ on $\bT$.
From the first equality in~\eqref{eq:psi_transf_def_rep} it follows that $\psi_{\mu_{n,k}}(z) \to \psi_{\Pi_t}(z)$ for \textit{all} $z\in \bD$, which extends the already established convergence~\eqref{eq:psi_transf_conv_unif}. In fact, the convergence is locally uniform on $\bD$, which can be verified using that the family $(g_z)_{z\in K}$ of functions $g_z(u) = \frac{uz}{1-uz}$, $u\in \bT$, is compact in $C(\bT)$ for every compact set $K\subseteq \bD$.  Applying Lemma~\ref{lem:psi_transf_emp_zeroes} and then Lemma~\ref{lem:free_poi_psi_transf} we conclude that
$$
\frac 1 {n} \frac{W_{n,k}'(\theta)}{W_{n,k}(\theta)}
=
-\ii \psi_{\mu_{n,k}} (\eee^{\ii \theta})  - \frac \ii 2
\ton
-\ii \psi_{\Pi_t}(\eee^{\ii \theta}) - \frac \ii 2
=
\frac 12 \cot \zeta_t(\theta/2)
=
\frac{t}{2 \zeta_t(\theta/2) - \theta}
$$
for \textit{all} $\theta \in \bH$, and the convergence is locally uniform. This establishes~\eqref{eq:asympt_W_nk_local_2}.
To prove  that~\eqref{eq:asympt_W_nk_local_1} holds locally uniformly on $\bH$, we take some $\theta_0\in \bH$ with sufficiently large $\Im \theta_0$ and write
$$
\frac 1n \log |W_{n,k}(\theta)| - \frac 1n \log |W_{n,k}(\theta_0)|  = \frac 1n \Re \left(\log W_{n,k}(\theta) - \log W_{n,k}(\theta_0)\right) =  \Re \int_{\theta_0}^\theta \frac 1{n} \frac{W_{n,k}'(x)}{W_{n,k}(x)}\dd x.
$$
Now we let $n\to\infty$.  We can apply~\eqref{eq:asympt_W_nk_local_1} to $\frac 1n \log |W_{n,k}(\theta_0)|$. Regarding the integral on the right-hand side, we can apply the just established relation~\eqref{eq:asympt_W_nk_local_2} to obtain
$$
\int_{\theta_0}^\theta \frac 1{n} \frac{W_{n,k}'(x)}{W_{n,k}(x)}\dd x \ton \int_{\theta_0}^\theta \frac{t\, \dint x}{2 \zeta_t(x/2) - x}
=
\left(\log \sin \zeta_t(x/2) -  t \log (2\zeta_t(x/2) - x)\right) \big |_{x=\theta_0}^{x=\theta},
$$
see~\eqref{eq:h_n} and~\eqref{eq:h_n_prime} for the last equality. Taking the real part yields~\eqref{eq:asympt_W_nk_local_1}.
\end{proof}

\begin{proof}[Proof of Theorem~\ref{theo:L_n_k_asymptotics}]
Recall from~\eqref{eq:W_n_def} that $L_{n,k}(\eee^{2\ii \theta}) = (2\ii)^n k! \eee^{\ii n \theta} W_{n,k}(2\theta)$.  Taking into account~\eqref{eq:W_log_der_conv} we obtain
$$
\frac 1n \frac{L_{n,k}'(\eee^{2\ii \theta})}{L_{n,k}(\eee^{2\ii \theta})}
=
\eee^{-2 \ii \theta} \left(\frac 12  - \frac \ii {n} \frac{W_{n,k}'(2\theta)}{W_{n,k}(2\theta)}\right)
\ton
\frac {1 - \ii \cot \zeta_t(\theta)} {2 \eee^{2\ii \theta}}.
$$
This proves~\eqref{eq:asympt_L_n_k_2}. To prove~\eqref{eq:asympt_L_n_k_1} we consider
$$
H_n(z) :=
\frac 1n \log \left(\frac{L_{n,k}(\eee^{2\ii \theta})}{L_{n,k}(0)}\right) +\ii \left(\zeta_t(\theta) - \theta - \ii t\right) - \log (1- \eee^{2\ii \zeta_t(\theta)}) + t \log\left(\frac{\zeta_t(\theta) - \theta}{\ii t}\right),
\quad \theta \in \bH,
$$
as an analytic function of the argument $z = \eee^{2\ii \theta} \in \bD\backslash\{1\}$. Observe that  $H_n(z)$ becomes analytic on $\bD$ if we put $H_n(0):=0$. After some calculus, one gets
$$
\frac{\dint }{\dint z} H_n(z) = \frac 1n \frac{L_{n,k}'(\eee^{2\ii \theta})}{L_{n,k}(\eee^{2\ii \theta})} - \frac{1 - \ii \cot \zeta_t(\theta)}{2 \eee^{2\ii \theta}}.
$$
Note that the terms involving $\zeta_t'(\theta)$ cancel.  We already know that, as $n\to\infty$, the right-hand side converges to $0$ locally uniformly in $\theta\in \bH$ or, which is the same, $z\in \bD \backslash\{0\}$. However, we are dealing with analytic functions and Cauchy's formula implies that the convergence is locally uniform in $z\in \bD$. Integrating and taking into account that $H_n(0) =0$ proves that $H_n(z) \to 0$ locally uniformly in $z\in \bD$, thus establishing~\eqref{eq:asympt_L_n_k_1}.
\end{proof}

\section{Analysis of the equation \texorpdfstring{$\zeta - t \tan \zeta = \theta$}{zeta - t * tan zeta = theta}}
\label{sec:principal_branch}
%Proof of Theorem~\ref{theo:riemann_surface_z_i}:

\subsection{Construction of the principal branch}\label{subsec:principal_branch}
Fix some $t>0$.  In what follows, we study the multivalued function $\zeta$,  defined implicitly as a function of the complex variable $\theta$ by the equation
\begin{equation}\label{eq:implicit_eqn_tan}
\zeta - t \tan \zeta = \theta.
\end{equation}
Our aim is to construct the \textit{principal branch} $\zeta_t(\theta)$ of $\zeta$.   This branch is an analytic function uniquely characterized by the properties stated in the following theorem.
%analytic on the upper half-plane $\bH$, continuous on its closure $\bar \bH$,  and satisfies the following boundary condition:
%\begin{equation}\label{eq:implicit_boundary_condition}
%\zeta_t (\theta) - \theta \text{ is bounded on } \{\theta \in \C: \Im \theta >\tau_0\} \text{ for some } \tau_0>0.
%\end{equation}
%%%%%%%%%%The function $z_t(\theta)$ appearing in Theorem~\ref{theo:riemann_surface_z_i} is related to $\zeta_t(\theta)$ by $z_t(\theta) = 2 \zeta_t(\theta/2)$.
\begin{theorem}\label{theo:principal_branch_existence_properties}
Let $t>0$.  There exists a unique function $\theta \mapsto \zeta_t(\theta)$ with the following $3$ properties:
\begin{itemize}
\item[(i)] $\zeta_t$ is analytic on the upper half-plane $\bH = \{\theta\in \C: \Im \theta > 0\}$;
\item[(ii)] $\zeta= \zeta_t$ is a solution to~\eqref{eq:implicit_eqn_tan}, more precisely, $\zeta_t(\theta) - t \tan \zeta_t(\theta) = \theta $ for all $\theta \in \bH$;
\item[(iii)] $\zeta_t$ satisfies the boundary condition
\begin{equation}\label{eq:implicit_boundary_condition}
\zeta_t (\theta) - \theta \text{ is bounded on } \{\theta \in \C: \Im \theta >\tau_0\} \text{ for some } \tau_0>0.
\end{equation}
\end{itemize}
Additionally, the function $\zeta_t$ satisfies $\zeta_t(\bH)\subseteq \bH$.  For all $\theta\in \bH$,  we have $\zeta_t(\theta + \pi ) = \zeta_t(\theta) + \pi$ and $\Im \zeta_t(\theta) > \Im \theta$. Also, $\zeta_t$ admits a continuous extension to the closed upper half-plane $\bar \bH$.   Finally, we have $\zeta_{t} (\theta) -\theta \to  \ii t$ as $\Im \theta\to +\infty$ (uniformly in $\Re \theta\in \R$).
\end{theorem}
In the rest of Section~\ref{subsec:principal_branch} we work toward proving Theorem~\ref{theo:principal_branch_existence_properties}.  

\medskip
\noindent
\textit{Existence for sufficiently large $\Im \theta$.}
%Fix some $t>0$. Let us show that an analytic function $z_t(\theta)$ with the properties (i) and (ii) listed in Theorem~\ref{theo:riemann_surface_z_i}
Let us construct the principal branch $\zeta_t(\theta)$  on a domain $\{\theta\in \C: \Im \theta > \tau_0\}$, where $\tau_0>0$ is sufficiently large. To this end, we consider the function
\begin{equation}\label{eq:construction_zeta_large_im_theta}
y(r):= \eee^{2tr} \frac {1-r}{1+r}.
\end{equation}
It is analytic on some neighborhood of $r=1$ and satisfies $y(1) = 0$ and $y'(1) = - \frac 12 \eee^{2 t} \neq 0$. By the inverse function theorem,  there is a well-defined inverse function $r=r_t(y)$ which is analytic on a small neighborhood of $y=0$. Let us put
$\zeta_t(\theta) := \theta + \ii t \cdot r_t(\eee^{2 \ii \theta})$, which is well-defined and analytic provided $\Im \theta$ is sufficiently large.  With this definition, we have
$$
\tan \zeta_t(\theta)
=
-\ii \cdot \frac {1 - \eee^{-2\ii \zeta_t(\theta)}}{1 + \eee^{- 2 \ii \zeta_t(\theta)}}
=
-\ii \cdot \frac {1 - \eee^{-2\ii \theta} \eee^{2t r_t(\eee^{2\ii \theta})}}{1 + \eee^{-2\ii \theta} \eee^{2t r_t(\eee^{2\ii \theta})}}
=
-\ii \cdot \frac {1 - \frac{1 + r_t(\eee^{2\ii \theta})}{1 - r_t(\eee^{2\ii \theta})}}{1 + \frac{1 + r_t(\eee^{\ii \theta})}{1 - r_t(\eee^{\ii \theta})}}
=
\ii r_t(\eee^{2\ii \theta})
=
\frac {\zeta_t(\theta) - \theta}{t},
$$
meaning that $\zeta_t(\theta)$ solves~\eqref{eq:implicit_eqn_tan}. It follows from $r_t(0) =1$ that~\eqref{eq:implicit_boundary_condition} holds and
\begin{equation}\label{eq:principal_branch_limit1}
\zeta_t(\theta) - \theta \to \ii t
\quad\text{ as }\quad
\Im \theta \to +\infty
\quad (\text{uniformly in } \Re \theta\in \R).
\end{equation}
%we have $z_t(\theta) - \theta \to 2 \ii t$ as $\Im \theta \to +\infty$ (uniformly in $\Re \theta\in \R$).

\medskip
\noindent
\textit{Uniqueness.} We prove that an analytic function $\zeta_t(\theta)$ satisfying~\eqref{eq:implicit_eqn_tan}  and~\eqref{eq:implicit_boundary_condition} is unique. Write $\zeta_t(\theta) = \theta + \ii t \cdot s_t(\theta)$ for some analytic function $s_t(\theta)$. By~\eqref{eq:implicit_boundary_condition} we have $|s_t(\theta)| \leq B$ for some constant $B>0$ and all $\theta \in \C$ with $\Im \theta > \tau_0$. After some transformations, Equation~\eqref{eq:implicit_eqn_tan} takes the form
$$
G_\theta(s_t(\theta)) = 0,
\text{ where }
G_\theta(s) := s  +  \frac{\eee^{2\ii \theta} -  \eee^{2ts}}{\eee^{2\ii \theta} +  \eee^{2t s}}.
$$
Note that for sufficiently large $\Im \theta$ the function $s\mapsto G_\theta(s)$ is analytic on $\{|s|\leq 2B\}$ and that $G_\theta (s) \to s-1$ uniformly on that disc as $\Im \theta \to +\infty$. By Rouch\'e's theorem, the solution of $G_\theta(s)=0$ is unique if $\Im \theta$ is sufficiently large,  proving that $s_t(\theta)$ (and hence $\zeta_t(\theta)$) is uniquely defined if $\Im \theta$ is sufficiently large. The rest follows by the principle of analytic continuation.
Note in passing that the uniqueness just proved implies that $\zeta_t(\theta + \pi ) = \zeta_t(\theta) + \pi$ and $\zeta_t(-\bar \theta) = -\overline{\zeta_t(\theta)}$.
%Proof of $\Im \zeta_t(\theta)>0$

\medskip
\noindent
\textit{Behavior on the imaginary half-axis.}  So far we defined the principal branch  $\zeta_t(\theta)$ for sufficiently large $\Im \theta$. The next step will be done in the following lemma.
\begin{lemma}
The principal branch  $\zeta_t(\theta)$ can be extended to an analytic function on some open set of $\theta$'s containing the imaginary half-axis $\{\ii \tau: \tau >0\}$.
\end{lemma}

Let $\theta = \ii \tau$ with $\tau>0$ and $\zeta = \ii y$ with $y>0$. Equation~\eqref{eq:implicit_eqn_tan} takes the form
\begin{equation}\label{eq:implicit_eqn_tan_imaginary}
y - t \tanh y = \tau.
\end{equation}
Our claim is a consequence of the following
%Moreover, by the simplicity of the solution stated in Lemma~\ref{lem:solution_imaginary_axis} the principal branch $z_t(\theta)$ can be analytically extended to some open set containing

\begin{lemma}\label{lem:solution_imaginary_axis}
Fix $t>0$. Then, for every $\tau>0$, Equation~\eqref{eq:implicit_eqn_tan_imaginary} has a unique, simple solution $y= y_t(\tau)>\tau$. This solution continuously depends on $\tau>0$ and satisfies $y_t(\tau) - \tau \to t$ as $\tau \to +\infty$. Moreover, we have
\begin{equation}
\lim_{\tau \downarrow 0} y_t(\tau) =
\begin{cases}
0, & \text{ if }  0 < t \leq 1,\\
y_t(0), &\text{ if } t>1,
\end{cases}
\end{equation}
where $y_t(0)>0$ is the unique strictly positive zero of the convex function $y\mapsto y - t \tanh y$, $y\geq 0$.
\end{lemma}
\begin{figure}[t]
	\centering
	\includegraphics[width=0.32\columnwidth]{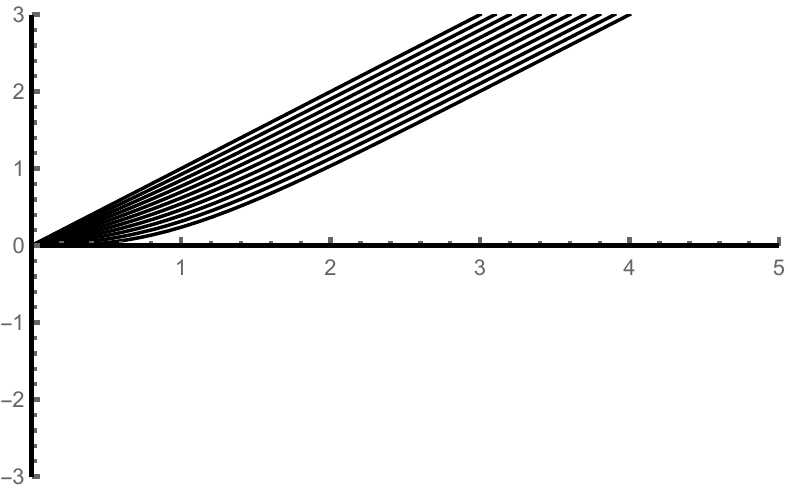}
	\includegraphics[width=0.32\columnwidth]{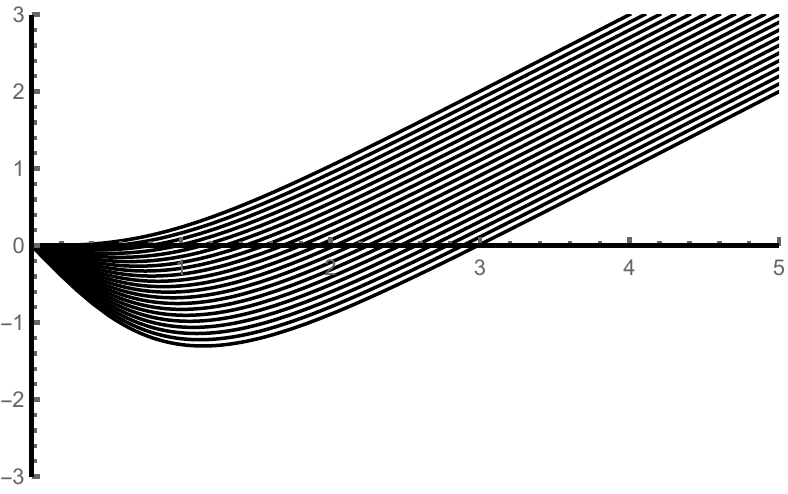}
	\includegraphics[width=0.32\columnwidth]{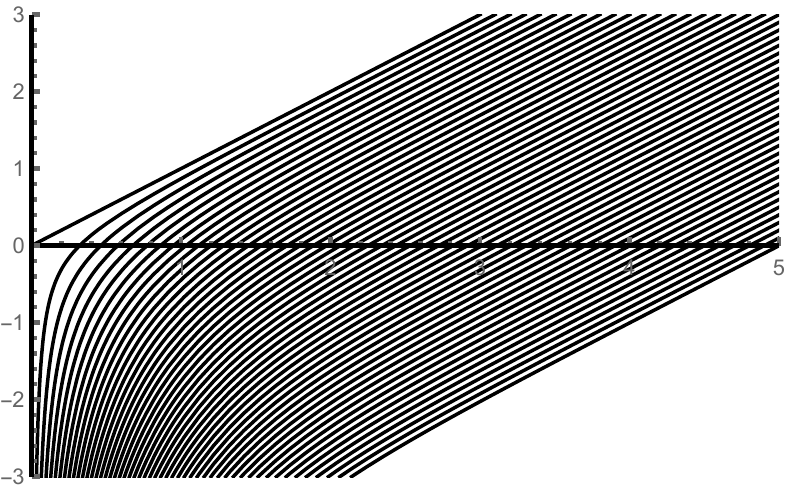}
	\caption{The first two panels show the graphs of the functions $y\mapsto  y - t \tanh y$, $y\geq 0$. Left: $0\leq t \leq 1$. Middle: $t>1$. The right panel shows the graphs of the functions $\tilde y \mapsto  \tilde y - t \cotanh \tilde y$, $\tilde y\geq 0$, for $0\leq t \leq 5$.}
\label{fig:y_minus_tanh}
\end{figure}
\begin{proof}
For every $t>0$, the function $g(y) = g_t(y) := y - t \tanh y$, $y\geq 0$, is strictly convex since its second derivative equals $2t (\sinh y) /(\cosh y)^3 > 0$. Moreover, $g(0)=0$, $g'(0)=1-t$ and $g(y) \to +\infty$ as $y\to +\infty$. It follows that for $0< t \leq 1$, the function $g$  strictly increases from $0$ to $+\infty$ on the interval $[0,+\infty)$, while for $t>1$ it has a unique positive zero $y_t(0)>0$ and is strictly increasing from $0$ to $+\infty$ on the interval $[y_t(0), +\infty)$; see the first two panels of Figure~\ref{fig:y_minus_tanh}. These properties imply the lemma.
\end{proof}

Since the solution mentioned in Lemma~\ref{lem:solution_imaginary_axis} is simple, it can be analytically continued to some open neighborhood of $\{\ii \tau: \tau >0\}$. By the arguments used in the proof of  uniqueness, we have $\zeta_t(\ii \tau) = \ii y_t(\tau)$ for all $\tau>0$. More generally, on the half-lines $\{\ii \tau + \pi n: \tau >0\}$, $n\in \Z$, we have $\zeta_t(\ii \tau + \pi n) = \ii y_t(\tau) + \pi n$.
Arguing in a similar way, it is possible to analyze the behavior of $\zeta_t(\theta)$ on the half-lines $\{\ii \tau + \frac \pi 2 + \pi n: \tau \geq 0\}$, $n\in \Z$, where we have $\zeta_t(\ii \tau + \frac \pi 2 + \pi n) = \ii \tilde y_t(\tau) + \frac \pi 2 + \pi n$ for all $\tau\geq 0$ with $\tilde y = \tilde y_t(\tau) >0$ being the unique solution to the equation $\tilde y - t \cotanh \tilde y = \tau$. This solution continuously depends on $\tau\geq 0$, satisfies $\tilde y_t(\tau) - \tau \to t$ as $\tau \to +\infty$, and we have $\lim_{\tau \downarrow 0} \tilde y_t(\tau) = \tilde y_t(0)>0$, where $\tilde y_t(0)>0$ is the unique  zero of the concave function $\tilde y\mapsto \tilde y - t \cotanh \tilde y$, $y\geq 0$; see the right panel of Figure~\ref{fig:y_minus_tanh}.

\medskip
\noindent
\textit{Analytic continuation aside the ramification points.}
In order to extend the function $\zeta_t$ to the whole of $\bH$, we need to  analyze the critical points and values of the meromorphic function $f_t(\zeta) := \zeta - t \tan \zeta$, defined for $\zeta\in \C\backslash \{\frac \pi 2 +\pi n: n \in \Z\}$. The \textit{critical points} of $f_t$ are the zeroes of its derivative, i.e.\ the solutions of the equation $\cos^2 \zeta = t$. The images of the critical points under $f_t$ are called the \textit{critical values} of $f_t$. Let $\mathcal D\subseteq \C$ be any simply-connected domain in the $\theta$-plane not containing any critical value of $f_t$. Take some $\theta_*\in \mathcal D$ and let $\zeta_*\in \C\backslash \{\frac \pi 2 +\pi n: n \in \Z\}$ be such that $f_t(\zeta_*) = \theta_*$. Standard arguments~\cite[Chapter~16]{rudin_book} show that there is a unique analytic function $\zeta_t(\theta)$, defined on the whole of $\cD$ and satisfying $f_t(\zeta_t(\theta)) = \theta$ for all $\theta\in \cD$ as well as $\zeta_t(\theta_*) = \zeta_*$. Indeed, the local invertibility of the function $f_t$ near its non-critical values is known and we need only to check that when continuing the function $\zeta_t$ along some path $\gamma :[0,1]\to \cD$ we cannot run into a singularity of $\zeta_t(\theta)$, that is the function cannot become unbounded. Assume, by contraposition, that $\zeta_t(\gamma(s)) \to \infty$ as $s\uparrow s_0$. Then, it follows from~\eqref{eq:implicit_eqn_tan} that $\tan \zeta_t(\gamma(s))\to \infty$ as $s\uparrow s_0$, hence the distance between $\zeta_t(\gamma(s))$ and the set $\frac \pi  2 + \pi \Z$ goes to $0$ as $s\uparrow s_0$, which is a contradiction since $\zeta_t(\gamma(s)) \to \infty$  as $s\uparrow s_0$ and at the same time $s\mapsto \zeta_t(\gamma(s))$ is continuous for $s<s_0$. Hence, we can analytically continue $\zeta_t(\theta)$ along any path in $\mathcal D$.

\medskip
\noindent
\textit{Critical points and values.}
Recall that the critical points of $\zeta \mapsto \zeta - t \tan \zeta$ are the solutions of the equation $\cos^2 \zeta = t$. We consider three cases.
%Assuming that $t\neq 1$ let us describe the ramification points of the Riemann surface $\mathcal S$, i.e.\ the pairs $(\theta, %\zeta)$ satisfying
%\begin{equation}\label{eq:ramification_points}
%\zeta - t \tanh \zeta = \theta,
%\qquad
%\cos^2 \zeta = t.
%\end{equation}
%

\medskip
\noindent
\textit{Case 1: $0 < t < 1$}. The critical points $\zeta_{n,\eps}$ and the corresponding critical values $\theta_{n,\eps}$ are given by
%are given by $\zeta =  \eps \arccos \sqrt t + \pi n$ for $n\in \Z$ and $\eps\in \{+1,-1\}$.
$$
\zeta_{n,\eps} =  \eps \arccos \sqrt t + \pi n,
\quad
\theta_{n,\eps} = \eps x_t + \pi n,
\quad
n\in \Z,
\;
\eps\in \{\pm 1\},
$$
where
$$
x_t:=\arccos\sqrt t - \sqrt{t(1-t)}\in \left(0, \frac \pi 2\right),
\qquad 0 < t <1.
$$
Note that all critical values are real and it is easy to check\footnote{If $\cos \varphi = \sqrt t$ for $\varphi\in (0, \frac \pi 2)$, then $2\arccos \sqrt t - 2\sqrt {t(1-t)} = 2\varphi - \sin (2\varphi) >0$.} that $-\frac \pi2 < \theta_{0,-1} < 0 < \theta_{0,+1}<\frac \pi 2$. It follows that the principal branch $\zeta=\zeta_t(\theta)$ which we already defined for sufficiently large $\Im \theta$ can be analytically extended to the complex plane with infinitely many vertical slits at $(\theta_{n,\eps} - \ii \infty, \theta_{n,\eps}]$, $n\in \Z$,
$\eps\in \{\pm 1\}$. This simply connected domain contains the upper half-plane $\bH$. In the next lemma we analyze the behavior of the function $\Im \zeta_t(\theta)$ for real $\theta$. Since it is periodic with period $\pi$, it suffices to consider the interval  $[-\frac \pi 2, +\frac \pi 2]$.
%It will become clear from the proof that $\zeta_t(\theta)$ can be extended to a continuous function on $\bar \bH$.

\begin{figure}[t]
	\centering
	\includegraphics[width=0.25\columnwidth]{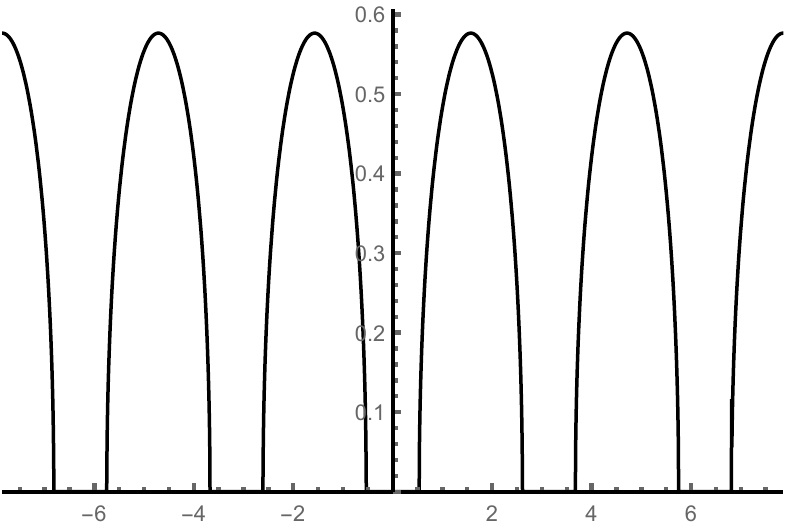}
     \hspace{0.05\columnwidth}
	\includegraphics[width=0.25\columnwidth]{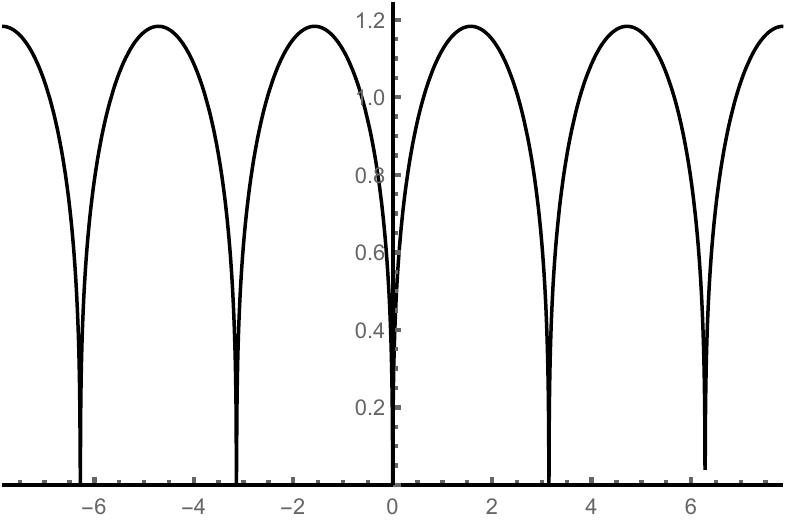}
     \hspace{0.05\columnwidth}
	\includegraphics[width=0.25\columnwidth]{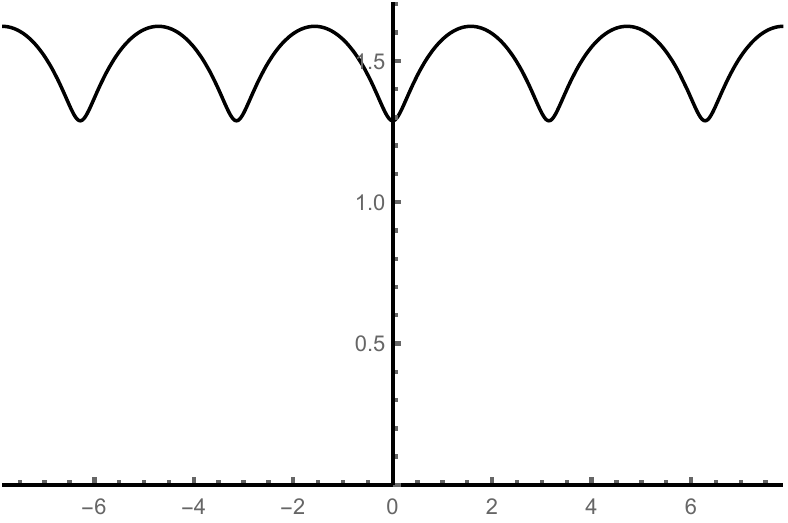}
	\caption{The graph of the function $\Im \zeta_t(\theta)$ for real $\theta$. Left: $t=0.3$. Middle: $t=1$. Right: $t=1.5$.}
\label{fig:zeta_on_real_axis}
\end{figure}

\begin{figure}[t]
	\centering
	\includegraphics[width=0.32\columnwidth]{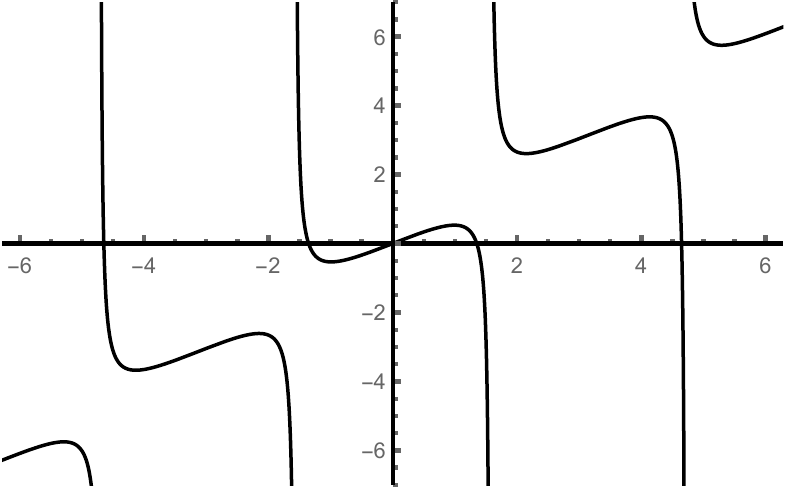}
	\includegraphics[width=0.32\columnwidth]{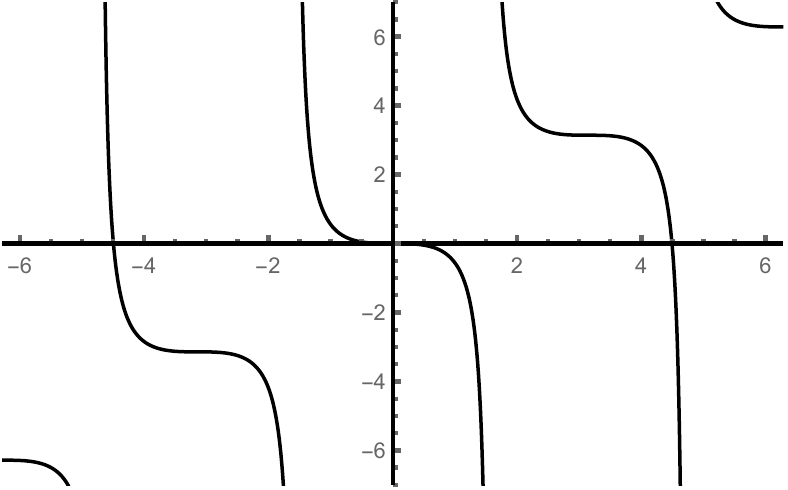}
	\includegraphics[width=0.32\columnwidth]{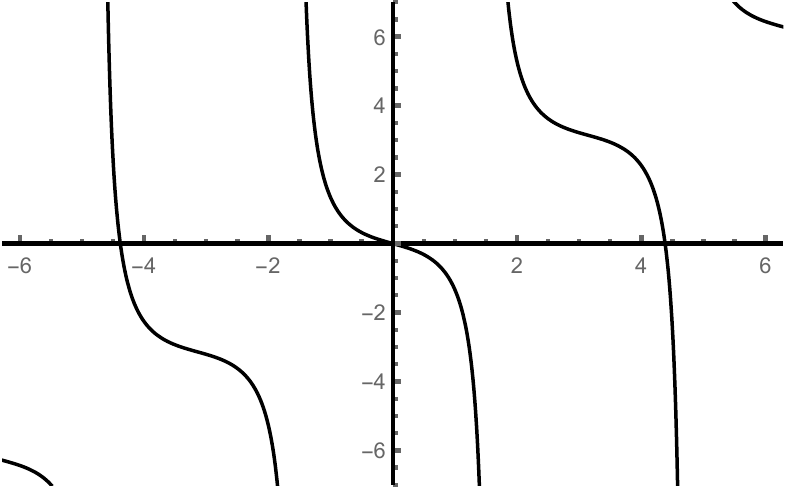}
	\caption{The graph of the function $\zeta\mapsto  \zeta - t \tan \zeta$, $\zeta\in \R$. Left: $t=0.3$. Middle: $t=1$. Right: $t=1.5$.}
\label{fig:on_the_real_axis}
\end{figure}

\begin{lemma}
Let $0<t<1$. Then, $\Im \zeta_t(\theta) = 0$ for $\theta \in [-x_t,x_t]$ and $\Im \zeta_t(\theta)>0$ for $\theta \in [-\frac \pi 2, +\frac \pi 2]\backslash[-x_t, x_t]$;  see the left panel of Figure~\ref{fig:zeta_on_real_axis}.
\end{lemma}
\begin{proof}
We have already seen in Lemma~\ref{lem:solution_imaginary_axis} that $\zeta_t(0) = 0$, while $\Im \zeta_t(\pm \pi/2) = \tilde y_t(0) > 0$.  The function $\zeta \mapsto \zeta - t\tan \zeta$ is an increasing diffeomorphism between the interval $[-\arccos \sqrt t, +\arccos \sqrt t]$ and its image $[-x_t, +x_t]$. Its derivative vanishes at $\pm \arccos \sqrt t$; see the left panel of Figure~\ref{fig:on_the_real_axis}.  Given the known value $\zeta_t(0)=0$, we conclude that the principal branch on the interval $[-x_t, x_t]$ coincides with the inverse of this diffeomorphism. In particular, $\Im \zeta_t(\theta) = 0$ for $\theta \in [-x_t,x_t]$. Also, $\zeta_t(\pm x_t) = \pm\arccos \sqrt t$.  Let us see what happens if $\theta \in (x_t, \frac \pi 2]$.  Although the equation $\zeta - t\tan \zeta = \theta$ has infinitely many real solutions, see the left panel of Figure~\ref{fig:on_the_real_axis}, these  solutions belong to the branches which cannot attain the value $\arccos \sqrt t$ as $\theta \downarrow x_t$. Therefore, the principal branch cannot attain real values on $(x_t, \frac \pi 2]$. From the known value $\Im \zeta_t(\pm \pi/2) = \tilde y_t(0) > 0$ we conclude that $\Im \zeta_t(\theta)>0$ for all $\theta \in (x_t, \frac \pi 2 ]$. The analysis of the interval $[-\frac \pi 2, -x_t)$ is similar.
\end{proof}
\begin{remark}%%%%%%%%%%% Numerically:OK
Let us look more carefully at the square-root singularity of $\zeta_t(\theta)$ at the branch point $x_t$ (the analysis of $-x_t$ being similar).
We have the expansion
$$
\zeta - t \tan \zeta = x_t - \sqrt{\frac{1-t}{t}} (\zeta - \arccos \sqrt t)^2 + o((\zeta - \arccos \sqrt t)^2),
\qquad
\text{ as }
\zeta \to \arccos \sqrt t.
$$
Consequently, the Puiseux expansion of the inverse function begins with
$$
\zeta_t(\theta) = \arccos \sqrt t - \sqrt[4]{\frac{t}{1-t}} (x_t - \theta)^{1/2} + o((x_t - \theta)^{1/2}),
\qquad
\text{ as }
\theta \to x_t.
$$
The branch of the square root appearing here is defined on $-\bar \bH$ and satisfies $\sqrt 1 = +1$ (because $\zeta_t(\theta)$ is real and increases for $\theta\uparrow x_t$) and, consequently, $\sqrt{-1} = -\ii$.
%Bypassing $x_t$ by a small half-loop starting at $x_t - \eps$ and ending at $x_t+\eps$, we can continue our function to
\end{remark}

\medskip
\noindent
\textit{Case 2: $t > 1$}.  The critical points $\zeta_{n,\eps}$ and the corresponding critical values $\theta_{n,\eps}$ are given by
$$
\zeta_{n,\eps} = -\ii \eps \arccosh \sqrt t + \pi n,
\quad
\theta_{n,\eps} = \ii \eps \left(\sqrt {t(t-1)} - \arccosh \sqrt t\right) + \pi n,
\quad
n\in \Z,
\;
\eps\in \{\pm 1\}.
$$
It is easy to check\footnote{If $\cosh \varphi = \sqrt t$ for $\varphi>0$, then $2\sqrt {t(t-1)} - 2\arccosh \sqrt t = \sinh (2\varphi)  - 2\varphi >0$.} that $\Im \theta_{n, +1} > 0$ while $\Im \theta_{n, -1} < 0$.
Although the critical values $\theta_{n, +1}$, $n\in \Z$,  lie in the upper half-plane $\bH$, we can argue that none of the points $(\theta_{n, +1}, \zeta_{n, +1})$ belongs to the principal sheet of the Riemann surface we are interested in. Let us to do it  for $n=0$. We have already identified the values of $\zeta_t(\theta)$ on the imaginary half-axis. In particular, we know that $\zeta_t(\theta_{0, +1}) = \ii y$ with some $y>0$. On the other hand, we have  $\Im \zeta_{0, +1} = - \arccosh \sqrt t <0$, meaning that $\zeta_t (\theta_{0, +1}) \neq \zeta_{0,+1}$. Therefore, the principal sheet of the Riemann surface avoids the branching points $(\theta_{n, +1}, \zeta_{n, +1})$ and it follows that  $\zeta_t(\theta)$ admits an analytic continuation to the complex plane with infinitely many vertical slits at $(\theta_{n, -1}-\ii \infty, \theta_{n,-1})$, $n\in \Z$. Note that this domain contains $\bH$ and, in fact, it contains the larger half-plane
$$
\left\{\theta\in \C: \Im \theta > \arccosh \sqrt t - \sqrt {t(t-1)}\right\}.
$$
\begin{lemma}
Let $t>1$. Then, $\Im \zeta_t(\theta) >0$ for all $\theta\in \R$; see the right panel of Figure~\ref{fig:zeta_on_real_axis}.
\end{lemma}
\begin{proof}
We have already seen in Lemma~\ref{lem:solution_imaginary_axis} that $\zeta_t(0) = \ii y_t(0)$ with  $y_t(0)>0$. Taking the derivative we see that  for $t>1$ the function $\zeta \mapsto \zeta - t \tan \zeta$ is strictly decreasing on every interval $(-\frac \pi 2 + \pi n, +\frac \pi 2 + \pi n)$, $n\in \Z$; see the right panel of Figure~\ref{fig:on_the_real_axis}. It follows that the equation $\zeta - t \tan \zeta = \theta$ has infinitely many simple  real solutions for all real $\theta$. Every such solution determines a branch of the inverse function which is different from the principal branch because at $\theta = 0$, these solutions are real, while at the same time $\Im \zeta_t(0) >0$.
It follows that the curve $\zeta_t(\R)$ cannot cross the real axis. Hence it stays in $\bH$.
\end{proof}

\medskip
\noindent
\textit{Case 3: $t = 1$}. The critical points are at $\pi n$, $n\in \Z$, and the corresponding critical values also equal $\pi n$. The principal branch is defined on the complex plane with slits at $(\pi n - \ii \infty, \pi n]$, $n\in \Z$. This domain contains the upper half-plane $\bH$. Defining $\zeta_1(\pi n) = \pi n$, $n\in \Z$, yields a continuous extension of the principal branch to the closed upper half-plane $\bar \bH$.
\begin{lemma}
Let $t=1$. Then, $\Im \zeta_1(\theta) > 0$ for all $\theta\in \R\backslash \pi \Z$ and $\Im \zeta_1(\pi n) = 0$ for all $n\in \Z$; see the middle panel of Figure~\ref{fig:zeta_on_real_axis}. We also have
\begin{equation}\label{eq:zeta_cubic_root_asympt}
\Im \zeta_1(\theta) \sim (\sqrt 3/2) \cdot   |3\theta|^{1/3},
\qquad
\text{ as }
\theta\to 0.
\end{equation}
\end{lemma}
\begin{proof}
%We have already seen that $\Im \zeta_1(\pi n) = 0$ for $n\in \Z$.
Let us analyze the behavior of the principal branch  near $0$. Since $\zeta - \tan \zeta = - \frac 13 \zeta^3 + O(\zeta^5)$ as $\zeta\to 0$,  the  inverse function has a Puiseux expansion in powers of $\theta^{1/3}$ as $\theta\to 0$. In fact, we have
\begin{equation}\label{eq:zeta_cubic_root_series}
\zeta_1(\theta) = (-3\theta)^{1/3} - \frac 2{15} (-3\theta) + \frac 3 {175} (-3\theta)^{5/3} - \frac 2 {1575} (-3\theta)^{7/3}
- \frac {16} {202125} (-3\theta)^3 +\ldots.
\end{equation}
Different choices of the cubic root correspond to three different branches of the inverse function. Let us identify the choice corresponding to the principal branch $\zeta_1(\theta)$. Since $\zeta \mapsto \zeta - \tan \zeta$ is decreasing from $+\infty$ to $-\infty$ on $(-\frac \pi 2, + \frac \pi 2)$ and the derivative is strictly negative for $\zeta\neq 0$, see the middle panel of Figure~\ref{fig:on_the_real_axis}, there exist two real branches, the first one taking real values for $\theta>0$, the second one being real for $\theta < 0$.  On the other hand, the principal branch $\zeta_1(\theta)$ is obtained by choosing the cubic root $\sqrt[3]{w}$, $w\in -\bar \bH$, such that $\sqrt[3]{-\ii} = \ii$ (to be conform with the fact that $\zeta_1(\ii \tau) = \ii y_1(\tau)$ with $y_1(\tau)>0$ for $\tau>0$; see Lemma~\ref{lem:solution_imaginary_axis}). Consequently, $\sqrt[3]{+1} = \eee^{2\pi \ii /3}$ and $\sqrt[3]{-1} = \eee^{\pi \ii /3}$ for the principal branch and it follows from the above expansion that $\Im \zeta_1(\theta) >0$ for real, sufficiently small $\theta$. Since $\zeta_1(\theta)$ cannot become real for $\theta \in (-\pi, \pi)$ (because the real solution belongs to one of the two real branches), we conclude that $\Im \zeta_1(\theta) >0$ for all $\theta \in (-\pi, \pi)$.
\end{proof}

\begin{lemma}\label{lem:zeta_nevanlinna}
For every $t>0$, $\zeta_t(\theta)$ is a Nevanlinna function, that is $\zeta_t(\bH)\subseteq \bH$, and, moreover, $\Im \zeta_t(\theta) > \Im \theta$ for all $\theta\in \bH$.
\end{lemma}
\begin{proof}
From the above analysis it follows that, in all three cases, $\Im \zeta_t(\theta) \geq 0$ for $\theta \in \R$.  The rest can be derived from  the maximum principle. Indeed,  the function $\zeta_t(\theta) - \theta$ is periodic with period $\pi$ and hence can be written as $\zeta_t(\theta) - \theta = p_t(\eee^{2 \ii \theta})$ for some function $p_t$ analytic on $\bD$ and continuous on $\bar \bD$. Since $\Im p_t(z)$ is non-negative on $\bT$, by the maximum principle we conclude that $\Im p_t(z) \geq 0$ for all $z\in \bD$, which proves that $\Im \zeta_t(\theta) \geq \Im \theta>0$ for all $\theta \in \bH$. To prove the strict inequality, observe that $\zeta_t(\theta) - \theta = t \tan \zeta_t(\theta)$ and the right hand-side has strictly positive imaginary part for all $\theta \in \bH$ since, as already shown, $\zeta_t(\theta) \in \bH$ and the tangent function maps $\bH$ to $\bH$.
%For every $\theta \in \bH$ we have $\zeta_t(\theta) - \theta = t \tan \zeta_t (\theta) \in \bH$ INCORRECT since $z\mapsto \tan z %= \ii \frac{1-\eee^{2\ii z}}{1+\eee^{2\ii z}}$ maps $\bH$ to $\bH$.
\end{proof}

Altogether, the results of Section~\ref{subsec:principal_branch} prove Theorem~\ref{theo:principal_branch_existence_properties}.

\subsection{Principal branch as limit of iterated functions}
The implicit equation we are interested in can be written as $\zeta = \theta + t \tan \zeta$. Iterating it, we obtain
$$
\zeta = \theta + t \tan \zeta = \theta + t \tan (\theta + t \tan \zeta) = \theta + t \tan (\theta + t \tan (\theta + t \tan \zeta)) =\ldots.
$$
In the next theorem we investigate whether iterations of this type converge.
\begin{theorem}\label{theo:iterations1}
Fix $t > 0$. Take some $\theta \in \C$ with $\Im \theta\geq 0$ and define the function $f_\theta(x):= \theta + t \tan x$ for  $x\in \bH$.  Then, the sequence $x, f_\theta(x), f_\theta(f_\theta(x)), \ldots, f_\theta^{\circ n}(x), \ldots$ of iterations of $f_\theta$ converges to $\zeta_t(\theta)$ (which does not depend on $x$) locally uniformly in $x\in \bH$.
\end{theorem}
\begin{proof}
The function $f_\theta$ is a holomorphic self-map of $\bH$ (because so is $z\mapsto \tan z = \ii\, \frac{1-\eee^{2\ii z}}{1+\eee^{2\ii z}}$). At the same time,  $f_\theta$ is not an automorphism of $\bH$ (and, in fact, not bijective) since $f_\theta(x+\pi)= f_\theta(x)$. By the Denjoy-Wolff theorem~\cite[Theorems~5.3, 5.4]{burckel}, the iterations $f_\theta^{\circ n}(x)$  converge, as $n\to\infty$, to some constant limit $c\in \bH$ or $c\in \R \cup\{\infty\}$ locally uniformly in $x\in \bH$. We need to show that $c= \zeta_t(\theta)$.

If $\theta\in \bH$, then we already know from Section~\ref{subsec:principal_branch} that $\zeta_t(\theta)\in \bH$ is a fixed point of $f_\theta$, which implies that $c = \zeta_t(\theta)$. Note in passing that we have shown that, for $\theta\in \bH$, $\zeta_t(\theta)$ is the \textit{unique} solution to $\zeta - t \tan \zeta = \theta$ satisfying $\zeta \in \bH$.

Consider now the case\footnote{The real $\theta$ case is more difficult but interesting because when combined with Proposition~\ref{prop:free_poi_properties} it yields a method for numerical computation of the density of the free unitary Poisson distribution; see Figure~\ref{fig:density_free_poi}.} when $\theta\in \R$. Let first $c\in \bH$. Then, $f_\theta(c) = c$ and, by the Schwarz lemma, $|f'_\theta(c)|<1$ (equality is not possible since $f_\theta$ is not an automorphism of $\bH$). Hence, $c$ is a simple solution of the equation $f_\theta(c) =c$. Let us look what happens to the solution if we perturb $\theta$. By the inverse function theorem, there is an analytic function $c(y)$ defined  on a small disc $\{|y|<\eps\}$ such that $c(0)= c$ and $f_{\theta + y}(c(y)) = c(y)$. If $y \in \bH$, then $\theta + y\in \bH$ and we know from Section~\ref{subsec:principal_branch} and the above discussion that $\zeta_t(\theta+y)$ is the \textit{unique} fixed point of $f_{\theta+y}$ in $\bH$, implying that $c(y) = \zeta_t(\theta +  y)$. Letting $y\to 0$ (in $\bH$), we obtain that $c= c(0)=\lim_{y\to 0, \Im y >0} \zeta_t(\theta+y) = \zeta_t(\theta)$.

Let now $\theta \in \R$ and $c\in \R\cup\{\infty\}$. Let us first exclude the case $c=\infty$. For $c=\infty$, the Denjoy-Wolff theorem states that $\Im f^{\circ n}_\theta (x) \to +\infty$ as $n\to\infty$, but then $f_\theta(f^{\circ n}_\theta (x)) \to  \theta + \ii t$ since $\tan x \to \ii$ as $\Im x\to +\infty$, which is a contradiction. So, let $c\in \R$. By Wolff's theorem~\cite[Theorem~3.1]{burckel}, $f_\theta$ leaves invariant disks that are contained in $\bH$ and tangential to $\R$ at $c$. One consequence is that $c\notin \frac \pi 2 + \pi \Z$ (otherwise $\tan c = +\infty$). Hence, $f_\theta$ is well-defined in a neighborhood of $c$ and $f_\theta (c) = c$.  Another consequence is that $0\leq f'_\theta (c) \leq 1$. Consider first the case when $0 \leq f'_\theta(c)<1$. Then, $c$ is a \emph{simple} solution of the equation $f_\theta(c) = c$. Let us look what happens if we perturb $\theta$. By the inverse function theorem, there is an analytic function $c(y)$ defined on a small disc $\{|y|<\eps\}$ such that $c(0) = c$ and $f_{\theta + y}(c(y)) = c(y)$. Moreover, the  Taylor expansion of $c(y)$ around $y=0$ begins with $c(y) = c - y/(f'_\theta(c) - 1) + o(y)$. Since the coefficient of $y$ in this expansion is $>0$, it follows that $\Im c(\ii \delta)>0$ for sufficiently small $\delta>0$. However, we already know that $f_{\theta + \ii \delta}$ has a \textit{unique} fixed point $\zeta_t(\theta + \ii \delta)$ in $\bH$. We conclude that $c(\ii \delta) = \zeta_t(\theta + \ii \delta)$ provided $\delta>0$ is sufficiently small. It follows that $c= c(0) = \lim_{\delta\downarrow 0} \zeta_t(\theta + \ii \delta) = \zeta_t(\theta)$.

Finally, let us consider the case when $\theta\in \R$, $c\in \R$, $f_\theta(c) = c$ and $f_\theta'(c) = 1$. The latter equation  means that $\cos^2 c = t$. It follows that $t\leq 1$ and $c =\eps \arccos \sqrt t + \pi n$ with some $n\in \Z$ and $\eps\in \{\pm 1\}$. From the equation $f_\theta(c) = \theta + t \tan c = c$ we conclude that $\theta = \eps (\arccos \sqrt t - \sqrt{t(1-t)}) + \pi n$. In the notation of Section~\ref{subsec:principal_branch}, we have $c=\zeta_{n,\eps}$ and $\theta = \theta_{n,\eps}$. Hence, $c= \zeta_t(\theta)$ as we argued in Section~\ref{subsec:principal_branch}.
\end{proof}

\begin{corollary}\label{cor:uniqueness}
Fix $t > 0$. For $\theta\in \bH$,  the equation $\zeta - t \tan \zeta = \theta$ has a \emph{unique} solution $\zeta= \zeta_t(\theta)$ in $\bH$  and infinitely many solutions in $\C\backslash \bH$. Furthermore, the solution $\zeta = \zeta_t(\theta)$ is simple, i.e.\ its multiplicity is $1$.
\end{corollary}
\begin{proof}
The uniqueness of the fixed point of the map $f_\theta:\bH\to \bH$ follows from the Denjoy-Wolff theorem~\cite[Theorem~5.3]{burckel}. On the other hand, the fact that the entire function $\zeta \mapsto \zeta \cos \zeta - t \sin \zeta - \theta \cos \zeta$  has infinitely many complex zeroes follows from the Hadamard factorisation theorem because any entire function of order $1$ with finitely many zeroes has a representation $\eee^{c \zeta} Q(\zeta)$, where $Q$ is a polynomial, which can be ruled out by looking at the asymptotics when $\zeta = \ii \tau$ with $\tau \to \pm \infty$. The simplicity of the solution $\zeta = \zeta_t(\theta)$ follows from the description of critical points and critical values given in Section~\ref{subsec:principal_branch} above. In fact, we have seen that every solution of the system $\zeta - t \tan \zeta = \theta$, $\cos^2 \zeta = t$  satisfies $\zeta \notin \bH$ or $\theta \notin \bH$.
\end{proof}

If $\Im \theta >0$, we can take $x:= \theta$ in Theorem~\ref{theo:iterations1} implying  that the sequence $\theta, \theta + t \tan \theta, \theta + t \tan (\theta + t \tan \theta), \ldots$ converges to $\zeta_t(\theta)$ for all $\theta\in \bH$ and justifying the formula
\begin{equation}\label{eq:zeta_t_as_iteration}
\zeta_t(\theta) = \theta + t \tan (\theta + t \tan (\theta + t\tan(\theta + \ldots))), \qquad \theta \in \bH.
\end{equation}
Note that similar representations of the Lambert $W$-function, for example the following one in terms of the infinite tetration
$$
z^{z^{z^{\ldots}}} = \frac{W(-\log z)}{-\log z}, \qquad \eee^{-\eee} \leq z \leq \eee^{1/\eee},
$$
are well-known and go back to Euler; see~\cite[p.~213]{mezo_book_lambert_function} or~\cite{corless_etal,knoebel}. The next result states that the convergence in~\eqref{eq:zeta_t_as_iteration} holds locally uniformly on $\bH$.
\begin{corollary}\label{cor:iterations2}
Fix $t > 0$. Uniformly on compact subsets of $\bH$, we have
$
f_\theta^{\circ n} (\theta) \to \zeta_t(\theta)
$
as $n\to\infty$.
\end{corollary}
\begin{proof}
The pointwise convergence for all $\theta \in \bH$ follows by taking $x=\theta$ in Theorem~\ref{theo:iterations1}. To prove local uniformity, it suffices to check that the family of functions $\theta \mapsto f_\theta^{\circ n}(\theta)$, $n\in \N$, is precompact in $C(K)$ for every compact set $K\subseteq \bH$.  %Take some $0<\eps<1$ and consider the compact set
%$$
%K:=\{\theta \in \C: |\Re \theta| \leq 1/\eps, \eps\leq \Im \Theta\leq 1/\eps\}\subseteq \bH.
%$$
We have $\Im \theta \geq \eps > 0$ for all $\theta \in K$. It follows that
$$
\Im f_\theta^{\circ (n+1)} (\theta) = \Im \theta + t \Im \tan f_\theta^{\circ n} (\theta) \geq \Im \theta \geq \eps,
$$
for all $n\in \N$ and $\theta\in K$. Since $z\mapsto \tan z$ is bounded on $\{z\in \C: \Im z \geq \eps\}$, it follows that $f_\theta^{\circ (n+2)} (\theta)$ is bounded uniformly in $n\in \N$ and $\theta \in K$. By Montel's theorem, the family of analytic functions $\theta \mapsto f_\theta^{\circ n}(\theta)$, $n\geq 3$, is normal on $\bH$ and the claim follows.
\end{proof}

\begin{remark}
Let $\theta$ be real. Then, the equation $\zeta - t \tan \zeta = \theta$ has infinitely many real solutions; see Figure~\ref{fig:on_the_real_axis}.  In the upper half-plane $\bH$, the number of solutions is $0$ or $1$ (since any such solution remains in $\bH$ after a small perturbation of $\theta$ and we can apply Corollary~\ref{cor:uniqueness}). Similarly, the number of solutions in $-\bH$ is $0$ or $1$. Knowing this and extending the analysis of Section~\ref{subsec:principal_branch}, it should be possible to identify all branches of the multivalued function given by the implicit equation $\zeta - t \tan \zeta = \theta$ as it has been done for the Lambert $W$-function in~\cite{corless_etal}, \cite[pp.~39--70]{mezo_book_lambert_function}.  We refrain from doing this since we need only the principal branch here.
\end{remark}

\subsection{Main properties of the principal branch}
In the next theorem we summarize the main properties of the principal branch $\zeta_t(\theta)$.  It will be convenient to state the result in terms of the function $r_t(z)$ that already appeared in Section~\ref{subsec:principal_branch}.
\begin{theorem}\label{theo:properties_r_t}
Fix some $t>0$. The principal branch $\zeta_t(\theta)$  constructed in Section~\ref{subsec:principal_branch} admits a representation $\zeta_t(\theta) = \theta + \ii t \cdot r_t(\eee^{2 \ii \theta})$, $\theta\in \bH$,  where $r_t(z)$ is a function which is analytic on the unit disk $\bD$ and satisfies
\begin{equation}\label{eq:r_t_funct_eq}
z = \eee^{2tr_t(z)} \frac {1-r_t(z)}{1+r_t(z)},
\qquad
z\in \bD.
\end{equation}
The Taylor expansion of $r_t(z)$ is given by
\begin{equation}\label{eq:r_t_taylor_series}
r_t(z) = 1+ 2\sum_{\ell=1}^\infty \frac{(-1)^\ell}{\ell} \eee^{-2 \ell t } q_{\ell-1} (-4 \ell t ) z^\ell,
\qquad z\in \bD,
\end{equation}
where  $q_m(x)$ is a degree $m$ polynomial in $x$ given by
\begin{equation}\label{eq:q_m_polynomials_def}
q_m (x) =  \sum_{j=0}^{m} \frac{x^j}{j!} \binom {m+1}{j+1},
\qquad m\in \N_0.
\end{equation}
%%%%%%%%%%%%%%%%%%%The coefficients of these polynomials (numerators and denominators) are not in the online encyclopedia
For every $z\in \bD$ we have $\Re r_t(z) >0$. On the interval $(-1,1)$, the function $r_t(z)$ takes real values, is strictly decreasing and satisfies $r_t(0)=1$.   Regarding its behavior near the boundary of $\bD$, we have the following claims.
\begin{itemize}
\item[(i)] If $t>1$, then $r_t(z)$ admits an analytic continuation to the disk $\{|z| \leq \eee^{2\sqrt {t(t-1)} - 2\arccosh \sqrt t}\}$ whose radius is strictly greater than $1$, and the Taylor series~\eqref{eq:r_t_taylor_series} converges locally uniformly there. In particular, it converges uniformly on $\bT$.  We have $\Re r_t(z)>0$ for all $z\in \bT$.
\item[(ii)]  Let $0<t\leq 1$. Then, $r_t(z)$ can be extended to a continuous function on the closed unit disk $\bar \bD$ and the series~\eqref{eq:r_t_taylor_series} converges to $r_t(z)$ uniformly on $\bar \bD$, but not on a larger disk. Also, we have $r_t(1) = 0$.
\item[(iii)] Let $0<t < 1$. Then, for every $z_0\in\bT\backslash\{\eee^{2 \ii  x_t}, \eee^{-2\ii x_t}\}$, where $x_t:=\arccos\sqrt t - \sqrt{t(1-t)}\in (0, \frac \pi 2)$, the function $r_t(z)$ can be extended analytically to a small disk centered at $z_0$.  At $z_0 = \eee^{2 \ii  x_t}$ and  $z_0 = \eee^{-2\ii x_t}$, the function $r_t(z)$ has square-root singularities and Puiseux expansions in powers of $(z-\eee^{\pm 2 \ii  x_t})^{1/2}$.
    For $\theta \in [-x_t, x_t]$ we have $\Re r_t(\eee^{2\ii \theta}) = 0$, while for $\theta \in [-\frac \pi 2, \frac \pi 2] \backslash [-x_t,x_t]$ we have $\Re r_t(\eee^{2\ii \theta}) > 0$. Also, we have  $r_t'(1) = 1/(2t - 2) = - r_t''(1)$.
\item[(iv)] Let $t=1$. Then, for every $z_0\in\bT\backslash\{1\}$ the function $r_1(z)$ can be extended analytically to a small disk centered at $z_0$.  At $z_0 = 1$, the function $r_1(z)$ has a cubic-root singularity and a Puiseux expansion $r_1(z) =  (\frac 32 (1-z))^{1/3}+\ldots$  in powers of $(1-z)^{1/3}$.
%    $$
%   r_t(z) = (3\ii (z-1)/2)^{1/3}/(\ii t)(1+o(1)), \qquad z\to 1.
%    $$
    We have $\Re r_1(z)>0$ for all $z\in \bT\backslash\{1\}$ and $r_1(1) = 0$.
\end{itemize}
\end{theorem}
\begin{proof}
The existence of $r_t(z)$, its analyticity on $\bD$ and~\eqref{eq:r_t_funct_eq} follow from the analysis of Section~\ref{subsec:principal_branch}; see, in particular, \eqref{eq:construction_zeta_large_im_theta}.
%Write $r_t(\theta) := (z_t(\theta) - \theta)/(2t\ii) = - \ii \tan \frac {z_t(\theta)}{2}$. Then, we claim that
%$$
%\eee^{\ii \theta} = \eee^{2t r_t(\theta)} \frac {1 - r_t(\theta)}{1 + r_t(\theta)}.
%$$
%Indeed, using the definition of $r_t(\theta)$ we can write
%$$
%\eee^{2t r_t(\theta)} \frac {1 - r_t(\theta)}{1 + r_t(\theta)}
%=
%eee^{2t r_t(\theta)} \frac{1 + \ii \tan \frac{z_t(\theta)}{2}}{ 1 - \ii \tan \frac{z_t(\theta)}{2}} =
%\eee^{2t r_t(\theta)} \eee^{\ii z_t(\theta)} =  \eee^{(z_t(\theta) - \theta)/\ii} \eee^{\ii z_t(\theta)} = \eee^{\ii \theta},
%$$
%as claimed.
To prove~\eqref{eq:r_t_taylor_series}, let us write $r_t(z) = 1 + p_t(z)$, where $p(z)=p_t(z)$ is a series in $z$ of the form  $p(z) = a_1 z + a_2 z^2 +\ldots$ solving
$$
z = \eee^{2t (1+p)} \frac {-p}{2 + p} = \frac {p}{\phi(p)}
\quad
\text{ with }
\quad
\phi(p):= -(2+p) \eee^{-2t (1+p)}.
$$
%Note that the function $r_t(\theta)$ is analytic on $\bH$ by Theorem~\ref{theo:riemann_surface_z_i}, which implies that $p_t(y)$ is analytic on the open unit disc $\bD$.
The Lagrange inversion formula~\cite[p.~148, Theorem~A]{comtet_book} gives
\begin{align*}
[z^\ell] p(z)
&=
\frac 1 \ell [p^{\ell-1}] (\phi(p)^\ell)
=
\frac {(-1)^\ell} \ell \eee^{-2\ell t}[p^{\ell-1}] ((2+p)^\ell \eee^{-2\ell t p})\\
&=
\frac {(-1)^\ell} \ell \eee^{-2\ell t}  \sum_{j=0}^{\ell-1} \frac{(-2\ell t)^{j}}{j!} \binom{\ell}{j+1} 2^{j+1}
=
2 \frac {(-1)^\ell} \ell \eee^{-2\ell t} q_{\ell-1} (-4\ell t)
\end{align*}
with $q_{\ell-1}$ defined as in~\eqref{eq:q_m_polynomials_def}. This identifies the Taylor series of $p_t$ (which converges on $\bD$) and proves~\eqref{eq:r_t_taylor_series}. %It follows that
%$$
%r_t(\theta) = 1 + p_t(\eee^{\ii \theta}) = 1+ 2\sum_{\ell=1}^\infty \frac{(-1)^\ell}{\ell} \eee^{-2 \ell t } q_{\ell-1} (-4 \ell t % \eee^{\ii \theta \ell},
%$$
%and the series converges for all $\theta\in \bH$.  This completes the proof of~\eqref{eq:r_t_taylor_series}.
Properties (i), (iii), (iv) follow from the analysis of Section~\ref{subsec:principal_branch}. To prove~(ii) observe that the existence of the continuous continuation of $r_t$ to $\bar \bD$ follows from Section~\ref{subsec:principal_branch}. Moreover, the description of non-differentiability points  of $r_t(z)$ on $\bT$ implies that the function $r_t(z)$ is $\alpha$-H\"older continuous with some $\alpha>0$ on $\bT$. Hence, its Fourier series converges to $r_t(z)$ uniformly by the Dini-Lipschitz condition; see~\cite[p.22, Corollary~II]{jackson_book} for a result on the speed of convergence. The maximum principle implies uniform convergence of~\eqref{eq:r_t_taylor_series} on $\bar \bD$.
\end{proof}
\begin{lemma}\label{lem:zeroes_r_t}
For $0< t \leq 1$, the unique zero of the function $r_t(z)$ on $\bT$  is at  $z=1$. For $t>1$, this function has no zeroes on $\bT$.
\end{lemma}
\begin{proof}
For $0<t\leq 1$ we know from  Lemma~\ref{lem:solution_imaginary_axis} that $\zeta_t(0) = 0$, hence $r_t(1)=0$. For $t>1$, we know that $\Im \zeta_t(0) >0$, hence $r_t(1)\neq 0$.  On the other hand, for every $t>0$, if $r_t(\eee^{2\ii \theta}) = 0$ for some $\theta \in (-\frac \pi 2, \frac\pi 2)$, then $\zeta_t(\theta) = \theta$, which implies that $\tan \zeta_t(\theta) = 0$ and hence $\zeta_t(\theta) = \theta = 0$.
\end{proof}
\begin{remark}
Another natural function appearing in connection with $\zeta_t(\theta)$ is the function $v_t: \bD \to \bD$ defined by $v_t(\eee^{2\ii \theta}) = \eee^{2\ii \zeta_t(\theta)}$ for $\theta \in \bH$. For every $z\in \bD$, $v=v_t(z)$ is the unique solution of the equation
$$
v \eee^{2t \frac{1-v}{1+v}} = z
$$
in the unit disk $\bD$. This follows from Corollary~\ref{cor:uniqueness} after mapping $\bH$ to $\bD$ via $\theta\mapsto \eee^{2\ii \theta}$.  The functions $r_t(z)$ and $v_t(z)$ are related by
$$
\frac{1-r_t(z)}{1+r_t(z)} = v_t(z)
\quad
\text{ and }
\quad
\frac{1-v_t(z)}{1+v_t(z)} = r_t(z),
\qquad
z\in \bD.
$$
\end{remark}

\begin{remark}
It is well-known~\cite{corless_etal}, \cite[Section~1.1.3]{mezo_book_lambert_function} that the equation $\eee^x = a + b x$ can be solved in terms of the Lambert $W$-function. On the other hand, the equation $\eee^w = a + \frac b w$ can be reduced to the equation $\zeta - t \tan \zeta = \theta$ studied above. Indeed, it is elementary to check that $\zeta$ solves $\zeta - t \tan \zeta = \theta$ if and only if $w:=2\ii (\zeta - \theta + \ii t)$ solves $\eee^w = a(1 + \frac{4t}{w})$ with $a := -\eee^{-2t - 2\ii \theta}$. Let us also mention that the \emph{$r$-Lambert function} $W_r$, extensively studied in~\cite[Chapter~6]{mezo_book_lambert_function},   is defined as a solution $z = W_r(b)$ of the implicit equation $z \eee^{z} + r z = b$. It follows that $\zeta$ is one of the branches of $W_{-a} (4ta)$. The branch structure of $W_r$, for real $r$, is known~\cite[Chapter~6]{mezo_book_lambert_function}, but note that we need $r= -a = \eee^{-2t - 2\ii \theta}$.
\end{remark}

\section{Properties of the free unitary Poisson distribution}\label{sec:free_unitary_poi_properties}

\subsection{Formulas for the density}
In this section we shall derive some properties of the free unitary Poisson distribution $\Pi_t$ with parameter $t>0$. Recall from Lemma~\ref{lem:free_poi_psi_transf} that the $\psi$-transform of $\Pi_t$ is given by
\begin{equation}\label{eq:free_poi_psi_transf_rep}
\psi_{\Pi_t}(\eee^{\ii \theta})
:=
\sum_{\ell=1}^\infty \eee^{\ii \ell \theta} \int_{\bT} u^\ell \Pi_t(\dint u)
%=
%\frac \ii {2} \cotan \zeta_t(\theta/2) - \frac 12
=
\frac{\ii t}{2\zeta_t(\theta/2) - \theta}-\frac 12
=
\frac 1 {2 r_t(\eee^{\ii \theta})} - \frac 12,
\qquad
\theta \in \bH,
\end{equation}
where $r_t(z)$ is the function whose properties are listed in Theorem~\ref{theo:properties_r_t}. The next result describes the atoms of $\Pi_t$ and provides a formula for the density of the absolutely continuous part; see  Figure~\ref{fig:density_free_poi}.
%defined by~\eqref{eq:r_t_taylor_series} and~\eqref{eq:q_m_polynomials_def}.

\begin{figure}[t]
	\centering
	\includegraphics[width=0.45\columnwidth]{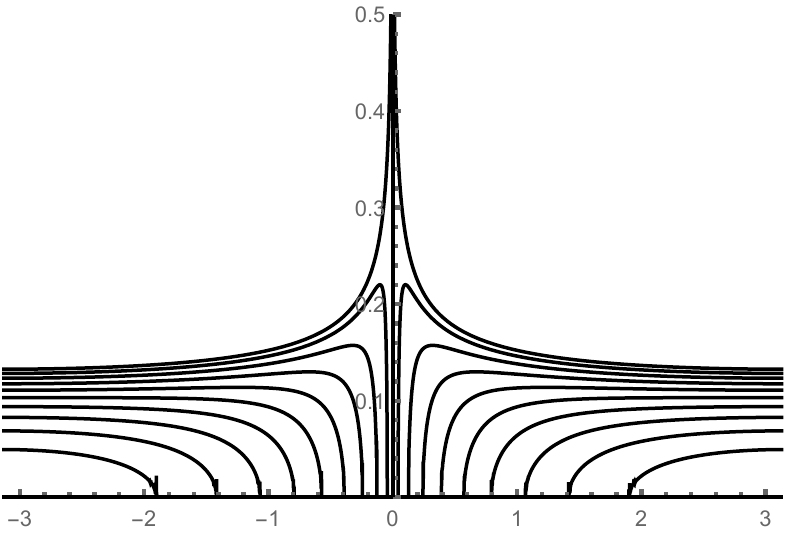}
	\includegraphics[width=0.45\columnwidth]{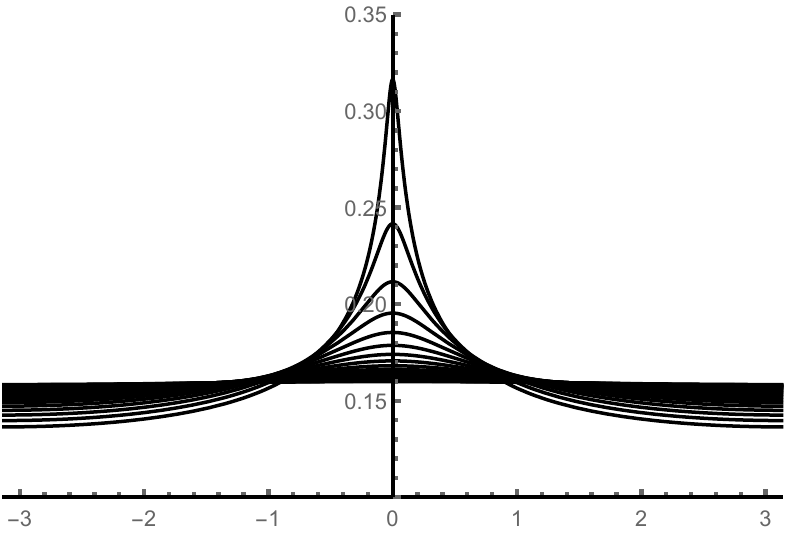}
	\caption{The densities of $\Pi_t$ and $\Pi_t^*$. Left: $t\in \{0.1,0.2,\ldots, 1\}$. The atom at $0$ is not shown. For $t=1$ the density has a singularity at $0$. Right: $t\in \{1.1,1.2,\ldots, 3.0\}$. As $t\to\infty$, the density approaches $1/(2\pi)$. For numerical computations we used an iterative method based on Proposition~\ref{prop:free_poi_properties} combined with Theorem~\ref{theo:iterations1}.}
\label{fig:density_free_poi}
\end{figure}

%\begin{proposition}
%\end{proposition}

%\begin{remark}
%It follows from~\eqref{eq:free_poi_psi_transf} combined with~\eqref{eq:psi_transf_poi_series} that
%%\begin{multline*}
%$$
%z_t(\theta)
%=
%\theta +  2 \ii t + 4\ii t \sum_{\ell=1}^\infty (-1)^\ell  \eee^{-2\ell t}  q_\ell(t) \eee^{\ii \ell \theta}
%$$
%%\end{multline*}
%with $q_1(t) = 1$, $q_2(t) = 1 - 4 t$, $q_3(t) = 1 - 12 t + 24 t^2$, $q_4(t) = 1 - 24 t + 128 t^2 - \frac {512}{3} t^3$, $q_5(t) = %1 -  40 t + 400 t^2 - \frac {4000}{3} t^3 + \frac {4000} 3 t^4,\ldots$.
%\end{remark}

%\begin{theorem}
%Fix $t>0$. Then, for all $\theta \in \bH$ we have
%$$
%z_t(\theta) = \theta + 2\ii t +  4 \ii t  \sum_{\ell=1}^\infty \frac{(-1)^\ell}{\ell} \eee^{-2 \ell t } q_{\ell-1} (-4 \ell t ) %\eee^{\ii \theta \ell},
%$$
%there  $q_m(s)$ is a degree $m$ polynomial in $s$ given by
%\begin{equation}\label{eq:q_m_polynomials_def}
%q_m (s) =  \sum_{j=0}^{m} \frac{s^j}{j!} \binom {m+1}{j+1},
%\qquad m\in \N_0.
%\end{equation}
%\end{theorem}

\begin{proposition}\label{prop:free_poi_properties}
For $t>1$ the distribution $\Pi_t$ is absolutely continuous w.r.t.\ the length measure on $\bT$ and its density is real-analytic, strictly positive and given by
$$
f_{\Pi_t} (z) = \frac {1}{2\pi} \Re \left(\frac 1{r_t(z)}\right)
=
-\frac 1 {2\pi} \Im \cot \zeta_t(\theta/2),
%\qquad
%r_t(z) = 1+ 2\sum_{\ell=1}^\infty \frac{(-1)^\ell}{\ell} \eee^{-2 \ell t } q_{\ell-1} (-4 \ell t ) z^{\ell},
\qquad
z=\eee^{\ii \theta} \in \bT.
$$
For $0<t\leq 1$, the distribution $\Pi_t$ is a sum of an atom at $1$ with weight $1-t$ and an absolutely continuous part $\Pi_t^* := \Pi_t - (1-t)\delta_1$ whose density w.r.t.\ the length measure on $\bT$ is  given by
$$
f_{\Pi_t^*} (z)  =  \frac {1}{2\pi} \Re \left(\frac 1{r_t(z)}\right) =-\frac 1 {2\pi} \Im \cot \zeta_t(\theta/2),
%\qquad
%r_t(z) = 1+ 2\sum_{\ell=1}^\infty \frac{(-1)^\ell}{\ell} \eee^{-2 \ell t } q_{\ell-1} (-4 \ell t ) z^{\ell},
\quad
z=\eee^{\ii \theta}\in \bT\backslash\{1\},
\quad
f_{\Pi_t^*} (1)
=
\begin{cases}
0, & \text{ if } 0<t<1,\\
+\infty, & \text{ if } t=1.
\end{cases}
$$
For $0<t<1$, the function $\theta \mapsto f_{\Pi_t^*}(\eee^{\ii \theta})$ is continuous on $[-\pi, \pi]$, vanishes on the interval $[-2x_t,2x_t]$, is real analytic and strictly positive outside this interval, and has square-root singularities at the end-points of the interval, where
\begin{equation}\label{eq:supp_free_poi}
x_t = \arccos\sqrt t - \sqrt{t(1-t)} = \frac \pi 2 - \int_{0}^t \sqrt{\frac {1-u}u} \,\dint u >0.
\end{equation}
For $t=1$, the function $\theta \mapsto f_{\Pi_1}(\eee^{\ii \theta})$ is real-analytic and strictly positive on $[-\pi,\pi]\backslash\{0\}$, while $\theta=0$ is a singularity with
\begin{equation}\label{eq:free_poi_density_cubic_root}
f_{\Pi_1} (\eee^{\ii \theta}) \sim \frac 1{2\pi}\cdot   \left(\frac{\sqrt 3}{4 |\theta|}\right)^{1/3}
,
\qquad
\text{ as }
\theta\to 0.
\end{equation}
\end{proposition}
All claims follow from the classical properties of the Poisson integral~\cite[Chapter~11]{rudin_book}, \cite[Chapter~1]{garnett_book}, \cite[Chapter~3]{hoffman_book} combined with Theorem~\ref{theo:properties_r_t}. Recall that the Poisson integral of a finite, complex-valued measure $\mu$ on $\bT$ is defined by
$$
F_{\mu}(z)
%=
%\frac {1}{2\pi} \Re \int_\bT \frac{u + z}{u-z} \, \Pi_t(\dd u)
:=
\frac {1}{2\pi} \Re \int_\bT \frac{u + z}{u-z} \, \mu(\dd u)
=
\frac {1}{2\pi} \Re \int_\bT \frac{1+ z\bar u}{1 - z \bar u} \, \mu(\dd u),
\qquad
z\in \bD.
$$
It is well-known~\cite[p.~33]{hoffman_book} that $\mu$ can be recovered from its Poisson integral as the weak limit, as $r\uparrow 1$,  of the measures with density $z\mapsto F(rz)$, $z\in \bT$, w.r.t.\ the length measure on $\bT$.
\begin{lemma}\label{lem:poisson_integral}
If the Poisson integral of some finite complex-valued measure $\mu$ on $\bT$ can be extended to a continuous function on $\bar \bD$, then $\mu$ is absolutely continuous w.r.t.\ the length measure on $\bT$ and its density is given by the restriction of the Poisson integral to $\bT$.
\end{lemma}
\begin{proof}
Let $\mu_1$ be the complex-valued measure with density $F_{\mu}(z)$ on $\bT$. The Poisson integral of $\mu_1$ coincides with $F_{\mu}(z)$ on $\bD$; see~\cite[Theorem~11.9]{rudin_book}. Hence the Poisson integral of $\mu-\mu_1$ vanishes on $\bD$ and it follows that $\mu= \mu_1$.
\end{proof}

\begin{proof}[Proof of Proposition~\ref{prop:free_poi_properties}]
In view of the invariance of $\Pi_t$ w.r.t.\ the complex conjugation, the Poisson integral of $\Pi_t$ takes the form
$$
F_{\Pi_t}(z)
%:=
%\frac {1}{2\pi} \Re \int_\bT \frac{u + z}{u-z} \, \Pi_t(\dd u)
=
\frac {1}{2\pi} \Re \int_\bT \frac{1 + z u}{1-z u} \, \Pi_t(\dd u)
=
\frac 1 {2\pi} \Re (1 + 2\psi_{\Pi_t}(z))
=
\frac {1}{2\pi} \Re \left(\frac 1{r_t(z)}\right),
\qquad
z\in \bD.
$$

For $t>1$ the function $r_t(z)$ is analytic on a disk of radius $>1$ and satisfies $\Re r_t(z) > 0$ for $z\in \bT$ (in particular, it has no zeroes on $\bT$); see Theorem~\ref{theo:properties_r_t} and Lemma~\ref{lem:zeroes_r_t}. Hence, $z\mapsto \frac 1 {2\pi}\Re(1/r_t(z))>0$ is a real-analytic function on $\bT$ and the claim follows from Lemma~\ref{lem:poisson_integral}.

For $0< t \leq 1$, the function $r_t(z)$ admits a continuous extension to $\bar \bD$ and its unique zero on $\bT$  is at  $z=1$; see Lemma~\ref{lem:zeroes_r_t}. Let first $0<t<1$. Then, the function $r_t(z)$ can be  analytically continued to a small disk around $1$ and $r_t'(1) = -r_t''(1) = 1/(2t - 2)$; see Theorem~\ref{theo:properties_r_t}. It follows that
$$
z\mapsto \frac 1 {2\pi} \Re\left(\frac 1{r_t(z)} - (1-t)\frac{1+z}{1-z}\right), %= \frac 1 {2\pi} \Re\left(\frac 1{r_t(z)}\right),
\qquad
z\in \bar \bD \backslash\{1\}
$$
becomes a continuous function on $\bar \bD$ if we define its value at $z=1$ to be $0$.
The Poisson integral of the measure $\Pi_t^* = \Pi_t - (1-t)\delta_1$ is given by
$$
F_{\Pi_t^*} (z) = \frac {1}{2\pi} \Re \left(\frac 1{r_t(z)} - (1-t)\frac{1+z}{1-z}\right),
%=
%\frac {1}{2\pi} \Re \left(\frac 1{r_t(z)}\right),
\qquad
z\in \bD.
$$
By Lemma~\ref{lem:poisson_integral}, the density of the measure $\Pi_t^*$ is given by the restriction of $F_{\Pi_t^*}(z)$ to $\bT$ which can be simplified by dropping the $\frac{1+z}{1-z}$-term because its real part vanishes for $z\in \bT$.  The remaining claims follow from the properties of the function $r_t(z)$ listed in Theorem~\ref{theo:properties_r_t}.

For $t=1$, the function $r_1(z)$ admits a continuous extension to $\bar \bD$ with the unique zero on $\bT$ being at $z=1$, but this time $z=1$ is a branch point where we have $r_1(z) \sim  (\frac 32 (1-z))^{1/3}$ as $z\to 1$, $z\in \bD$. Although the function $z\mapsto \frac 1 {2\pi} \Re(1/r_1(z))$ is not continuous on $\bT$, it belongs to $L^p$ and is the a.e.\ limit of the functions $z\mapsto F_{\Pi_1}(rz)$ as $r\uparrow 1$ whose $L^p$-norm stays bounded as $r\uparrow 1$, for all $1\leq p <3$. A uniform integrability argument implies that the weak limit of  measures with densities $z\mapsto F_{\Pi_1}(rz)$ as $r\uparrow 1$ is the measure with density $z\mapsto \frac 1 {2\pi} \Re(1/r_1(z))$ on $\bT$. Hence, the latter function is the density of $\Pi_1$. The remaining claims are easy to verify. In particular, \eqref{eq:free_poi_density_cubic_root} follows from~\eqref{eq:zeta_cubic_root_asympt} and~\eqref{eq:zeta_cubic_root_series}.
\end{proof}

\subsection{Formulas for the moments}\label{subsec:free_unitary_poi_properties}
Recall from~\eqref{eq:S_transf_def} and~\eqref{eq:S_transf_poi} that the $\psi$-transform of $\Pi_t$ satisfies the following equation:
$$
y = \frac{\psi_{\Pi_t}(y)}{1+ \psi_{\Pi_t}(y)}\exp\left\{\frac{t}{\psi_{\Pi_t}(y) + \frac 12}\right\},
\qquad
y\in \bD.
$$
The formal power series solving this equation and satisfying the condition $\psi_{\Pi_t}(0) = 0$ has the form
\begin{multline}\label{eq:psi_transf_poi_series}
\psi_{\Pi_t}(y)
=
\sum_{\ell = 1}^\infty \eee^{-2\ell t} p_\ell(t)y^\ell
=
\eee^{-2t} y + \eee^{-4t} (1+4t) y^2 + \eee^{-6t}(1 + 4t + 24t^2)y^3
\\
+ \eee^{-8t} \left(1+ 8t + \frac {512}{3}t^3\right)  y^4
+ \eee^{-10 t} \left( 1 + 8 t + 80 t^2 - \frac{800}{3} t^3 + \frac{4000}{3} t^4\right) y^5 +  \ldots
,
\end{multline}
where $p_1(t)=1,p_2(t)=1+4t,\ldots$  are certain polynomials that are related to the moments (or rather Fourier coefficients) of $\Pi_t$ by $\eee^{-2\ell t} p_\ell(t) = \int_{\bT} u^\ell \Pi_t (\dint u)$; see~\eqref{eq:psi_transf_def}.  The next proposition provides an explicit formula for $p_\ell(t)$.

\begin{proposition}\label{prop:free_poisson_moments}
For every $t>0$ the moments of $\Pi_t$ are given by
$$
\int_{\bT} u^\ell \Pi_t (\dint u) = \int_{\bT} u^{-\ell} \Pi_t (\dint u) =
\eee^{ - 2 \ell t} p_\ell(t), \qquad \ell\in \N_0,
$$
where %$p_0(t), p_1(t),\ldots$ are polynomials given by
$$
p_\ell(t) = 1 +  \frac 12 \sum_{\substack{a,b,c\in \N_0, a\neq 0\\ a+b+c = \ell-1}} \binom{\ell-1}{a,b,c} \frac{(2\ell t)^a 2^{a+b+1}(-1)^b}{(a-1)! (a+b)(a+b+1)},
\qquad \ell\in \N_0.
$$
\end{proposition}
%A proof, which is an application of the Lagrange inversion formula, will be given in Section~\ref{subsec:moments_free_poisson}.
\begin{proof}
%\subsection{Moments of the free unitary Poisson law}\label{subsec:moments_free_poisson}
Writing $w = w(y) := \psi_{\Pi_t}(y)$ it follows from~\eqref{eq:S_transf_def} and~\eqref{eq:S_transf_poi} that
$$
y = \frac{w}{1+w} \exp\left\{\frac {t}{w + \frac 12} \right\} = \frac {w}{\phi(w)}
\quad
\text{ with }
\quad
\phi(w) = (1+w) \exp\left\{ - \frac {2t}{2w + 1} \right\}.
$$
The Lagrange inversion formula~\cite[p.~148, Theorem~A]{comtet_book} yields
$$
[y^\ell] w(y)
=
\frac 1\ell [w^{\ell-1}] ((\phi(w))^\ell)
=
\frac 1\ell [w^{\ell-1}] \left((1+w)^\ell \eee^{- \frac {2\ell t}{2w + 1}}\right)
=
\frac {\eee^{-2\ell t}} \ell    [w^{\ell-1}] \left((1+w)^\ell \eee^{- \frac {2 \ell t\cdot 2w}{2w + 1}}\right),
$$
for all $\ell\in \N$. Now we can use the expansion
%related resut: ~\cite[pp.~157--158, Exercise~7]{comtet_book}
\begin{equation}\label{eq:exp_ap_1+p_expansion}
\eee^{\frac{a p}{1+p}} = 1 + \sum_{m=1}^\infty \frac {p^m}{m!} \sum_{j=1}^m (-1)^{j-m} a^j L(m, j)
\end{equation}
which follows from the generating function of the Lah numbers $L(n,k) := \binom{n-1}{k-1}\frac{n!}{k!}$:
$$
\sum_{m=j}^\infty L(m,j) \frac {x^m}{m!} = \frac 1 {j!} \left(\frac x {1-x}\right)^j,
\qquad
j\in \N.
$$
Using the binomial theorem for $(1+w)^\ell$ in combination with~\eqref{eq:exp_ap_1+p_expansion} yields
$$
[y^\ell] w(y) = \frac{\eee^{-2\ell t}}{\ell} \sum_{m=0}^{\ell-1} \binom{\ell}{m+1} \left(\ind_{\{m=0\}} + \frac{2^m \ind_{\{m\neq 0\}}}{m!} \sum_{j=1}^{m} (-1)^{j-m} (2\ell t)^j L(m,j)\right).
$$
After some lengthy but elementary  transformations, we arrive at
$$
[y^\ell] w(y) = \eee^{-2\ell t} +  \frac 1 2 \eee^{-2 \ell t} \sum_{\substack{a,b,c\in \N_0, a\neq 0\\ a+b+c = \ell-1}} \binom{\ell-1}{a,b,c} \frac{(2\ell t)^a 2^{a+b+1}(-1)^b}{(a-1)! (a+b)(a+b+1)},
$$
and the proof is complete.
%\hfill $\Box$
\end{proof}

\begin{proposition}
For $t>1$, the density of $\Pi_t$ w.r.t.\  the length measure on $\bT$ is given by the uniformly convergent series
\begin{equation}\label{eq:fourier_series_density_1}
f_{\Pi_t} (\eee^{\ii \theta}) = \frac {1}{2\pi} + \frac 1 \pi \sum_{\ell=1}^\infty \eee^{-2\ell t} p_\ell(t) \cos (\ell \theta),
\qquad \theta \in [-\pi, \pi].
\end{equation}
For $0<t <1$, the density of $\Pi_t^*= \Pi_t - (1-t)\delta_1$ is given by the uniformly convergent series
\begin{equation}\label{eq:fourier_series_density_2}
f_{\Pi_t}^* (\eee^{\ii \theta}) = \frac {t}{2\pi}  + \frac 1 \pi \sum_{\ell=1}^\infty (\eee^{-2\ell t} p_\ell(t)  - 1 + t) \cos (\ell \theta)
\qquad
\theta \in [-\pi, \pi].
\end{equation}
Finally, for $t=1$, the density of $\Pi_1= \Pi_1^*$ is given by the series~\eqref{eq:fourier_series_density_1} which converges pointwise for $\theta \in [-\pi, \pi]\backslash\{0\}$ and in $L^p$ for $1\leq p  < 3$.
%This series also converges in $L^p$ for $1\leq p  < ?$.
\end{proposition}
\begin{proof}
By Proposition~\ref{prop:free_poisson_moments},  the Fourier coefficients of $f_{\Pi_t}(\eee^{\ii \theta})$, respectively $f_{\Pi_t^*}(\eee^{\ii \theta})$, are given by
$$
\int_{-\pi}^\pi \eee^{\ii \ell \theta}  f_{\Pi_t}(\eee^{\ii \theta}) \dint \theta
=
\eee^{ - 2 |\ell| t} p_{|\ell|}(t),
\qquad
\int_{-\pi}^\pi \eee^{\ii \ell \theta}  f_{\Pi_t^*}(\eee^{\ii \theta}) \dint \theta
=
\eee^{ - 2 |\ell| t} p_{|\ell|}(t) - (1-t),
\qquad
\ell\in \Z.
$$
The series in~\eqref{eq:fourier_series_density_1} and~\eqref{eq:fourier_series_density_2} are the Fourier series of the respective densities and their claimed convergence modes follow from the regularity properties of these densities  established in the proof of Proposition~\ref{prop:free_poi_properties}. In particular, for $t=1$ the density belongs to $L^p$ for $1\leq p <3$, see~\eqref{eq:free_poi_density_cubic_root}, and the $L^p$-convergence of it Fourier series follows from~\cite[p.~59]{katznelson_book}, while the pointwise convergence follows from the local differentiability of the density at any $\theta \in [-\pi, \pi]\backslash\{0\}$.
\end{proof}

\subsection{Formal solution to Steinerberger's PDE}
Let us finally comment on the solutions to Steinerberger's PDE~\eqref{eq:PDE_steinerberger} with periodic initial condition.  Let $(T_{2d}(\theta))_{d\in \N}$ be a sequence of real-rooted trigonometric polynomials such that $\nu \lsem T_{2d}\rsem$ converges weakly to some probability measure $\nu=\nu_0$ on the unit circle $\bT$ and assume that the density of $\nu_0$  w.r.t.\ the length measure on $\bT$ is $\eee^{\ii x}\mapsto u_0(x)$, where $u:\R\to [0,\infty)$ is some $2\pi$-periodic function. Theorem~\ref{theo:main} states that the empirical distribution of zeroes of the $[2 t d]$-th derivative of $T_{2d}$ converges to $\nu_t := \nu_0 \boxtimes \Pi_t$. On the other hand, Kiselev and Tan~\cite{kiselev_tan} proved that (under certain regularity assumptions), the density of $\nu_t$, which we write as $\eee^{\ii x} \mapsto u_t(x)$, solves the PDE~\eqref{eq:PDE_steinerberger}. With the notation $m_\ell := \int_{-\pi}^{\pi} \eee^{\ii \ell x} u_0(x)\dd x$, $\ell\in \Z$, the $\psi$-transform of $\nu_0$ is $\psi_0(z) = \sum_{\ell=1}^\infty m_\ell z^\ell$. Assuming that $m_1\neq 0$ we can invert this series to compute the $S$-transform of $\nu_0$ via~\eqref{eq:S_transf_def}. Then, we can compute the $S$-transform of $\nu_t = \nu_0 \boxtimes \Pi_t$ using~\eqref{eq:S_transform_linearizes_conv} and~\eqref{eq:S_transf_poi}. Using inversion we can compute the $\psi$-transform $\psi_t$ of $\nu_t$. This results in the following formula:
$$
\psi_t(z) = \eee^{-2 t} m_1 z + \eee^{-4 t} (4 t m_1^2 + m_2) z^2 +
\eee^{-6 t} (-8 t m_1^3 + 24 t^2 m_1^3 + 12 t m_1 m_2 + m_3) z^3 + \ldots.
$$
The formal solution to the PDE~\eqref{eq:PDE_steinerberger} is thus
$$
u_t(x) = \frac {1}{2\pi}\Re (1 + 2\psi_{t}(\eee^{-\ii x})) = \frac 1 {2\pi} \left(1 + \eee^{-2t} (m_1 \eee^{-\ii x} + m_{-1} \eee^{\ii x}) + \ldots \right),
\qquad
t\geq 0, \; x\in \R.
$$
It is interesting that only integer powers of $\eee^{-2t}$ appear in the above series. Kiselev and Tan~\cite{kiselev_tan} have shown that as $t\to\infty$, $u_t(x)$ converges to the steady-state solution $1/(2\pi)$ exponentially fast. The above formula yields the precise exponential rate of convergence. 
%The connection to free probability theory, in particular the use of analytic subordination, may lead to further progress on the PDE~\eqref{eq:PDE_steinerberger}.

\section*{Acknowledgements}
The author is grateful to Octavio Arizmendi, Christoph B\"ohm, Jorge Garza-Vargas and Daniel Perales  for useful discussions.
Supported by the German Research Foundation under Germany's Excellence Strategy  EXC 2044 -- 390685587, Mathematics M\"unster: Dynamics - Geometry - Structure.

%\addcontentsline{toc}{section}{References}
%:Referenzen

\bibliography{free_unitary_poisson_bib}
\bibliographystyle{plainnat}

\end{document}